\documentclass[a4paper]{amsart}
\usepackage{a4wide,amsfonts,amssymb,amsmath,amscd,color,tikz,tikz-3dplot,tikz-cd}
\usepackage[english]{babel}
\usepackage{lipsum} 
\usepackage{hyperref}
\usepackage{xcolor}
\usepackage{fancybox}
\usepackage{scalerel}
\hypersetup{
  colorlinks,
  linkcolor={green!50!black},
  citecolor={blue!50!black},
  urlcolor={blue!80!black}
}
\usepackage{cite}
\usepackage{cleveref}
\usepackage{booktabs}
\allowdisplaybreaks
\usepackage{bm}
\usepackage{enumitem}
\usepackage{longtable}
\usepackage[final]{showlabels}

\newcommand{\cred}[1]{{\color{red}#1}}

\newcommand{\defbox}[1]{\noindent\fbox{\parbox{0.975\textwidth}{#1}}}
\newcommand{\thmbox}[1]{\noindent\fbox{\parbox{0.975\textwidth}{#1}}}

\newtheorem{lem}{Lemma}[section]
\newtheorem{defi}[lem]{Definition}
\newtheorem{theo}[lem]{Theorem}
\newtheorem{prop}[lem]{Proposition}
\newtheorem{cor}[lem]{Corollary}
\newtheorem{rem}[lem]{Remark}

\numberwithin{equation}{section}

\newcommand\bra[1]{(#1)}

\makeatletter
\newcommand{\pushright}[1]{\ifmeasuring@#1\else\omit\hfill$\displaystyle#1$\fi\ignorespaces}
\newcommand{\pushleft}[1]{\ifmeasuring@#1\else\omit$\displaystyle#1$\hfill\fi\ignorespaces}
\makeatother

\newcommand{\mr}[1]{\mathring{#1}}

\def\dif{\mathrm{d}}

\def\sfX{\mathbf{X}}
\def\sfY{\mathbf{Y}}
\def\sfW{\mathbf{W}}

\renewcommand{\L}[2]{\sfL^{#1}_{#2}}
\renewcommand{\H}[2]{\sfH^{#1}_{#2}}

\def\X{\sfX}
\def\Y{\sfY}

\def\W{\sfW}
\def\null{\mathcal{N}}
\def\range{\mathcal{R}}
\def\dom{\mathcal{D}}
\def\x{\mathbf{x}}
\def\y{\mathbf{y}}
\def\v{\mathbf{v}}
\def\w{\mathbf{w}}
\def\z{\mathbf{z}}
\def\u{\mathbf{u}}
\def\H{\mathbf{H}}

\def\Wp{\mathbf{W}^{+}}

\def\Ht{\mathbf{H}^{1}_t(\Omega)}
\def\Hn{\mathbf{H}^{1}_n(\Omega)}

\def\Wm{\mathbf{W}^{-}}

\def\Ltwo{L^2(\Omega)}
\def\bLtwo{\mathbf{L}^2(\Omega)}
\def\Hcurl{\mathbf{H}(\mathbf{curl},\Omega)}
\def\Hdiv{\mathbf{H}(\text{div},\Omega)}
\def\Hcurlnot{\mr{\mathbf{H}}(\mathbf{curl},\Omega)}
\def\Hdivnot{\mr{\mathbf{H}}(\text{div},\Omega)}
\def\Honenot{\mr{H}^1(\Omega)}
\def\AT{\A^{\top}}

\def\trace{\mathsf{T}^{t}}
\def\traceA{\mathsf{T}^{t}_{k}}

\def\dualtrace{\mathsf{T}^{n}}

\def\Hone{H^1(\Omega)}
\def\tracespace{\mathcal{T}(\A_k)}

\def\dualtracespace{\mathcal{T}(\AT_k)}

\def\S{\mathsf{S}}
\DeclareMathOperator{\A}{\mathsf{A}}

\def\D{\mathsf{D}}
\def\S{\mathsf{S}}

\def\L{\mathsf{L}}

\DeclareMathOperator{\rA}{\mr{\A}}

\DeclareMathOperator{\id}{\mathsf{Id}}

\newcommand{\norm}[1]{\|#1\|}

\usepackage[most]{tcolorbox}
\definecolor{electricindigo}{rgb}{0.44, 0.0, 1.0}
\definecolor{ultramarineblue}{rgb}{0.25, 0.4, 0.96}
\definecolor{darkspringgreen}{rgb}{0.09, 0.45, 0.27}
\definecolor{deepcarminepink}{rgb}{0.94, 0.19, 0.22}

\usepackage{xparse,l3draw}

\ExplSyntaxOn
\NewDocumentCommand{\whitetriangle}{}
{
  \mathord{ \text { \sukan_triangle:n { stroke } } }
}
\NewDocumentCommand{\filledtriangle}{}
{
  \mathord{ \text { \sukan_triangle:n { fill } } }
}

\cs_new_protected:Nn \sukan_triangle:n
{
  \hspace{0.1em}
  \draw_begin:
  \draw_linewidth:n {0.1ex}
  \draw_join_round:
  \draw_path_moveto:n { 0.0em, 0.0em }
  \draw_path_lineto:n { 0.7em, 0.0em }
  \draw_path_lineto:n { 0.7em, 0.7em }
  \draw_path_close:
  \draw_path_use_clear:n { #1 }
  \draw_end:
  \hspace{0.1em}
}
\ExplSyntaxOff

\newcounter{exampleCounter}
\renewcommand{\theexampleCounter}{\Roman{exampleCounter}}
\newenvironment*{example}[1]
{
  \refstepcounter{exampleCounter}
  \par\medskip%
  \begin{small}
  \noindent\textbf{3D de Rham setting \theexampleCounter: #1.}
  \ignorespaces\noindent
}
{\hfill $\scriptsize\whitetriangle$
\end{small}
\par\medskip%
}%

\newcounter{statementCounter}
\renewcommand{\thestatementCounter}{\Alph{statementCounter}}
\newenvironment*{statement}
{
  \refstepcounter{statementCounter}
  \par\medskip%
  \noindent\textbf{{\color{blue}Assumption \thestatementCounter}.}
  \itshape
  \ignorespaces\noindent
}

\makeatletter
\newsavebox{\@brx}
\newcommand{\llangle}[1][]{\savebox{\@brx}{\(\m@th{#1\langle}\)}%
  \mathopen{\copy\@brx\kern-0.5\wd\@brx\usebox{\@brx}}}
\newcommand{\rrangle}[1][]{\savebox{\@brx}{\(\m@th{#1\rangle}\)}%
  \mathclose{\copy\@brx\kern-0.5\wd\@brx\usebox{\@brx}}}
\makeatother

\title[Traces for Hilbert Complexes]
{Traces for Hilbert Complexes}
\author{Ralf Hiptmair}
\author{Dirk Pauly}
\author{Erick Schulz}
\address{SAM - Seminar for Applied Mathematics, ETH Z\"urich, Switzerland}
\email[Ralf Hiptmair]{ralfh@ethz.ch}
\email[Erick Schulz]{erick.schulz@sam.math.ethz.ch}
\address{Institut f\"ur Analysis, Technische Universit\"at Dresden, Germany}
\email[Dirk Pauly]{dirk.pauly@tu-dresden.de}
\thanks{}

\begin{document}

\def\titlerepude{\sf Traces for Hilbert Complexes}
\def\authorrepude{Ralf Hiptmair \& Dirk Pauly \& Erick Schulz}
\def\daterepdue{\today}
\def\reportudemathyesno{no}
\def\reportudemathnumber{SM-UDE-828}
\def\reportudemathyear{2021}
\def\reportudematheingang{\daterepdue}
\newcommand{\preprintudemath}[5]{
  \thispagestyle{empty}
  \begin{center}\normalsize SCHRIFTENREIHE DER FAKULT\"AT F\"UR MATHEMATIK\end{center}
  \vspace*{5mm}
  \begin{center}#1\end{center}
  \vspace*{5mm}
  \begin{center}by\end{center}
  \vspace*{0mm}
  \begin{center}#2\end{center}
  \vspace*{5mm}
  \normalsize 
  \begin{center}#3\hspace{69mm}#4\end{center}
  \newpage
  \thispagestyle{empty}
  \vspace*{210mm}
  Received: #5
  \newpage
  \addtocounter{page}{-2}
  \normalsize
}
\ifthenelse{\equal{\reportudemathyesno}{yes}}
{\preprintudemath{\titlerepude}{\authorrepude}{\reportudemathnumber}{\reportudemathyear}{\reportudematheingang}}
{}


\begin{abstract}
  We study a new notion of trace operators and trace spaces for abstract Hilbert
  complexes. We introduce trace spaces as quotient spaces/annihilators. We characterize
  the kernels and images of the related trace operators and discuss duality relationships
  between trace spaces. We elaborate that many properties of the classical boundary traces
  associated with the Euclidean de Rham complex on bounded Lipschitz domains are rooted in the
  general structure of Hilbert complexes. We arrive at abstract trace Hilbert complexes
  that can be formulated using quotient spaces/annihilators. We show that, if a Hilbert
  complex admits stable ``regular decompositions'' with compact lifting operators, then
  the associated trace Hilbert complex is Fredholm. Incarnations of abstract concepts and results
  in the concrete case of the de Rham complex in three-dimensional Euclidean space will
  be discussed throughout.
\end{abstract}


\maketitle
\setcounter{tocdepth}{3}


\section{Introduction}
\label{sec: intro}

\subsection{Starting point: the de Rham complex}
In vector-analytic notation, the $L^2$ de Rham complex in a bounded domain
$\Omega\subset\mathbb{R}^3$ reads\footnote{Throughout, we use special arrows to indicate
  properties of mappings: `$\twoheadrightarrow$' for surjectivity, `$\hookrightarrow$' for
  injectivity and `$\dashrightarrow$' for isometry.}
\begin{equation}
  \label{L2 de Rham complex intro}
  \def\arrowlength{5.1ex}
  \def\arrowdistance{.4}
  \begin{tikzcd}[column sep=\arrowlength]
    \mathbb{R}
    \arrow[r, hookrightarrow, shift left=\arrowdistance, "\imath_{\mathbb{R}}"] 
    & 
    {L}^2(\Omega) 
    \arrow[r, rightarrow, shift left=\arrowdistance, "\mathbf{grad}"] 
    & 
    {L}^2(\Omega) 
    \ar[r, rightarrow, shift left=\arrowdistance, "\mathbf{curl}"] 
    & 
    {L}^2(\Omega) 
    \arrow[r, rightarrow, shift left=\arrowdistance, "\text{div}"] 
    & 
    {L}^2(\Omega)
    \arrow[r, twoheadrightarrow, shift left=\arrowdistance, "\pi_{\{0\}}"] 
    & 
    \{0\}.
  \end{tikzcd}
\end{equation}
It involves unbounded first-order differential operators inducing the domain Hilbert complex
\begin{equation}
  \label{L2 domain de Rham complex intro}
  \def\arrowlength{5.1ex}
  \def\arrowdistance{.4}
  \begin{tikzcd}[column sep=\arrowlength]
    \mathbb{R}
    \arrow[r, hookrightarrow, shift left=\arrowdistance, "\imath_{\mathbb{R}}"] 
    & 
    {H}^1(\Omega) 
    \arrow[r, rightarrow, shift left=\arrowdistance, "\mathbf{grad}"] 
    & 
    \Hcurl 
    \ar[r, rightarrow, shift left=\arrowdistance, "\mathbf{curl}"] 
    & 
    \Hdiv
    \arrow[r, rightarrow, shift left=\arrowdistance, "\text{div}"] 
    & 
    {L}^2(\Omega)
    \arrow[r, twoheadrightarrow, shift left=\arrowdistance, "\pi_{\{0\}}"] 
    & 
    \{0\},
  \end{tikzcd}
\end{equation}
where customary notation for Sobolev spaces equipped with graph inner products was
adopted\footnote{For instance, the spaces $\Hone$, $\Hcurl$ and $\Hdiv$ are discussed in
  \cite{girault2012finite}. They are equipped with the obvious graph norms making the
  operators involved in the domain Hilbert complex trivially bounded. In the Euclidean
  setting, we distinguish vector quantities from scalars by using a bold font.}. Taking
the closure of compactly supported functions in these Sobolev spaces and tagging the
resulting closed subspaces with `$\,\circ\,$' on top, we obtain a subcomplex
\begin{equation}
  \label{L2 domain de Rham complex circ intro}
  \def\arrowlength{5.1ex}
  \def\arrowdistance{.4}
  \begin{tikzcd}[column sep=\arrowlength]
    \{0\}
    \arrow[r, hookrightarrow, shift left=\arrowdistance, "\imath"] 
    & 
    \Honenot
    \arrow[r, rightarrow, shift left=\arrowdistance, "\mathbf{grad}"] 
    & 
    \Hcurlnot
    \ar[r, rightarrow, shift left=\arrowdistance, "\mathbf{curl}"] 
    & 
    \Hdivnot
    \arrow[r, rightarrow, shift left=\arrowdistance, "\text{div}"] 
    & 
    L^2(\Omega)
    \arrow[r, twoheadrightarrow, shift left=\arrowdistance, "0\,\,\,"] 
    & 
    \{0\},
  \end{tikzcd}
\end{equation}
giving rise to the following structure:
\begin{equation}\label{structure Hilbert complexes 3D de Rham}
  \def\arrowlength{6ex}
  \def\arrowdistance{.8}
  \begin{tikzcd}[column sep=\arrowlength, row sep=0.05cm]
    \Hone
    \arrow[r, rightarrow, shift left=\arrowdistance, "\mathbf{grad}"] 
    & 
    \Hcurl
    \ar[r, rightarrow, shift left=\arrowdistance, "\mathbf{curl}"] 
    & 
    \Hdiv
    \arrow[r, rightarrow, shift left=\arrowdistance, "\text{div}"] 
    & 
    \Ltwo \\
    \cup 
    & 
    \cup 
    & 
    \cup 
    &
    \cup  \\
    \Honenot
    \arrow[r, rightarrow, shift left=\arrowdistance, "\mathbf{grad}"] 
    & 
    \Hcurlnot
    \ar[r, rightarrow, shift left=\arrowdistance, "\mathbf{curl}"] 
    & 
    \Hdivnot
    \arrow[r, rightarrow, shift left=\arrowdistance, "\text{div}"] 
    & 
    \Ltwo.
  \end{tikzcd}
\end{equation}

\subsection{The de Rham complex and trace operators }
The focus of this work is on trace operators. For the de Rham complex above, those are
usually introduced as linear mappings of functions in $\Omega$ to functions on
$\Gamma=\partial\Omega$. The classical traces are obtained by extending the restriction
operators\footnote{We denote by $\mathbf{n}\in\mathbf{L}^\infty(\Gamma)$ the exterior unit
  normal vector-field on the boundary $\Gamma$.}
\begin{subequations}
  \label{eq:ctr}
  \begin{align}
    \gamma u&:=u\big\vert_{\Gamma}\label{eq: def Dirichlet trace} &&\text{(pointwise trace)},\\
    \gamma_{t}\u&:=\mathbf{n}\times(\u\big\vert_{\Gamma}\times\mathbf{n}) \label{eq: def t
      trace}&&\text{(pointwise tangential component trace)},\\
    \gamma_n\u&:=\u\big\vert_{\Gamma}\cdot\mathbf{n} \label{eq: def neumann
      trace}&&\text{(pointwise normal component trace)},
  \end{align}
\end{subequations}
to continuous and surjective mappings from the Sobolev spaces involved in the domain de
Rham complex to so-called trace spaces whose characterization is the main assertion of the
standard trace theorems for a Lipschitz domain $\Omega$:
\begin{subequations}
  \label{eq:euctr}
  \begin{align}
    \gamma &: H^1\bra{\Omega}\rightarrow H^{1/2}\bra{\Gamma} &&\text{\cite[Thm. 4.2.1]{Hsiao2008}}\label{eq: ext of Dirichlet trace},\\
    \gamma_{t}&:\mathbf{H}\bra{\mathbf{curl},\Omega} \rightarrow \mathbf{H}^{-1/2}(\text{curl}_\Gamma,\Gamma) &&\text{\cite[Thm. 4.1]{buffa2002traces}}\label{eq: ext of t trace},\\
    \gamma_n &: \mathbf{H}(\text{div},\Omega)\rightarrow H^{-1/2}\bra{\Gamma}&&\text{\cite[Thm. 2.5, Cor. 2.8]{girault2012finite}}. \label{eq: ext normal trace}   
  \end{align}
\end{subequations}

The classical trace spaces can be defined based on the vector-valued rotated surface
gradient $\mathbf{curl}_{\Gamma}$ and the scalar-valued surface rotation
$\text{curl}_{\Gamma}$ as
\begin{subequations}
  \label{eq:17}
  \begin{align}
    H^{1/2}(\Gamma):=\left\{\,\bm{\phi}\in H^{-1/2}(\Gamma)\,\,\vert\,\,\mathbf{curl}_{\Gamma}\,\bm{\phi}\in \mathbf{H}_t^{-1/2}(\Gamma)\,\right\},\label{def:minushalf trace space}\\
    \mathbf{H}^{-1/2}(\text{curl}_\Gamma,\Gamma):=\left\{\,\bm{\phi}\in \mathbf{H}_t^{-1/2}(\Gamma)\,\,\vert\,\,\text{curl}_{\Gamma}\,\bm{\phi}\in H^{-1/2}(\Gamma)\,\right\},\label{def:minushalf curl trace space}
  \end{align}
\end{subequations}
where $\mathbf{H}_t^{-1/2}(\Gamma)$ is defined as the dual of the range of the tangential
trace applied to $\mathbf{H}^1(\Omega)$. The mathematical theory of the pointwise trace
$\gamma$ is well established, cf. \cite[Chap. 3]{mclean2000strongly}. That for the normal
component trace $\gamma_{n}$ is carefully developed in \cite[Chap. 1]{girault2012finite}.
Regarding the tangential trace $\gamma_{t}$ in \eqref{eq: ext of t trace} and the trace space
\eqref{def:minushalf curl trace space}, we recommend the comprehensive and profound analysis of
\cite{buffa2002traces}, based on the earlier works \cite{BUC99,BUC99a,ALV96}.

These important results were generalized to arbitrary dimensions by Weck in
\cite{MR2085953} using the framework of differential forms, where pullback by the
boundary's inclusion map provides a unified description and generalization of the traces
\eqref{eq:euctr}. A similar characterization of the range of the boundary restriction operator for
Lipschitz subdomains of compact manifolds is given in \cite{MR2463962}, where a boundary
de Rham complex involving surface operators is also studied.

One may wonder whether the structures shining through in \eqref{def:minushalf trace space}
and \eqref{def:minushalf curl trace space} hint at a more general pattern governing the
structure of trace spaces. Thus, in this article, we are going to elaborate this structure
in the abstract framework of Hilbert complexes, of which the de Rham complex is the
best-known representative. Since there is no notion of ``boundary'' in that abstract
framework, we have to detach the concept of a trace space from the idea of a function
space on a boundary. This can be accomplished by adopting a quotient-space view of
traces.

Let us sketch this idea for the Euclidean de Rham complex. Since the kernels of the classical trace
operators \eqref{eq: ext of Dirichlet trace}-\eqref{eq: ext normal trace} are\footnote{We
  write $\mathcal{N}(\mathsf{T})$ and $\mathcal{R}(\mathsf{T})$ for the kernel/nullspace
  and range/image space, respectively, of a linear operator $\mathsf{T}$.}
\begin{subequations}
  \begin{align}
    \null( \gamma)&=\Honenot:=\overline{C_0^{\infty}(\Omega)}^{H^1(\Omega)} &&\text{\cite[Thm. 3.40]{mclean2000strongly}},\label{eq: ker Dirichlet trace}\\
    \null(\gamma_t)&=\Hcurlnot:=\overline{C_0^{\infty}(\Omega)^3}^{\mathbf{H}\bra{\mathbf{curl},\Omega}} && \text{\cite[Thm. 3.33]{monk2003finite}},\label{eq: ker t trace}\\
    \null(\gamma_n)&=\Hdivnot:=\overline{C_0^{\infty}(\Omega)^3}^{\mathbf{H}\bra{\text{div},\Omega}} && \text{\cite[Thm. 3.25]{monk2003finite}},\label{eq: ket n trace}
  \end{align}
\end{subequations}
we immediately conclude that these trace operators induce isomorphisms between the classical trace spaces and the quotient spaces:
\begin{subequations}
  \label{eq:drqs}
  \begin{align}
    \Hone/\Honenot&\cong H^{1/2}(\Gamma),\\
    \Hcurl/\Hcurlnot&\cong \mathbf{H}^{-1/2}(\text{curl}_\Gamma,\Gamma),\\
    \Hdiv/\Hdivnot &\cong H^{-1/2}(\Gamma).
  \end{align}
\end{subequations}
This paves the way for an alternative characterization of trace spaces independent of the
notion of ``function space on $\Gamma$''.  We remark that the quotient space approach to
the definition of trace spaces has also proved successful for the de Rham complex in order
to define traces on sets more complicated than boundaries of Lipschitz domains
\cite{MR3101780,MR3439201}.

Classical theory of trace spaces for $\Hone$, $\Hcurl$ and $\Hdiv$ also addresses duality
between trace spaces:
\begin{itemize}
\item The $L^2(\Gamma)$ inner product induces a duality between $H^{1/2}(\Gamma)$ and $H^{-1/2}(\Gamma)$; cf. \cite[Chap. 4.2]{Hsiao2008} and \cite[Chap. 3]{mclean2000strongly}.
\item The skew-symmetric pairing\footnote{We denote by $\sigma$ the surface measure on the boundary.}
  \begin{equation}\label{skew-symmetric pairing}
    \langle\u,\v\rangle_{\times}:=\int_{\Gamma}(\u\times\mathbf{n})\cdot\v \,\dif\sigma
  \end{equation}
  can be extended from $\mathbf{L}^2(\Gamma)\times\mathbf{L}^2(\Gamma)$ to
  $\mathbf{H}^{-1/2}(\text{curl}_\Gamma,\Gamma)\times\mathbf{H}^{-1/2}(\text{curl}_\Gamma,\Gamma)$,
  allowing the identification of $\mathbf{H}^{-1/2}(\text{curl}_\Gamma,\Gamma)$ with its
  own dual space; cf. \cite{buffa2002traces, MR2032868,monk2003finite}.
\end{itemize}
The possibility to put trace spaces for the 3D de Rham complex into duality seems to follow general rules:
\begin{equation}
  \label{eq:drdual}
  \def\arrowlength{17ex}
  \begin{tikzcd}[column sep=\arrowlength]
    H^1(\Omega)
    \arrow[r, rightarrow,  "\mathbf{grad}"] \arrow[d, rightarrow,  "\gamma"] 
    & 
    \Hcurl
    \ar[r, rightarrow,"\mathbf{curl}"] \arrow[d, rightarrow,  "\gamma_t"] 
    & 
    \Hdiv\arrow[d, rightarrow,  "\gamma_n"] \\
    \H^{1/2}(\Gamma)\arrow[rr, dash, bend right=19, "L^2\text{-duality}" description] 
    &
    \mathbf{H}^{-1/2}(\text{curl}_\Gamma,\Gamma)\ar[dash,loop right, distance=6em, "L^2\text{-self duality}" description]
    &
    H^{-1/2}(\Gamma)
  \end{tikzcd}
\end{equation}

\subsection{Goals, outline, and main results}
There are obvious parallels in the definitions of the different trace spaces and their
duality relations. One may wonder if this kind of resemblance between the trace spaces
arise only for the de Rham complex or whether it is already manifest in a more
basic/general setting, of which the de Rham complex is just a prominent specimen. That
setting is the framework of \emph{Hilbert complexes}\footnote{For the functional analytic
  foundations, we refer to parts of the FA-ToolBox from \cite[Sec. 2]{PS2021a}, which is a
  compilation of useful functional analysis results that grew from its use in previous
  works, cf.  \cite[Sec. 4.1]{P2019b}, \cite[Sec. 2]{P2020a}, \cite[Sec. 2.1]{PZ2016a},
  \cite[Sec. 2.1]{PZ2020a}, \cite[2.2]{PZ2020b}, \cite[Sec. 2]{PS2021a} and
  \cite[App. 3]{P2017a}. We find the introduction in \cite[Chap. 4]{MR3908678} to be an
  accessible resource for readers unacquainted with Hilbert complexes, because it reviews
  in detail the material more concisely presented in \cite[Sec. 3]{MR2594630},
  cf. \cite[Sec. 2]{MR2269741} and \cite{MR1174159}.}, first introduced in
\cite{MR1174159}. Therefore, the guiding question behind this work is:
\begin{quote}
  \emph{To what extent can results about traces for the de Rham domain complex be
    transferred to abstract Hilbert complexes?}
\end{quote}
Of course, abstract Hilbert complexes know neither domains nor boundaries. Therefore, as
already mentioned above, we cannot expect to arrive at a characterization of trace spaces
as function spaces on a boundary. Yet, a theory based on the quotient space view of trace
spaces is feasible. Its development will be pursued in \Cref{sec: Basic traces}. There, we
first propose trace operators induced by ``generalized integration by parts formulas'' and
mapping into dual spaces, and then generalize \eqref{eq:drqs} to a quotient-space
understanding of trace spaces.

Next, in \Cref{sec: Duality}, we shed light on duality relationships between trace spaces
and find that the observation made in \eqref{eq:drdual} is a generic pattern; see
\Cref{def of B}. This even holds in a setting simpler than Hilbert complexes. ``Minimal
Hilbert complexes'' will only enter the stage in \Cref{sec: Operators on trace spaces} in
order to define so-called ``surface operators'', which are abstract counterparts of the
classical surface differential operators such as $\mathbf{grad}_{\Gamma}$ and
$\textbf{curl}_{\Gamma}$. The full structure of Hilbert complexes is exploited starting
from \Cref{Characterization by regular subspaces}. Augmenting it by assumptions about the
existence of so-called stable regular decompositions (Assumptions \ref{Assumption continuous and
  dense embeddings AT} and \ref{dual regular decomposition}), we obtain
characterizations of traces spaces, in \Cref{thm:
	characterization of range Tt} and \Cref{thm: characterization of range Tn}, which reveal that the definitions \eqref{def:minushalf
  trace space} and \eqref{def:minushalf curl trace space} of classical trace spaces
reflect a more general pattern. This paves the way for the key insight expressed in \Cref{thm: trace
  Hilbert complexes definition} that trace spaces and surface operators are the building
blocks of what we call a trace Hilbert complex, a full-fledged Hilbert complex of
unbounded, densely defined, and closed operators. 

Parallel to its development, we will apply our new abstract theory to the de Rham complex in
three-dimensional Euclidean space. We hope that this will motivate some of the assumptions
made on the abstract spaces. The discussion will take the form of an ongoing
specialization of the definitions and results, set apart from the main line of reasoning.

\begin{example}{Traces and integration by parts}
  The key trace operators and trace spaces associated with the Euclidean de Rham complex
  in three space dimensions have already been introduced in \eqref{eq:ctr} and
  \eqref{eq:euctr}.  We just want to add the well-known fact that the trace operators
  \eqref{eq: ext of Dirichlet trace}-\eqref{eq: ext normal trace} have a close link with
  Green's formulas
  \begin{subequations}
    \begin{align}
      \langle\gamma u, \gamma_n\v\rangle_{\Gamma}&=\int_{\Omega} \mathbf{grad} u\cdot v + u\,\text{div}(\v) \, \dif \mathbf{x} &&\forall u\in H^1(\Omega),\forall\v\in\Hdiv,\label{IBP div}\\
      \langle\gamma_{t}\u, \gamma_{t}\v\rangle_{\times}&=\int_{\Omega}\mathbf{curl\,}\u\cdot\v -\u\cdot\mathbf{curl\,}\v\dif \mathbf{x} &&\forall\u,\v\in\Hcurl. \label{IBP curl}
    \end{align}
  \end{subequations}
  On the left, we denoted the duality pairing between $H^{1/2}(\Gamma)$ and $H^{-1/2}(\Gamma)$ by $\langle\cdot,\cdot\rangle_{\Gamma}$, but wrote $\langle\cdot,\cdot\rangle_{\times}$ for the skew-symmetric self-duality pairing on $\mathbf{H}^{-1/2}(\text{curl}_\Gamma,\Gamma)$, cf. \cite[Lem. 5.6]{buffa2002traces}.
\end{example}

Finally, we stress that we could have demonstrated the specialization of our results also
in the setting of general exterior calculus, but refrained from it in the interest of
readability.
\bigskip\bigskip

\newlength{\symlistskip}
\setlength{\symlistskip}{0ex}
\begin{center}
  \textbf{List of symbols}
  \bigskip\bigskip
  
  \begin{tabular}{p{3em}@{$\hat{=}$ }p{9cm}p{5cm}}
    $\A_k$                         & closed densely defined unbounded operators
    &\Cref{sec: Hilbert complexes}, \eqref{Hilbert complex Ak}
    \rule{0pt}{\symlistskip}\\
    $\A^*_k$                         & Hilbert space adjoint of $\A_k$         &\Cref{sec: Hilbert complexes}, \eqref{Hilbert dual complex Ak}
    \rule{0pt}{\symlistskip}\\
    $\rA_k$                         & closed densely defined unbounded operator $\rA_k\subset\A_k$          &\Cref{sec: Baisc setting}, \eqref{Hilbert complex rAk}\rule{0pt}{\symlistskip}\\
    $\AT_k$                         & Hilbert space adjoint of $\rA_k$         &\Cref{sec: Baisc setting}, \eqref{Hilbert dual complex ATk}\rule{0pt}{\symlistskip}\\
    $\mathsf{R}_{\dom(\AT_k)}$ &Riesz isomorphism $\dom(\AT_k)\rightarrow\dom(\AT_k)'$ & \Cref{sec: Riesz representatives}, \eqref{lem: riesz rep in domains}\rule{0pt}{\symlistskip}\\
    $\trace_k$                     & primal Hilbert trace $\dom(\A_k)\rightarrow\dom(\AT_k)'$                                       &\Cref{sec: Hilbert traces}, \eqref{eq: def of trace}\rule{0pt}{\symlistskip}\\
    $\dualtrace_k$                 & dual Hilbert trace $\dom(\AT_k)\rightarrow\dom(\A_k)'$                                         &\Cref{sec: Dual traces}, \eqref{eq: def of n trace} \rule{0pt}{\symlistskip}\\
    $\tracespace$                 & quotient space $\dom(\A_k)/\dom(\rA_k)$ &\Cref{sec: Trace spaces}, \eqref{eq: def of trace space}\rule{0pt}{\symlistskip}\\
    $\dualtracespace$             & quotient space $\dom(\AT_k)/\dom(\A^*_k)$ &\Cref{sec: Dual traces}, \eqref{eq: def of dual trace space}\rule{0pt}{\symlistskip}\\
    $\mathsf{I}^t_k$ &isometric isomorphism $\dom(\A_k)\rightarrow\range(\trace_k)$ & \Cref{sec: Trace spaces}, \eqref{eq: def of It iso}\rule{0pt}{\symlistskip}\\
    $\mathsf{I}^n_k$ &isometric isomorphism $\dom(\AT_k)\rightarrow\range(\dualtrace_k)$ & \Cref{sec: Dual traces}, \eqref{eq: def of In iso}\rule{0pt}{\symlistskip}\\
    $\llangle\cdot,\cdot\rrangle_k$  & duality pairing                                                 &\Cref{sec: Dual system}, \eqref{eq: def of duality pairing}\rule{0pt}{\symlistskip}\\
    $\mathsf{K}_k$                 & isometric isomorphism induced by $\langle\cdot,\cdot\rangle_k$ &\Cref{sec: Dual system}, \eqref{eq: definition of K pairing}\rule{0pt}{\symlistskip}\\
    $\mathsf{P}^t_k$                 & orthogonal projection $\dom(\A_k)\rightarrow\dom(\rA)^{\perp}$ &\Cref{sec: Hilbert traces}, \eqref{eq: commutative diag Pt and Gt}\rule{0pt}{\symlistskip}\\
    $\mathsf{P}^n_k$                 & orthogonal projection $\dom(\AT_k)\rightarrow\dom(\A^*_k)^{\perp}$ &\Cref{sec: Dual traces}, \eqref{commutative diagram Pn and Gn}\rule{0pt}{\symlistskip}\\
    $\bm{\pi}^t_k$                 & canonical quotient map $\dom(\A_k)\rightarrow\tracespace$ &\Cref{sec: Hilbert traces}, \eqref{eq: commutative diag Pt and Gt}\rule{0pt}{\symlistskip}\\
    $\bm{\pi}^n_k$                 & canonical quotient map $\dom(\AT_k)\rightarrow\dualtracespace$ &\Cref{sec: Hilbert traces}, \eqref{commutative diagram Pn and Gn}\rule{0pt}{\symlistskip}\\
    $\Wp_k$             &dense inclusion $\Wp_k\hookrightarrow\dom(\A_k)$ and/or $\Wp_k\hookrightarrow\dom(\AT_{k-1})$ &\Cref{sec: Bounded regular decompositions}, \eqref{dense embedding Yp and ZP}\rule{0pt}{\symlistskip}\\
    $\Wm_k$        & dual space $(\Wp_k)'$ &\Cref{sec: Bounded regular decompositions}, \eqref{eq: def of Wminus}\rule{0pt}{\symlistskip}\\
    $\mathring{\W}^{n,+}_k$   & intersection space $\dom(\A^*_{k-1})\cap\Wp_k=\null(\dualtrace_{k-1})\cap\Wp_k$ &\Cref{sec: Characterization of trace spaces}, \eqref{def of not t space} \rule{0pt}{\symlistskip}\\
    $\mathring{\W}^{t,+}_k$   & intersection space $\dom(\rA_k)\cap\Wp_k=\null(\trace_k)\cap\Wp_k$ &\Cref{sec: Characterization of trace spaces}, \eqref{def of not t space}\rule{0pt}{\symlistskip} \\
    $\mathbf{T}^{n,+}_{k}$ & quotient space $\Wp_k/\mathring{\W}^{n,+}$ &\Cref{Characterization of trace spaces based on quotient spaces}, \eqref{eq: def of Tn}\rule{0pt}{\symlistskip}\\
    $\mathbf{T}^{t,+}_{k}$ & quotient space $\Wp_k/\mathring{\W}^{t,+}$ &\Cref{Characterization of trace spaces based on quotient spaces}, \eqref{eq: def of Tt}\rule{0pt}{\symlistskip}\\
    $\mathbf{T}^{n,-}_{k}$ & dual space $(\mathbf{T}^{n,+}_{k})'$ &\Cref{Characterization of trace spaces based on quotient spaces}, \eqref{eq: def of Tn}\rule{0pt}{\symlistskip}\\
    $\mathbf{T}^{t,-}_{k}$ & dual space $(\mathbf{T}^{t,+}_{k})'$ &\Cref{Characterization of trace spaces based on quotient spaces}, \eqref{eq: def of Tt}\rule{0pt}{\symlistskip}\\
    $\mathsf{D}^t_k$                 & surface operator $(\AT_{k+1})':\dom(\AT_{k})'\rightarrow \dom(\AT_{k+1})'$ &\Cref{sec: Surface operators in domains}, \eqref{def: surface operator Dt}\rule{0pt}{\symlistskip}\\
    $\mathsf{D}^n_k$                 & surface operator $\A_{k-1}':\dom(\A_{k})'\rightarrow \dom(\A_{k-1})'$ &\Cref{sec: Surface operators in domains}, \eqref{def: surface operator Dn}\rule{0pt}{\symlistskip}\\
    $\mathsf{S}^t_k$                 & surface operator $\A_k:\mathcal{T}(\A_k)\rightarrow \mathcal{T}(\A_{k+1})$ &\Cref{sec:Surface operators in quotient spaces}, \eqref{eq: def of S trace operators}\rule{0pt}{\symlistskip}\\
    $\mathsf{S}^t_k$                 & surface operator $\AT_k:\mathcal{T}(\AT_k)\rightarrow \mathcal{T}(\AT_{k-1})$ &\Cref{sec:Surface operators in quotient spaces}, \eqref{eq: def of S trace operators}\rule{0pt}{\symlistskip}\\
    $\hat{\mathsf{S}}^t_k$                 & surface operator $\A_k:\mathbf{T}^{t,+}_{k+1}\rightarrow \mathcal{T}(\A_{k+1})$ &\Cref{Characterization of trace spaces based on quotient spaces} \eqref{eq: def of the S operators in lemma}\rule{0pt}{\symlistskip}\\
    $\hat{\mathsf{S}}^{n}_k$                 & surface operator $\AT_k:\mathbf{T}^{n,+}_{k+1}\rightarrow \mathcal{T}(\AT_{k-1})$ &\Cref{Characterization of trace spaces based on quotient spaces}, \eqref{eq: def of the S operators in lemma}\rule{0pt}{\symlistskip}\\
    $\hat{\mathsf{D}}^t_{k}$ & surface operator $(	\hat{\mathsf{S}}^n_{k+1})':\mathcal{T}(\AT_k)'\rightarrow \mathbf{T}^{n,-}_{k+2}$ &\Cref{Characterization of trace spaces based on quotient spaces}, \eqref{extension of D hat}\rule{0pt}{\symlistskip}\\
    $\hat{\mathsf{D}}^n_{k}$ & surface operator $(
    \hat{\mathsf{S}}^t_{k})':\mathcal{T}(\A_{k+1})'\rightarrow \mathbf{T}^{t,-}_k$
    &\Cref{Characterization of trace spaces based on quotient
      spaces}, \eqref{extension of D hat}\rule{0pt}{\symlistskip}
  \end{tabular}
\end{center}

\newpage
\section{Hilbert Complexes}
\label{sec:hc}

\subsection{Operators on Hilbert spaces}\label{sec: Operators on Hilbert spaces}
In this article, both \emph{bounded} and \emph{unbounded} linear operators take center
stage\footnote{Standard references concerning bounded and unbounded linear operators are
  \cite[Chap. 3]{MR1335452} and \cite[Chap. 7]{MR0350358}. We also particularly recommend
  \cite[Chap. 3]{MR3908678}, \cite[Chap. 1-6]{MR2759829} and
  \cite[Chap. 6-8]{MR751959}.}. We distinguish them using the following notation. Let $\X$
and $\Y$ be two Hilbert spaces equipped with the inner products $(\cdot,\cdot)_{\X}$ and
$(\cdot,\cdot)_{\Y}$, respectively. We will consistently write
$ \A:\dom(\A)\subset\X\to\Y $ to indicate that $\A$ is regarded as an \emph{unbounded}
linear operator from $\X$ to $\Y$ with domain $\dom(\A)$, whereas we mean by
$ \A:\X\to\Y $ that $\A$ is viewed as a \emph{bounded} operator from $\X$ to $\Y$ defined
on the whole space $\X$.

Recall that the difference between $\A:\dom(\A)\subset\X\to\Y$ and $\A:\dom(\A)\to\Y$
comes from whether the topology of the subspace $\dom(\A)\subset\X$ is given by the norm
of $\X$ or the graph norm induced by the inner product
$ (\x_1,\x_2)_{\dom(\A)}:= (\x_1,\x_2)_{\X}+(\A\x_1,\A\x_2)_{\Y} $
$\forall\x_1,\x_2\in \dom(\A)$.

An unbounded operator $\A:\dom(\A)\subset\X\to\Y$ is said to be \emph{closed} if and only if its domain $\dom(\A)$ is a Hilbert space when endowed with the graph norm, cf. \cite[Prop. 3.1]{MR3908678}. It is \emph{densely defined} if $\dom(A)$ is a dense subset of $\X$. The kernel and range of $\A$, whether it is bounded or not, will be denoted $\null(\A)$ and $\range(\A)$, respectively. 

Topological dual spaces will be tagged with prime, e.g. $\X'$. We use angle brackets for
duality pairings, e.g. $\langle \bm{\phi},\x\rangle_{\X'}$, $\bm{\phi}\in\X'$,
$\x\in\X$. Accordingly, the operator dual to a \emph{bounded} linear operator
$\A:\X\rightarrow\Y$ is a bounded operator $\A':\Y'\rightarrow \X'$.

The Hilbert space adjoint of $\A:\dom(\A)\subset\X\to\Y$ is written $\A^*: \dom(\A^*)\subset\Y\rightarrow \X$.
Recall that it is the unbounded linear operator satisfying
\begin{align}\label{def: hilbert space adjoint}
  (\A^*\y,\x)_{\X} = (\y,\A\x)_{\Y} \quad \forall\y\in \mathcal{D}(\A^*),\forall\x\in \dom(\A),
\end{align}
whose domain $\dom(\A^*)$ consists of all $\y\in\Y$ for which the linear functional
$\dom(\A)\rightarrow\mathbb{R}$ defined by $\x\mapsto (\y,\A\x)_{\Y}$ is continuous in the
$\X$ norm, i.e. for every $\y\in \dom(\A^*)$, $\exists C_{\y} >0$ such that
$\vert(\y,\A\x)_{\Y}\vert\leq C_{\y}\norm{\x}_{\X}$, $\forall\x\in \dom(\A)$.  If $\A$ is
closed and densely defined, then $\A^*$ is also closed and densely defined
\cite[Prop. 3.3]{MR3908678}---in which case $\A^{**}=\A$.

We write $\rA\subset\A$ and say that an unbounded linear operator $\A:\dom(\A)\subset\X\rightarrow\Y$ is an extension of another unbounded linear operator $\rA:\dom(\rA)\subset\X\rightarrow\Y$ when $\dom(\rA)\subset\dom(\A)$ and $\A\x_{\circ}=\rA\x_{\circ}$ for all $\x_{\circ}\in\dom(\rA)$.

\begin{example}{Differential operators}\label{spe:Differential operators}
  We refer to \cite[Chap. 3]{MR3908678} for the following mappings properties. The linear differential operators
  \begin{subequations}
    \begin{align}
      \mathbf{grad}:&\,\Hone\subset\Ltwo\rightarrow\bLtwo \label{nabla},\\
      \mathbf{curl}:&\,\Hcurl\subset\bLtwo\rightarrow\bLtwo,\\
      \text{div}:&\,\Hdiv\subset\bLtwo\rightarrow\Ltwo\label{div},
    \end{align}
  \end{subequations}
  are densely defined and closed unbounded linear operators. They are extensions of
  \begin{subequations}
    \begin{align}
      \mr{\mathbf{grad}}:&\,\Honenot\subset\Ltwo\rightarrow\bLtwo \label{rnabla},\\
      \mr{\mathbf{curl}}:&\,\Hcurlnot\subset\bLtwo\rightarrow\bLtwo,\\
      \mr{\text{div}}:&\,\Hdivnot\subset\bLtwo\rightarrow\Ltwo\label{rdiv}.
    \end{align}
  \end{subequations}

  The $L^2$ Hilbert space adjoints of \eqref{nabla}-\eqref{div} are
  \begin{subequations}
    \begin{align}
      \mathbf{grad}^*=-\mr{\text{div}}:&\,\Hdivnot\subset\bLtwo\rightarrow\Ltwo,\label{def grad star}\\
      \mathbf{curl}^*=\mr{\mathbf{curl}}:&\,\Hcurlnot\subset\bLtwo\rightarrow\bLtwo,\\
      \text{div}^*=-\mr{\mathbf{grad}}:&\,\Honenot\subset\Ltwo\rightarrow\bLtwo,\label{def div star}\
    \end{align}
  \end{subequations}
  respectively. Then, the adjoint operators of \eqref{rnabla}-\eqref{rdiv} are obtained using the fact that $\A^{**}=\A$ for all densely defined and closed unbounded linear operators between Hilbert spaces.

  By abuse of notation, we generally write $\mathbf{grad}=\mr{\mathbf{grad}}$, $\mathbf{curl}=\mr{\mathbf{curl}}$ and $\text{div}=\mr{\text{div}}$.
\end{example}

\subsection{Definition}
\label{sec: Hilbert complexes}
A \emph{Hilbert complex} is a sequence of Hilbert spaces $\W_k$, $k\in \mathbb{Z}$,
together with a sequence of closed and densely defined unbounded linear operators
$\A_k:\dom(\A_k)\subset\W_k\rightarrow \W_{k+1}$ such that
$\range(\A_k)\subset \null(\A_{k+1})$, i.e. $\A_{k+1}\circ\A_{k}\equiv0$ for all
$k\in\mathbb{Z}$. It can be written as
\begin{subequations}
  \begin{equation}
    \label{Hilbert complex Ak}
    \def\arrowlength{5.8ex}
    \def\arrowdistance{.8}
    \begin{tikzcd}[column sep=\arrowlength]
      \cdots 
      \arrow[r, rightarrow, shift left=\arrowdistance, "\A_{k-2}"] 
      & 
      \dom(\A_{k-1})\subset\W_{k-1} 
      \ar[r, rightarrow, shift left=\arrowdistance, "\A_{k-1}"] 
      & 
      \dom(\A_k)\subset\W_k
      \arrow[r, rightarrow, shift left=\arrowdistance, "\A_{k}"] 
      & 
      \dom(\A_{k+1})\subset\W_{k+1}
      \arrow[r, rightarrow, shift left=\arrowdistance, "\A_{k+1}"] 
      &
      \cdots, 
    \end{tikzcd}
  \end{equation}
  cf. \cite[Def. 4.1]{MR3908678}. The associated sequence of adjoint operators spawns the
  so-called dual Hilbert complex
  \begin{equation}\label{Hilbert dual complex Ak}
    \def\arrowlength{5.8ex}
    \def\arrowdistance{.8}
    \begin{tikzcd}[column sep=\arrowlength]
      \cdots 
      \arrow[r, leftarrow, shift right=\arrowdistance, "\A^*_{k-2}"']
      & 
      \dom(A_{k-2}^*) \subset\W_{k-1}
      \ar[r, leftarrow, shift right=\arrowdistance, "\A_{k-1}^{*}"']
      & 
      \dom(A_{k-1}^*) \subset\W_k
      \arrow[r, leftarrow, shift right=\arrowdistance, "\A_{k}^{*}"']
      & 
      \dom(A_{k}^*) \subset \W_{k+1}
      \arrow[r, leftarrow, shift right=\arrowdistance, "\A^*_{k+1}"']
      &
      \cdots, 
    \end{tikzcd}
  \end{equation}
\end{subequations}
which by \eqref{def: hilbert space adjoint} is itself a Hilbert complex, because
$\A^*_{k-1}\circ\A^*_{k}\equiv0$ for all $k\in\mathbb{Z}$. ``Finite'' Hilbert complexes
can be embedded into \eqref{Hilbert complex Ak} by setting $\W_k=\{0\}$ for all
$k\notin\{0,1,...,N\}$.

Notice that since $\range(\A_k)\subset \dom(\A_{k+1})$ and
$\range(\A^*_{k+1})\subset \dom(\A^*_{k})$, the sequences of \emph{bounded} operators
$\A_k:\dom(\A_k)\rightarrow \W_{k+1}$ and $\A^*_k:\dom(\A^*_{k})\rightarrow \W_{k}$ also
induce Hilbert complexes themselves:
\begin{subequations}
  \begin{equation}\label{Hilbert domain complex Ak}
    \def\arrowlength{5.8ex}
    \def\arrowdistance{.8}
    \begin{tikzcd}[column sep=\arrowlength]
      \cdots 
      \arrow[r, rightarrow, shift left=\arrowdistance, "\A_{k-2}"] 
      & 
      \dom(\A_{k-1}) 
      \ar[r, rightarrow, shift left=\arrowdistance, "\A_{k-1}"] 
      & 
      \dom(\A_k)
      \arrow[r, rightarrow, shift left=\arrowdistance, "\A_{k}"] 
      & 
      \dom(\A_{k+1})
      \arrow[r, rightarrow, shift left=\arrowdistance, "\A_{k+1}"] 
      &
      \cdots, 
    \end{tikzcd}
  \end{equation}
  \begin{equation}\label{Hilbert dual domain complex Ak}
    \def\arrowlength{5.8ex}
    \def\arrowdistance{.8}
    \begin{tikzcd}[column sep=\arrowlength]
      \cdots 
      \arrow[r, leftarrow, shift right=\arrowdistance, "\A^*_{k-2}"']
      & 
      \dom(\A^*_{k-2}) 
      \ar[r, leftarrow, shift right=\arrowdistance, "\A_{k-1}^{*}"']
      & 
      \dom(\A^*_{k-1})
      \arrow[r, leftarrow, shift right=\arrowdistance, "\A_{k}^{*}"']
      & 
      \dom(\A^*_{k})
      \arrow[r, leftarrow, shift right=\arrowdistance, "\A^*_{k+1}"']
      &
      \cdots. 
    \end{tikzcd}
  \end{equation}
\end{subequations}
These are examples of \emph{bounded} Hilbert complexes in which every operator is continuous. We refer to \eqref{Hilbert domain complex Ak} and \eqref{Hilbert domain complex Ak} as the \emph{domain complexes} of \eqref{Hilbert complex Ak} and \eqref{Hilbert dual complex Ak}.

If the range $\range(\A_k)$ is a closed subset of $\W_{k+1}$ for all $k$, we say that the
Hilbert complex \eqref{Hilbert complex Ak} is \emph{closed}. If this is the case, then
$\range(\A_k^*)$ is also closed in $\W_{k}$ by the closed range theorem
\cite[Thm. 3.7]{MR3908678}, making the dual complex \eqref{Hilbert dual complex Ak} a
closed Hilbert complex too. Furthermore, \eqref{Hilbert complex Ak} is said to be
\emph{Fredholm} if the codimension of $\range(\A_k)$ is finite in $\null(\A_{k+1})$---in
which case it is also closed by \cite[Thm. 3.8]{MR3908678}. Equivalently, a Hilbert
complex is Fredholm if the quotient spaces $\null(\A_{k+1})/\range(\A_{k})$ and
$\null(\A^*_{k})/\range(\A^*_{k+1})$ are finite dimensional, in other words, if the
cohomology spaces of \eqref{Hilbert complex Ak} and \eqref{Hilbert dual complex Ak} have
finite dimension. It is a sufficient condition for a Hilbert complex to be Fredholm to
satisfy the \emph{compactness property}, that is, the embedding
$\dom(\A_k)\cap \dom(\A_{k-1}^*) \hookrightarrow \W_{k}$ is compact for all
$k\in\mathbb{Z}$.

\begin{example}{The $L^2$ de Rham complex in $\mathbb{R}^3$}
  \label{spe: L2 de Rham complex}
  The $L^2$ de Rham complex \eqref{L2 de Rham complex intro} is a standard example of a
  Hilbert complex, where $\A_{k}\equiv 0$ and $\W_{k}=\{0\}$ is set for
  $k\in\mathbb{Z}\backslash\{0,1,2,3\}$. Its dual complex is represented by the sequence 
  \begin{equation}
    \def\arrowlength{5.4ex}
    \def\arrowdistance{.4}
    \begin{tikzcd}[column sep=\arrowlength]
      \{0\}
      \arrow[r, twoheadleftarrow, shift right=\arrowdistance, "\,\,\,0"']
      & 
      L^2(\Omega)
      \arrow[r, leftarrow, shift right=\arrowdistance, "-\text{div}"']
      & 
      \Hdivnot\subset\mathbf{L}^2(\Omega) 
      \ar[r, leftarrow, shift right=\arrowdistance, "\mathbf{curl}"']
      & 
      \Hcurlnot\subset\bLtwo
      \arrow[r, leftarrow, shift right=\arrowdistance, "-\mathbf{grad}"']
      & 
      \Honenot\subset L^2(\Omega)
      \arrow[r, hookleftarrow, shift right=\arrowdistance, "\imath"']
      & 
      \{0\},
    \end{tikzcd}
  \end{equation}
  cf. \cite[Sec. 3.4]{MR3908678} and \cite[Sec. 4.3]{MR3908678}, and its embedding into
  our abstract framework is summarized in the following table: 
  \begin{center}
    \begin{tabular}{ c|cccccc } 
      $k$ &$\W_k$& $\A_k$ & $\dom(\A_k)$ &  $\A^*_k$& $\dom(\A^*_k)$ & $\dom(\A_k)\cap \dom(\A^*_{k-1})$\\
      \midrule
      $0$ &$L^2(\Omega)$& $\mathbf{grad}$ & $H^1(\Omega)$ & $-\,\text{div}$ & $\Hdivnot$ &$H^1(\Omega)$\rule{0pt}{3ex}\\ 
      $1$ &$\mathbf{L}^2(\Omega)$ & $\mathbf{curl}$ & $\Hcurl$ & $\mathbf{curl}$ &$\Hcurlnot$ &$\Hcurl\cap\Hdivnot$\rule{0pt}{3ex}\\ 
      $2$ &$\mathbf{L}^2(\Omega)$ & $\text{div}$ & $\Hdiv$ & $-\mathbf{grad}$ &$\Honenot$ & $\Hdiv\cap\Hcurlnot$\rule{0pt}{3ex}\\
      $3$ & $L^2(\Omega)$& $ 0$ & $L^2(\Omega)$ & $\text{Id}$& $\{0\}$ & $\Honenot$\rule{0pt}{3ex}\\  
    \end{tabular}
  \end{center}
  
  The de Rham complex satisfies the compactness property, and thus it is Fredholm. Indeed,
  recall that Rellich's compact embedding theorem states that the inclusion of
  $H^1(\Omega)$ and $\Honenot$ in $L^2(\Omega)$ is compact. We refer to \cite{MR753428}
  for a proof that $\Hcurl\cap\Hdivnot$ and $\Hdiv\cap\Hcurlnot$ are compactly embedded in
  $\mathbf{L}^2(\Omega)$.  \medskip
\end{example}

\subsection{Basic setting}
\label{sec: Baisc setting}

Now, let a Hilbert complex as in \eqref{Hilbert complex Ak} be given and suppose that the
unbounded linear operators of a second Hilbert complex
\begin{subequations}
  \begin{equation}\label{Hilbert complex rAk}
    \def\arrowlength{6ex}
    \def\arrowdistance{.4}
    \begin{tikzcd}[column sep=\arrowlength]
      \cdots 
      \arrow[r, rightarrow, shift left=\arrowdistance, "\rA_{k-2}"] 
      & 
      \dom(\rA_{k-1})\subset\W_{k-1} 
      \ar[r, rightarrow, shift left=\arrowdistance, "\rA_{k-1}"] 
      & 
      \dom(\rA_k)\subset\W_k
      \arrow[r, rightarrow, shift left=\arrowdistance, "\rA_{k}"] 
      & 
      \dom(\rA_{k+1})\subset\W_{k+1}
      \arrow[r, rightarrow, shift left=\arrowdistance, "\rA_{k+1}"] 
      &
      \cdots 
    \end{tikzcd}
  \end{equation}
  are such that $\rA_k\subset\A_k$, i.e. $\dom(\rA_k)\subset \dom(\A_k)$ and
  $\A_k\vert_{\dom(\rA_k)}=\rA_k$. In other words, for all $k\in \mathbb{Z}$, $\A_k$ is an
  extension of $\rA_k$. It is easy to verify that the adjoint operators
  $\AT_k:=\rA^*_k:\dom(\rA^*_k)\subset\W_{k+1}\rightarrow\W_k$ involved in the dual
  complex
  \begin{equation}\label{Hilbert dual complex ATk}
    \def\arrowlength{6ex}
    \def\arrowdistance{.4}
    \begin{tikzcd}[column sep=\arrowlength]
      \cdots 
      \arrow[r, leftarrow, shift right=\arrowdistance, "\AT_{k-2}"']
      & 
      \dom(\AT_{k-2}) \subset\W_{k-1}
      \ar[r, leftarrow, shift right=\arrowdistance, "\AT_{k-1}"']
      & 
      \dom(\AT_{k-1}) \subset\W_k
      \arrow[r, leftarrow, shift right=\arrowdistance, "\AT_{k}"']
      & 
      \dom(\AT_{k}) \subset \W_{k+1}
      \arrow[r, leftarrow, shift right=\arrowdistance, "\AT_{k+1}"']
      &
      \cdots 
    \end{tikzcd}
  \end{equation}
\end{subequations}
are such that $\A^*_k\subset\AT_k$.  In particular, the bounded domain complexes
\begin{subequations}
  \begin{equation}\label{Hilbert domain complex rAk}
    \def\arrowlength{6ex}
    \def\arrowdistance{.8}
    \begin{tikzcd}[column sep=\arrowlength]
      \cdots 
      \arrow[r, rightarrow, shift left=\arrowdistance, "\rA_{k-2}"] 
      & 
      \dom(\rA_{k-1}) 
      \ar[r, rightarrow, shift left=\arrowdistance, "\rA_{k-1}"] 
      & 
      \dom(\rA_k)
      \arrow[r, rightarrow, shift left=\arrowdistance, "\rA_{k}"] 
      & 
      \dom(\rA_{k+1})
      \arrow[r, rightarrow, shift left=\arrowdistance, "\rA_{k+1}"] 
      &
      \cdots, 
    \end{tikzcd}
  \end{equation}
  \begin{equation}\label{Hilbert dual domain complex rAk}
    \def\arrowlength{6ex}
    \def\arrowdistance{.8}
    \begin{tikzcd}[column sep=\arrowlength]
      \cdots 
      \arrow[r, leftarrow, shift right=\arrowdistance, "\A^*_{k-2}"']
      & 
      \dom(\A^*_{k-2}) 
      \ar[r, leftarrow, shift right=\arrowdistance, "\A^*_{k-1}"']
      & 
      \dom(\A^*_{k-1})
      \arrow[r, leftarrow, shift right=\arrowdistance, "\A^*_{k}"']
      & 
      \dom(\A^*_{k})
      \arrow[r, leftarrow, shift right=\arrowdistance, "\A^*_{k+1}"']
      &
      \cdots, 
    \end{tikzcd}
  \end{equation}
\end{subequations}
are examples of Hilbert \emph{subcomplexes} of the domain Hilbert complexes \eqref{Hilbert
  domain complex Ak} and \eqref{Hilbert dual domain complex Ak}.

For reference, this basic setting is summarized in the following assumption.
\begin{statement}
  \label{A}
  \label{Assumption trace}
  For all $k\in\mathbb{Z}$ let $\W_k$ be real Hilbert spaces, and suppose that
  $\A_k:\dom(\A_k)\subset\W_k\rightarrow\W_{k+1}$ and
  $\rA_k:\dom(\rA_k)\subset\W_k\rightarrow\W_{k+1} $ are densely defined and closed
  unbounded linear operators such that $\range(\A_k)\subset \null(\A_{k+1})$,
  $\range(\rA_k)\subset \null(\rA_{k+1})$, and $\A_k$ is an extension of $\rA_k$,
  i.e. $\dom(\rA_k)\subset\dom(\A_k)$ and $\A_k\x_{\circ}=\rA_k\x_{\circ}$ for all
  $\x_{\circ}\in\dom(\rA_k)$.
\end{statement}

\begin{example}{Boundary conditions}\label{spe: Boundary conditions}
  The Hilbert complex 
  \begin{subequations}
    \begin{equation}
      \def\arrowlength{5.4ex}
      \def\arrowdistance{.4}
      \begin{tikzcd}[column sep=\arrowlength]
        \{0\}
        \arrow[r, hookrightarrow, shift left=\arrowdistance, "\imath"] 
        & 
        H_0^1(\Omega)\subset L^2(\Omega) 
        \arrow[r, rightarrow, shift left=\arrowdistance, "\mathbf{grad}"] 
        & 
        \Hcurlnot\subset\mathbf{L}^2(\Omega) 
        \ar[r, rightarrow, shift left=\arrowdistance, "\mathbf{curl}"] 
        & 
        \Hdivnot\subset\mathbf{L}^2(\Omega) 
        \arrow[r, rightarrow, shift left=\arrowdistance, "\text{div}"] 
        & 
        L^2(\Omega)
        \arrow[r, twoheadrightarrow, shift left=\arrowdistance, "0\,\,\,"] 
        & 
        \{0\}
      \end{tikzcd}
    \end{equation}
    fulfills the hypothesis on \eqref{Hilbert complex rAk} for the $L^2$ de Rham complex \eqref{L2 de Rham complex intro}. Owing to \eqref{def grad star}-\eqref{def div star}, its dual complex is written
    \begin{equation}
      \def\arrowlength{5.4ex}
      \def\arrowdistance{.4}
      \begin{tikzcd}[column sep=\arrowlength]
        \{0\}
        \arrow[r, twoheadleftarrow, shift right=\arrowdistance, "\,\,\,0"']
        & 
        L^2(\Omega)
        \arrow[r, leftarrow, shift right=\arrowdistance, "-\text{div}"']
        & 
        \Hdiv\subset\mathbf{L}^2(\Omega) 
        \ar[r, leftarrow, shift right=\arrowdistance, "\mathbf{curl}"']
        & 
        \Hcurl\subset\bLtwo
        \arrow[r, leftarrow, shift right=\arrowdistance, "-\mathbf{grad}"']
        & 
        H^1\subset L^2(\Omega)
        \arrow[r, hookleftarrow, shift right=\arrowdistance, "\imath"']
        & 
        \{0\}.
      \end{tikzcd}
    \end{equation}
  \end{subequations}
  Summing up, the various operators and spaces have the following incarnations for
  the de Rham complex in three-dimensional Euclidean space:
  \begin{center}
    \begin{tabular}{ c|cccccc } 
      $k$ &$\W_k$& $\rA_k$ & $\dom(\rA_k)$ &  $\AT_k$& $\dom(\AT_k)$ & $\dom(\rA_k)\cap \dom(\AT_{k-1})$\\
      \midrule
      $0$ &$L^2(\Omega)$& $\mathbf{grad}$ & $\Honenot$ & $-\,\text{div}$ & $\Hdiv$ &$\Honenot$\rule{0pt}{3ex}\\ 
      $1$ &$\mathbf{L}^2(\Omega)$ & $\mathbf{curl}$ & $\Hcurlnot$ & $\mathbf{curl}$ &$\Hcurl$ &$\Hcurlnot\cap\Hdiv$\rule{0pt}{3ex}\\ 
      $2$ &$\mathbf{L}^2(\Omega)$ & $\text{div}$ & $\Hdivnot$ & $-\mathbf{grad}$ &$H^1(\Omega)$ & $\Hdivnot\cap\Hcurl$\rule{0pt}{3ex}\\
      $3$ & $L^2(\Omega)$& $ 0$ & $L^2(\Omega)$ & $\text{Id}$& $\{0\}$ & $H^1(\Omega)$\rule{0pt}{3ex}\\  
    \end{tabular}
  \end{center}
\end{example}

\section{Trace Operators}
\label{sec: Basic traces}

The following sections lay the foundations of a general quotient-based abstract theory for
traces in Hilbert spaces. To that end, we do not require the full structure of Hilbert
complexes, but it suffices to focus on the following snippet of the Hilbert complexes
\eqref{Hilbert complex Ak} and \eqref{Hilbert complex rAk}:
\[
  \tikz[
  overlay]{
    \filldraw[fill=black!10,draw=black!10] (5.65,1.2) rectangle (11.75,-1);
  }
  \def\arrowlength{6ex}
  \def\arrowdistance{.8}
  \begin{tikzcd}[column sep=\arrowlength, row sep=0.05cm]
    \cdots 
    \arrow[r, rightarrow, shift left=\arrowdistance, "\A_{k-2}"] 
    & 
    \dom(\A_{k-1})\subset\W_{k-1} 
    \ar[r, rightarrow, shift left=\arrowdistance, "\A_{k-1}"] 
    & 
    \dom(\A_k)\subset\W_k
    \arrow[r, rightarrow, shift left=\arrowdistance, "\A_{k}"] 
    & 
    \dom(\A_{k+1})\subset\W_{k+1}
    \arrow[r, rightarrow, shift left=\arrowdistance, "\A_{k+1}"] 
    &
    \cdots, \\
    \,
    & 
    \cup\qquad\quad
    & 
    \cup\qquad\quad
    & 
    \cup\qquad\quad
    &
    \, \\
    \cdots 
    \arrow[r, rightarrow, shift left=\arrowdistance, "\rA_{k-2}"] 
    & 
    \dom(\rA_{k-1})\subset\W_{k-1} 
    \ar[r, rightarrow, shift left=\arrowdistance, "\rA_{k-1}"] 
    & 
    \dom(\rA_k)\subset\W_k
    \arrow[r, rightarrow, shift left=\arrowdistance, "\rA_{k}"] 
    & 
    \dom(\rA_{k+1})\subset\W_{k+1}
    \arrow[r, rightarrow, shift left=\arrowdistance, "\rA_{k+1}"] 
    &
    \cdots .
  \end{tikzcd}
\]

In the sequel, we fix $k\in\mathbb{Z}$ and take for granted Assumption \ref{Assumption trace}.

\subsection{Hilbert traces}
\label{sec: Hilbert traces}

Using the shorthand $\AT_k:=\rA_k^*:\dom(\AT_k)\subset\W_{k+1}\rightarrow\W_k$, it follows
from the estimate
\begin{align}
  \begin{split}\label{eq: trace continuity estimate}
    \vert (\A_k\x,\y)_{\W_{k+1}} - (\x,\AT_k\y)_{\W_k} \vert
    & \leq \|\A_k\x\|_{\W_{k+1}}\|\y\|_{\W_{k+1}} + \|\x\|_{\W_k}\|\AT_k\y\|_{\W_k}\\
    & \leq \|\x\|_{\dom(\A_k)}\|\y\|_{\dom(\AT_k)}
  \end{split}
\end{align}
that the following definition of a particular notion of a trace makes sense.
\medskip

\defbox{%
\begin{defi}\label{def: trace enviro}
  \label{def:trace1}
  In the setting of Assumption \ref{Assumption trace}, the bounded linear operator 
  \begin{equation}
    \cred{\traceA:\dom(\A_k)\rightarrow \dom(\AT_k)'}
  \end{equation} 
  defined  for all $\x\in \dom(\A_k)$ and $\y\in \dom(\AT_k)$ by
  \begin{align}\label{eq: def of trace}
    \langle\traceA\x,\y\rangle_{\dom(\AT_k)'}:= (\A_k\x,\y)_{\W_{k+1}} - (\x,\AT_k\y)_{\W_k}
  \end{align}
  is called the \cred{\emph{(primal)  Hilbert trace}} associated with the pair of operators $\A_k$ and $\rA_k$.
\end{defi}%
}

It also follows from \eqref{eq: trace continuity estimate} that
\begin{equation}
  \label{eq: norm of weak trace is one}
  \|\traceA\|=1,
\end{equation}
where $\|\cdot\|$ is the operator norm. 

We point out that defining a trace operator as a mapping into a dual space has precedents
in the theory of Friedrichs operators, has been pursued in \cite[Sect.~2.2]{ERG07} and
\cite[Sect.~56.3.2]{ERG21c}, and is also discussed in \cite{ANB11,ANB10a,ANB10}. In these
works, the authors have dubbed ``boundary operators'' what we have decided to call ``Hilbert traces''.

Let us motivate the above notion of trace with classical examples.
\begin{example}{Hilbert traces}
  \label{spe: Classical Hilbert traces}
  Applying \Cref{def:trace1} in the 3D de Rham setting \ref{spe:Differential operators},
  we obtain the Hilbert traces
  \begin{subequations}
    \begin{align}
      &\trace_0=\trace_{\scriptscriptstyle\mathbf{grad}}:\Hone\rightarrow \Hdiv',\label{trace: Dirichlet general mapping}\\
      &\trace_1=\trace_{\scriptscriptstyle\mathbf{curl}}\,:\Hcurl\rightarrow \Hcurl',\label{trace: t general mapping}\\
      &\trace_2=\trace_{\scriptscriptstyle\text{div}}\,\,\,:\Hdiv\rightarrow \Hone'\label{trace: normal general mapping},
    \end{align}
  \end{subequations}
  defined by
  \begin{subequations}
    \begin{align}
      \langle\trace_{\scriptscriptstyle\mathbf{grad}} \,v,\u\rangle_{\Hdiv'} &:= (\mathbf{grad} v,\u)_{\bLtwo} + (v,\text{div}\,\u)_{L^2(\Omega)},\label{trace: dirichlet variation}\\
      \langle\trace_{\scriptscriptstyle\mathbf{curl}} \z,\w\rangle_{\Hcurl'} &:= (\mathbf{curl}\,\z,\w)_{\bLtwo} - (\z,\mathbf{curl}\,\w)_{\bLtwo}, \label{trace: t variation}\\
      \langle\trace_{\scriptscriptstyle\text{div}} \u,v\rangle_{\Hone'} &:= (\text{div}\,\u,v)_{\bLtwo} + (\u,\mathbf{grad}\,v)_{\bLtwo},\label{trace: normal variation}
    \end{align}
  \end{subequations}
  for all $v\in H^1$, $\u\in \Hdiv$ and $\z,\w\in \Hcurl$.

  We recognize on the right hand sides of \eqref{trace: dirichlet variation}-\eqref{trace: normal variation} the continuous bilinear forms occurring in Green's formulas \eqref{IBP div} and \eqref{IBP curl}. Introducing the operators 
  \begin{align}
    \gamma_n':H^{1/2}(\Gamma)\rightarrow \Hdiv',
    &&\gamma_{t}'&:\mathbf{H}^{-1/2}(\text{curl}_\Gamma,\Gamma)\rightarrow\Hcurl',
    &&\gamma':H^{-1/2}\rightarrow\Hone',
  \end{align}
  dual to the classical traces, where we have identified $H^{-1/2}(\Gamma)$ with $(H^{1/2}(\Gamma))'$ through the $L^2(\Gamma)$-pairing on the boundary and $\mathbf{H}^{-1/2}(\text{curl}_\Gamma,\Gamma)$ with its own dual through the skew-symmetric pairing defined in \eqref{skew-symmetric pairing}, we obtain
  \begin{align}
    \trace_{\scriptscriptstyle\mathbf{grad}} = \gamma_n'\circ\gamma,
    &&\trace_{\scriptscriptstyle\mathbf{curl}} = \gamma_{t}'\circ\gamma_t,
    &&\trace_{\scriptscriptstyle\text{div}} = \gamma'\circ\gamma_n.
  \end{align}
  Observe that
  \begin{equation}\label{T0 and T2 are dual}
    (\trace_{\scriptscriptstyle\mathbf{grad}})'=\trace_{\scriptscriptstyle\text{div}}.
  \end{equation}

  The appeal of definitions \eqref{trace: dirichlet variation}-\eqref{trace: normal
    variation} is that they do not explicitly  depend on $\Gamma$. In fact, notice that
  they are well-defined for general bounded open sets $\Omega$ without any assumption on
  the regularity of their boundary $\Gamma:=\partial\Omega$.
\end{example}

\begin{prop}\label{lem: kernel of trace}
  Under Assumption \ref{Assumption trace},
  \begin{equation}
    \null(\traceA) = \dom(\rA_k).
  \end{equation}
\end{prop}
\begin{proof}
  On the one hand, for any $\x_{\circ }\in \dom(\rA_k)$, it follows from
  $\rA_k\subset\A_k$ and \eqref{def: hilbert space adjoint} that
  \begin{align}
    \begin{split}
      \langle\traceA\x_{\circ },\y\rangle_{\dom(\AT_k)'}&=(\A_k\x_{\circ },\y)_{\W_{k+1}} - (\x_{\circ },\AT_k\y)_{\W_k}
      =(\rA_k\x_{\circ },\y)_{\W_{k+1}} - (\x_{\circ},\AT_k\y)_{\W_k}  \\
      &=(\x_{\circ },\AT_k\y)_{\W_{k+1}} - (\x_{\circ },\AT_k\y)_{\W_k}=0
    \end{split}
  \end{align}
  for all $\y\in \dom(\AT_k)$. This shows that $\dom(\rA_k)\subset \null(\traceA)$. 

  On the other hand, if $\x\in \dom(\A_k)$ is such that $\x\in \null(\traceA)$, then
  \begin{align}
    0=\langle\traceA\x,\y\rangle_{\dom(\AT_k)'}=(\A_k\x,\y)_{\W_{k+1}} -
    (\x,\AT_k\y)_{\W_k}
    \quad \forall\y\in \dom(\AT_k).
  \end{align}
  If we set $C_{\x}:=\norm{\x}_{\dom(\A_k)}$, we see that
  \begin{align}
    \vert(\x,\AT_k\y)_{\W_k}\vert &= \vert(\A_k\x,\y)_{\W_{k+1}}\vert
    \leq\norm{\A_k\x}_{\W_{k+1}}\norm{\y}_{\W_{k+1}}
    \leq C_{\x}\norm{\y}_{\W_{k+1}}\quad\forall\y\in \dom(\AT_k).
  \end{align}
  As explained in \Cref{sec: Hilbert complexes}, this means that
  $\x\in \dom((\AT_k)^*)=\dom(\rA_k^{**})=\dom(\rA_k)$.
\end{proof}

\begin{example}{Kernels of classical Hilbert traces}\label{eg: Kernels of classical Hilbert traces}
  Comparing \Cref{lem: kernel of trace} with \eqref{eq: ker Dirichlet trace}-\eqref{eq:
    ket n trace}, we verify that
  \begin{align}
    \null(\trace_{\scriptscriptstyle\mathbf{grad}})=\null(\gamma), &&\null(\trace_{\scriptscriptstyle\mathbf{curl}})=\null(\gamma_t), &&\null(\trace_{\scriptscriptstyle\text{div}})=\null(\gamma_n).
  \end{align}
\end{example}

\begin{rem}
  Intuitively, we think of a trace operator as a means of imposing ``boundary
  conditions''. The idea behind \Cref{def:trace1} is to impose these boundary conditions
  on the operator itself, which is a common strategy in the analysis of variational
  problems and related operator equations. In this work, $\A_k$ is the operator of
  interest. We regard $\rA_k$ as the operator on which boundary conditions are
  imposed. From that perspective, the operator $\AT_k$ does not feature boundary conditions. The right
  hand side of \eqref{eq: def of trace} plays a role akin to the bilinear form involved in
  classical integration by parts formulas.
\end{rem}

\subsection{Trace spaces}
\label{sec: Trace spaces}

Recall that by hypothesis, $\dom(\A_k^*)\subset \dom(\AT_k)$. The next proposition involves
the annihilator of $\dom(\A_k^*)$ in $\dom(\AT_k)'$:
\begin{equation}
  \cred{\dom(\A_k^*)^{\circ}} :=\left\{\,\bm{\phi}\in \dom(\AT_k)'\,\,\vert\,\,\langle\bm{\phi},\y\rangle_{\dom(\AT_k)'} =0\,\,\forall\y\in \dom(\A_k^*)\,\right\}\subset \dom(\AT_k)'.
\end{equation}

\thmbox{
\begin{prop}\label{prop: range is annihilator}
  Under Assumption \ref{Assumption trace}, we find for the ranges of the Hilbert traces
  \begin{equation}
    \range(\traceA) = \dom(\A_k^*)^{\circ}.
  \end{equation}
\end{prop}}

\begin{proof}	
  Suppose that $\bm{\phi}\in \dom(\A_k^*)^{\circ}$ and let $\w\in \dom(\AT_k)$ be its Riesz representative in $\dom(\AT_k)$, that is
  \begin{align}\label{riesz rep range proof}
    \langle\bm{\phi},\y\rangle_{\dom(\AT_k)'} = (\w,\y)_{\dom(\AT_k)} &&\forall\y\in \dom(\AT_k).
  \end{align}

  We claim that $\x:=-\AT_k\w\in \dom(\A_k)$. Indeed, \eqref{riesz rep range proof}
  implies that for all $\y_*\in \dom(\A_k^*)$, we have
  \begin{equation}\label{eq: riesz rep evaluates to zero in range proof}
    0 = (\w,\y_*)_{\dom(\AT_k)} = (\w,\y_*)_{\W_{k+1}}+ (\AT_k\w,\AT_k\y_*)_{\W_k} = (\w,\y_*)_{\W_{k+1}}+ (\AT_k\w,\A_k^*\y_*)_{\W_k}.
  \end{equation}
  This means $(\w,\y_*)_{\W_{k+1}}=(\x,\A_k^*\y_*)_{\W_k}$.
  Therefore, if we set $C_{\x}:=\norm{\w}_{\W_{k+1}}$, we find the estimate
  \begin{align}
    \vert(\x,\A_k^*\y_*)_{\W_k}\vert = \vert(\w,\y_*)_{\W_{k+1}}\vert\leq\norm{\w}_{\W_{k+1}}\norm{\y_*}_{\W_{k+1}} = C_{\x}\norm{\y_*}_{\W_{k+1}} &&\forall\y_*\in \dom(\A_k^*),
  \end{align}
  which as explained in \Cref{sec: Operators on Hilbert spaces} implies that $\x\in \dom(\A_k^{**})=\dom(\A_k)$.

  In particular, according to \eqref{def: hilbert space adjoint}, it also follows from \eqref{eq: riesz rep evaluates to zero in range proof} that $\A_k\x=\w$. Hence, the inclusion $\range(\traceA) \supset \dom(\A_k^*)^{\circ}$ is verified by observing that for all $\y\in \dom(\AT_k)$,
  \begin{align}
    \begin{split}
      \langle\traceA\x,\y\rangle_{\dom(\AT_k)'} &= (\A_k\x,\y)_{\W_{k+1}} - (\x,\AT_k\y)_{\W_k}
      = (\w,\y)_{\W_{k+1}} + (\AT_k\w,\AT_k\y)_{\W_k}\\
      &= (\w,\y)_{\dom(\AT_k)}
      = \langle \bm{\phi},\y\rangle_{\dom(\AT_k)'},
    \end{split}
  \end{align}
  i.e. $\traceA\x=\bm{\phi}$.

  To show that $\range(\traceA) \subset \dom(\A_k^*)^{\circ}$, let $\bm{\phi}=\traceA\x$
  for some $\x\in \dom(\A_k)$. Then, since $\A_k^*\subset\AT_k$, we obtain by \eqref{def:
    hilbert space adjoint} that for all $\y_*\in \dom(\A_k^*)$
  \begin{align}
    \langle\bm{\phi},\y_*\rangle_{\dom(\AT_k)'} = (\A_k\x,\y_*)_{\W_{k+1}} - (\x,\AT_k\y_*)_{\W_k} = (\x,\A_k^*\y_*)_{\W_{k+1}} - (\x,\A_k^*\y_*)_{\W_k}=0,
  \end{align}
  i.e. $\bm{\phi}\in \dom(\A_k^*)^{\circ}$.
\end{proof}

Since $\dom(\rA_k)$ is a Hilbert subspace of $\dom(\A_k)$, it is closed and we can proceed
with the next definition.
\medskip

\defbox{
\begin{defi}
  \label{def:tracespace1}
  In the setting of \Cref{def:trace1}, we call \emph{trace spaces} the quotient spaces
  \begin{equation}\label{eq: def of trace space}
    \cred{\tracespace:=\dom(\A_k)/\dom(\rA_k)},
  \end{equation}
  equipped with the quotient norm 
  \begin{align}
    \norm{[\x]}_{\tracespace}:=\inf_{\mr{z}\in \dom(\rA_k)}\norm{\x-\mr{\z}}_{\dom(\A_k)}
    \quad
    \forall\x\in \dom(\A_k).
  \end{align}
\end{defi}}

\begin{rem}\label{rem: trace space is quotient by kernel}
  Notice that due to \Cref{lem: kernel of trace},
  \begin{equation}
    \tracespace =\dom(\A_k)/\null(\traceA).
  \end{equation}
\end{rem}

In \Cref{def:tracespace1}, the equivalence class in $\tracespace$ of $\x\in D(\A_k)$ is
denoted $[\x]=\{\x + \mr{\z} \,\vert\, \mr{\z}\in \dom(\rA_k)\}$. Write
$\bm{\pi}^{t}_{\A_k}: \dom(\A_k)\to\tracespace$ for the canonical projection (also
frequently called quotient map), i.e. $\bm{\pi}^t_{\A_k}(\x) = [\x]$. It is an application
of a classical theorem of functional analysis that there exists a bounded orthogonal
projection $ \mathsf{P}_{k}^t: \dom(\A_k)\to \dom(\rA_k)^{\perp} $ onto the complement
space
\begin{equation}
  \dom(\rA_k)^{\perp} :=\left\{\,\x\in \dom(\A_k)\,\,\vert\,\,(\x,\mr{\z})_{\dom(\A_k)}=0\,\,\forall\mr{\z}\in \dom(\rA_k)\,\right\}\subset\dom(\A_k)
\end{equation}
such that
\begin{align}\label{eq: projection norm is norm of quotient}
  \norm{\mathsf{P}_{k}^t\x}_{\dom(\A)} =\norm{[\x]}_{\tracespace} &&\forall\x\in \dom(\A_k),
\end{align}
cf. \cite[Chap. 3.1]{MR0350358} and \cite[Chap. 5]{MR2759829}. Write $\imath^t_{k}:\dom(\rA_k)^{\perp}\hookrightarrow \dom(\A_k)$ for canonical inclusion maps. Since 
$
\null(\mathsf{P}_{k}^t)= \dom(\rA_k)
$ by \eqref{eq: projection norm is norm of quotient} , the bounded linear map $\mathsf{G}_{k}^t:\tracespace\to \dom(\rA_k)^{\perp}$ defined by $\mathsf{G}_{k}^t[\x]:=\mathsf{P}_{k}^t\x$ and involved in the commutative diagram
\begin{equation}\label{eq: commutative diag Pt and Gt}
  \begin{tikzcd}
    \dom(\A_k) \arrow[rr, two heads, blue, "\mathsf{P}_{k}^t"] \arrow[rd, two heads, orange, "\bm{\pi}^t_{k}"'] &  & \dom(\rA_k)^{\perp} \\
    & \dom(\A_k)/\null(\mathsf{P}_{k}^t)=\tracespace \arrow[ru, hookrightarrow, two heads, blue, dotted, "\mathsf{G}_{k}^t"'] & 
  \end{tikzcd}
\end{equation}
as provided by the first isomorphism theorem for modules is a well-defined isometric isomorphism, cf. \cite[Chap. 10.2, Thm. 4]{MR2286236}. Since $\dom(\rA_k)^{\perp}$ is closed \cite[Chap. 3.1, Thm. 1]{MR0350358}, it is a Hilbert space, and therefore so is $\tracespace$. The quotient norm is induced by the inner product 
\begin{align}\label{eq: inner product tracespace}
  ([\x],[\z])_{\tracespace}:=(\mathsf{P}_{k}^t\x,\mathsf{P}_{k}^t\z)_{\dom(\A_k)} &&\forall [\x],[\z]\in \tracespace.
\end{align}
\begin{rem}
  Notice that	$\null(\mathsf{P}_{k}^t)=\dom(\rA_k)=\null(\traceA)$. 
\end{rem}

That the projection $\mathsf{P}_{k}^t$ is orthogonal means that $(\x-\mathsf{P}_{k}^t\x,\z_{\perp})_{\dom(\A_k)}=0$ for all $\x\in \dom(\A_k)$ and $\z_{\perp}\in \dom(\rA_k)^{\perp}$. In other words, $(\id-\mathsf{P}_{k}^t)\x\in \dom(\rA_k)$ for all $\x\in \dom(\A_k)$. Hence, the simple observation that $\id = \mathsf{P}_{k}^t + (\id - \mathsf{P}_{k}^t)$ shows that any element $\x\in \dom(\A_k)$ can be decomposed as
\begin{equation}\label{decomposition of D(A)}
  \x = \x_{\perp}+ \x_{\circ}
\end{equation}
where $\x_{\perp}\in \dom(\rA_k)^{\perp}$ and $\x_{\circ}\in D(\rA_k)$. It is easy to see that the decomposition \eqref{decomposition of D(A)} is unique.

\begin{example}{Trace spaces}\label{spe: trace spaces}
  In the 3D de Rham setting \ref{spe: Classical Hilbert traces}, applying \Cref{def:tracespace1} leads to
  \begin{subequations}
    \label{eq:drqspc}
    \begin{align}
      &\mathcal{T}(\A_0)=\mathcal{T}(\mathbf{grad})=  \Hone/\Honenot,\label{trace space Dirichlet}\\
      &\mathcal{T}(\A_1)=\mathcal{T}(\mathbf{curl})\,\,=  \Hcurl/\Hcurlnot,\\
      &\mathcal{T}(\A_2)=\mathcal{T}(\text{div})\quad =  \Hdiv/\Hdivnot. \label{trace space normal}
    \end{align}
  \end{subequations}
  Based on Example 3.3, the linear mappings
  \begin{subequations}
    \begin{align}
      \mathsf{X}_{\scriptscriptstyle\mathbf{grad}} :\,& \Hone/\Honenot 	\to H^{1/2}(\Gamma),\\
      \mathsf{X}_{\scriptscriptstyle\mathbf{curl}} :\,& \Hcurl/\Hcurlnot\to \mathbf{H}^{-1/2}(\text{curl}_\Gamma,\Gamma),\\
      \mathsf{X}_{\scriptscriptstyle\text{div}} :\,& \Hdiv/\Hdivnot \to H^{-1/2}(\Gamma)
    \end{align}
  \end{subequations}
  defined by
  \begin{subequations}
    \begin{align}
      \mathsf{X}_{\scriptscriptstyle\mathbf{grad}}[u] &= \gamma u &&\forall u\in \Hone,\\
      \mathsf{X}_{\scriptscriptstyle\mathbf{curl}} [\u] &= \gamma_t\u &&\forall \u\in \Hcurl,\\
      \mathsf{X}_{\scriptscriptstyle\text{div}}[\v] &= \gamma_n \v &&\forall \v\in\Hdiv,
    \end{align}
  \end{subequations}
  are the Hilbert space isomorphisms induced by the canonical projections involved in the following commutative diagrams:
  \begin{equation*}
    \begin{tikzcd}[column sep=small]
      \Hone \arrow[r, two heads, "\gamma"] \arrow[d, two heads, "\bm{\pi}^t_{\scriptscriptstyle\mathbf{grad}}"'] &   H^{1/2}(\Gamma)        &     \Hcurl \arrow[r, two heads, "\gamma_t"] \arrow[d, two heads, "\bm{\pi}^t_{\scriptscriptstyle\mathbf{curl}}"']        &   \mathbf{H}^{-1/2}(\text{curl}_\Gamma,\Gamma)    &   \Hcurl \arrow[r, two heads, "\gamma_n"] \arrow[d, two heads, "\bm{\pi}^t_{\scriptscriptstyle\text{div}}"']     &  H^{-1/2}(\Gamma)   \\
      \mathcal{T}(\mathbf{grad}) \arrow[ru, leftrightarrow,  "\mathsf{X}_{\scriptscriptstyle\mathbf{grad}}"']      &                          &     \mathcal{T}(\mathbf{curl}) \arrow[ru, leftrightarrow,  "\mathsf{X}_{\scriptscriptstyle\mathbf{curl}}"']        &      &    \mathcal{T}(\text{div}) \arrow[ru, leftrightarrow,  "\mathsf{X}_{\scriptscriptstyle\mathbf{curl}}"']      &
    \end{tikzcd}
  \end{equation*}
  The trace spaces $H^{1/2}(\Gamma)$, $\mathbf{H}^{-1/2}(\text{curl}_\Gamma,\Gamma)$ and
  $H^{-1/2}(\Gamma)$ can therefore be identified with the quotient spaces
  $\mathcal{T}(\mathbf{grad})$, $\mathcal{T}(\mathbf{curl})$ and
  $\mathcal{T}(\text{div})$, respectively, as we have already observed in
  \eqref{eq:drqs}. Under these identifications, the bounded inverse theorem guarantees
  that the quotient spaces are equipped with equivalent norms.  Moreover, due to the
  Lipschitz regularity of $\Gamma$ and Sobolev extension theorems, the definitions of
  $\mathcal{T}(\mathbf{grad})$, $\mathcal{T}(\mathbf{curl})$ and $\mathcal{T}(\text{div})$
  are intrinsic, in the sense that the quotient spaces
  $H^1(\mathbb{R}^3\backslash\overline{\Omega})/\mr{H}^1(\mathbb{R}^3\backslash\overline{\Omega})$,
  $\mathbf{H}(\mathbf{curl},\mathbb{R}^3\backslash\overline{\Omega})/\mr{\mathbf{H}}(\mathbf{curl},\mathbb{R}^3\backslash\overline{\Omega})$
  and
  $\mathbf{H}(\text{div},\mathbb{R}^3\backslash\overline{\Omega})/\mr{\mathbf{H}}(\text{div},\mathbb{R}^3\backslash\overline{\Omega})$
  are also Hilbert spaces with equivalent norms \cite{MR3101780}.
\end{example}

\begin{lem}\label{lem: equation for perp space}
  Under Assumption \ref{Assumption trace}, if $\x_{\perp}\in \dom(\rA_k)^{\perp}$, then $\A_k\x_{\perp}\in \dom(\AT_k)$ and
  \begin{equation}\label{lem eq ATAJid}
    (\AT_k\A_k + \id)\,\x_{\perp} = 0.
  \end{equation}
\end{lem}
\begin{proof}
  Suppose that $\x_{\perp}\in \dom(\rA_k)^{\perp}$. Since $\rA_k\subset\A_k$, we have by definition that
  \begin{align}
    \begin{split}
      0 &= (\x_{\perp},\z_{\circ})_{\dom(\A_k)} = (\x_{\perp},\z_{\circ})_{\W_k} + (\A_k\x_{\perp},\A_k\z_{\circ})_{\W_{k+1}}\\
      &= (\x_{\perp},\z_{\circ})_{\W_k} + (\A_k\x_{\perp},\rA_k\z_{\circ})_{\W_{k+1}}
    \end{split}
  \end{align}
  for all $\z_{\circ}\in \dom(\rA_k)$, which means
  \begin{align}\label{eq: ATA + id}
    (\A_k\x_{\perp},\rA_k\z_{\circ})_{\W_k} = - (\x_{\perp},\z_{\circ})_{\W_k} && \forall\z_{\circ}\in \dom(\rA_k).
  \end{align}
  So by setting $C_{\x_{\perp}}:=\norm{\x_{\perp}}_{\W_k}$, we conclude from the estimate
  \begin{align}
    \vert(\A_k\x_{\perp},\rA_k\z_{\circ})_{\W_{k+1}}\vert= \vert(\x_{\perp},\z_{\circ})_{\W_k}\vert\leq \norm{\x_{\perp}}_{\W_k}\norm{\z_{\circ}}_{\W_k}= C_{\x_{\perp}}\norm{\z_{\circ}}_{\W_k} &&\forall\z_{\circ}\in \dom(\rA_k),
  \end{align}
  that $\A_k\x_{\perp}\in \dom(\rA_k^*)=\dom(\AT_k)$. Then as in \eqref{def: hilbert space adjoint}, the identity \eqref{lem eq ATAJid} follows from \eqref{eq: ATA + id}.
\end{proof}
\begin{cor}\label{restriction of A is isometry}
  Under Assumption \ref{Assumption trace}, the linear map $\A_k: \dom(\rA_k)^{\perp}\to \dom(\AT_k)$ is an isometry.
\end{cor}
\begin{proof}
  Suppose that $\x_{\perp}\in \dom(\rA_k)^{\perp}$. Then, by \Cref{lem: equation for perp space},
  \begin{equation}
    \norm{\A_k\x_{\perp}}_{\dom(\AT_k)}^2=\norm{\A_k\x_{\perp}}_{\W_{k+1}}^2 + \norm{\AT_k\A_k\x_{\perp}}_{\W_{k}}^2 =\norm{\A_k\x_{\perp}}_{\W_{k+1}}^2 + \norm{\x_{\perp}}_{\W_{k}}^2 = \norm{\x_{\perp}}_{\dom(\A_k)}^2.
  \end{equation}
\end{proof}

\thmbox{%
\begin{theo}
  \label{theo:trchar1}
  Under Assumption \ref{Assumption trace}, the linear map
  \begin{equation}\label{eq: def of It iso}
    \cred{\mathsf{I}_{k}^t}: 
    \left\{
      \arraycolsep=1.4pt\def\arraystretch{1.2}
      \begin{array}{rcl}
	\tracespace&\to&\range(\traceA)\\
	{[}\x{]}&\mapsto&\traceA\x
      \end{array}
    \right.
  \end{equation}
  is a well-defined isometric isomorphism.
\end{theo}}

\begin{proof}
  Since $\dom(\rA_k)=\null(\trace)$ by \Cref{lem: kernel of trace}, notice that $\mathsf{I}_{k}^t:\tracespace\to \range(\A_k)$ is simply the well-defined induced isomorphism of modules involved in the commutative diagram
  \begin{equation*}
    \begin{tikzcd}
      \dom(\A_k) \arrow[rr, two heads, electricindigo, "\traceA"] \arrow[rd, two heads, orange, "\bm{\pi}^t_{\A_k}"'] &  & \range(\traceA) \\
      & \dom(\A_k)/\null(\traceA)=\tracespace \arrow[ru, hookrightarrow, two heads, dotted, electricindigo, "\mathsf{I}_{\A_k}^t"'] & 
    \end{tikzcd}
  \end{equation*}
  provided by the first isomorphism theorem \cite[Chap. 10.2, Thm. 4]{MR2286236}. It only remains to show that it is an isometry. 

  Let $\x\in \dom(\A_k)$. By \Cref{lem: kernel of trace},
  \begin{equation}\label{eq: iso Ik inproof equation}
    \norm{\,\mathsf{I}_{k}^t[\x]\,}_{\dom(\AT_k)'}=\norm{\traceA\x}_{\dom(\AT_k)'}
    =\norm{\traceA(\x_{\perp}+\x_{\circ})}_{\dom(\AT_k)'}=\norm{\traceA\x_{\perp}}_{\dom(\AT_k)'}.
  \end{equation}
  Using that $\AT_k\A_k\x_{\perp}=-\x_{\perp}$ by \Cref{lem: equation for perp space}, we can choose $\y=\A_k\x_{\perp}\in D(\AT_k)$ to obtain
  \begin{align}
    \begin{split}
      \norm{\traceA\x_{\perp}}_{\dom(\AT_k)'}
      &= \sup_{0\neq\y\in D(\AT_k)}\frac{\vert\langle \traceA\x_{\perp},\y \rangle\vert}{\norm{\y}_{\dom(\AT_k)}}
      \geq \frac{\vert\langle \traceA\x_{\perp},\A_k\x_{\perp} \rangle\vert}{\norm{\A_k\x_{\perp}}_{\dom(\AT_k)}}\\
      &= \frac{\vert(\A_k\x_{\perp},\A_k\x_{\perp})_{\W_{k+1}} - (\x_{\perp},\AT_k\A_k\x_{\perp})_{\W_k}\vert}{\norm{\A_k\x_{\perp}}_{\dom(\AT_k)}}
      =\frac{\norm{\x_{\perp}}_{D(\A_k)}^2}{\norm{\A_k\x_{\perp}}_{\dom(\AT_k)}}.
    \end{split}
  \end{align}
  Recalling that $\norm{\A_k\x_{\perp}}_{\dom(\AT_k)} = \norm{\x_{\perp}}_{\dom(\A_k)}$ by \Cref{restriction of A is isometry}, we arrive at the inequality
  \begin{equation}
    \norm{\traceA\x_{\perp}}_{\dom(\AT_k)'} \geq \frac{\norm{\x_{\perp}}_{D(\A_k)}^2}{\norm{\x_{\perp}}_{D(\A_k)}}
    = \norm{\x_{\perp}}_{D(\A_k)}.
  \end{equation}
  Therefore, on the one hand, $\norm{\,\mathsf{I}_{k}^t[\x]\,}_{\dom(\AT_k)'}\geq \norm{\x_{\perp}}_{\dom(\A_k)}=\norm{[\x]}_{\tracespace}$ by \eqref{eq: projection norm is norm of quotient}. 

  On the other hand, inserting \eqref{eq: norm of weak trace is one} in \eqref{eq: iso Ik inproof equation} leads to the estimate
  \begin{equation}
    \norm{\,\mathsf{I}_{k}^t[\x]\,}_{D(\AT_k)'} =\norm{\traceA\x_{\perp}}_{\dom(\AT_k)} \leq \norm{\traceA}\norm{\x_{\perp}}_{\dom(\A_k)}= \norm{\x_{\perp}}_{\dom(\A_k)} = \norm{[\x]}_{\tracespace},
  \end{equation}
  which concludes the proof.
\end{proof}

It is natural to think of a trace operator as a bounded linear operator from a domain to a trace space. Therefore, based on the identification provided by \Cref{theo:trchar1}, we introduce the following perspective: in the setting of \Cref{def:trace1}, we call \cred{\emph{quotient trace}} the canonical projection
\begin{equation}\label{eq: quotient map as trace t}
  \cred{\bm{\pi}^t_{k}}: \left\{
    \arraycolsep=1.4pt\def\arraystretch{1.2}
    \begin{array}{rcl}
      \dom(\A_k) &\rightarrow &\tracespace\\
      \x &\mapsto &{[}\x{]} 
    \end{array}
  \right. .
\end{equation}

Notice that because $\mathsf{I}_{k}^t$ is an isomorphism, it follows from $\mathsf{I}^t_{k}(\mathsf{I}_{k}^t)^{-1}\traceA\x=\traceA \x = \mathsf{I}_{k}^t[\x]$ that
\begin{equation}\label{eq: pi in terms of I inverse}
  \bm{\pi}^t_{k}\x=(\mathsf{I}_{k}^t)^{-1}\traceA\x.
\end{equation}

\subsection{Riesz representatives}
\label{sec: Riesz representatives}

Let $\mathsf{R}_{\scaleto{\dom(\AT_k)}{7pt}}: \dom(\AT_k)\rightarrow \dom(\AT_k)'$ be the
Riesz isomorphism defined by
$\mathsf{R}_{\scaleto{\dom(\AT_k)}{7pt}}\y=(\y,\cdot)_{\dom(\AT_k)}$ for all
$\y\in\dom(\AT_k)$, cf. \cite[Thm. 5.5]{MR2759829}. Notice that in the first part of the
proof of \Cref{prop: range is annihilator}, we have shown that the following result holds
with $\AT_k\mathsf{R}_{\scaleto{\dom(\AT_k)}{7pt}}^{-1}\bm{\phi}\in \dom(\A_k)$.
\begin{lem}\label{lem: riesz rep in domains}
  Under Assumption \ref{Assumption trace}, if $\bm{\phi}\in \dom(\A_k^* )^{\circ}$, then $\AT_k\mathsf{R}_{\scaleto{\dom(\AT_k)}{7pt}}^{-1}\bm{\phi}\in \dom(\rA_k)^{\perp}$ with
  \begin{align}
    &&(\A_k\AT_k+\id)\,\mathsf{R}_{\scaleto{\dom(\AT_k)}{7pt}}^{-1}\bm{\phi} = 0
    &&\text{and}
    && \trace_{\A_k}(\AT_k\mathsf{R}_{\scaleto{\dom(\AT_k)}{7pt}}^{-1}\bm{\phi})=-\bm{\phi}.
  \end{align}
\end{lem}
\begin{proof}
  It only remains to show that in particular $\AT_k\mathsf{R}_{\scaleto{\dom(\AT_k)}{7pt}}^{-1}\bm{\phi}\in D(\rA_k)^{\perp}$. Recall that $\AT_k:=\rA_k^*$. Since $\A_k^*\subset\AT_k$, we find, using $(\A_k\AT_k+\id)\,\mathsf{R}_{\scaleto{\dom(\AT_k)}{7pt}}^{-1}\bm{\phi} = 0$, that for all $\x_{\circ}\in \dom(\rA_k)$, 
  \begin{align}
    \begin{split}
      (\AT_k\mathsf{R}_{\scaleto{\dom(\AT_k)}{7pt}}^{-1}\bm{\phi},\x_{\circ})_{\dom(\A_k)}&= (\AT_k\mathsf{R}_{\scaleto{\dom(\AT_k)}{7pt}}^{-1}\bm{\phi},\x_{\circ})_{\W_k} +(\A_k\AT_k\mathsf{R}_{\scaleto{\dom(\AT_k)}{7pt}}^{-1}\bm{\phi},\A_k\x_{\circ})_{\W_{k+1}}\\
      &= (\mathsf{R}_{\scaleto{\dom(\AT_k)}{7pt}}^{-1}\bm{\phi},\rA_k\x_{\circ})_{\W_{k+1}} -(\mathsf{R}_{\scaleto{\dom(\AT_k)}{7pt}}^{-1}\bm{\phi},\rA_k\x_{\circ})_{\W_{k+1}}=0.
    \end{split}
  \end{align}
\end{proof}
Applying $(\mathsf{I}^t_{k})^{-1}$ on both sides of the second identity in \Cref{lem: riesz rep in domains}, we find using \eqref{eq: pi in terms of I inverse} a slightly more explicit expression of the inverse $(\mathsf{I}^t_{k})^{-1}$.
\begin{lem}\label{lem: inverses of Ik}
  Under Assumption \ref{Assumption trace}, we have
  \begin{align}
    (\mathsf{I}_{k}^t)^{-1}\bm{\phi}= -\bm{\pi}^t_{\A_k} (\AT_k\mathsf{R}_{\scaleto{\dom(\AT_k)}{7pt}}^{-1}\bm{\phi}) &&\forall\bm{\phi}\in D(\A_k^*)^{\circ}=\range(\traceA).\label{eq: inverse of I_k}
  \end{align}
\end{lem}
\begin{rem}
  The operators $\AT_k\mathsf{R}_{\scaleto{\dom(\AT_k)}{7pt}}^{-1}:\range(\traceA)\rightarrow\dom(\rA_k)^{\perp}\subset\dom(\A_k)$ could be called $\dom(\A_k)$-harmonic extension operators.
\end{rem}

In summary, we have shown so far in \Cref{sec: Basic traces} that the following diagram is commutative:
\begin{equation*}
  \begin{tikzcd}[row sep=huge,column sep=huge]
    & &  \dom(\A_k^*)^{\circ}=\range(\traceA)  \arrow[dl, rightarrow, "-\AT_k\mathsf{R}_{\scaleto{\dom(\AT_k)}{7pt}}^{-1}\quad" left, shift right = 1]\arrow[dd, hookrightarrow, two heads, dashed, electricindigo, "(\mathsf{I}^t_{k})^{-1}\,\," left, shift right = 1] \\
    \dom(\A_k) \arrow[r, two heads, blue, "\mathsf{P}^t_{k}" description, shift right= 2.2] \arrow[rru,two heads, electricindigo, "\traceA" description, bend left=10, shift left = 3] \arrow[rrd,two heads, orange, "\bm{\pi}^t_{k}" description, bend right=10, shift right=3]   &  \arrow[l, hookrightarrow, dashed, "\imath^t_{k}" description, shift right = 2.2] \dom(\rA_k)^{\perp} \arrow[ur,hookrightarrow, two heads, dashed, "\traceA" description, shift right = 2.5]  \arrow[dr, hookrightarrow, two heads, orange, dashed, "\bm{\pi}^t_{k}" description, shift right= 2.5]    \\
    &  & \arrow[ul, hookrightarrow, two heads, blue, dashed, "\mathsf{G}^t_{k}" description, shift right = 2.5]  \tracespace=\dom(\A_k)/\dom(\rA_k) 
    \arrow[uu, hookrightarrow, two heads, dashed, electricindigo, "\,\,\mathsf{I}^t_{k}" right, shift right =2.5] 
  \end{tikzcd}
\end{equation*} 

\section{Duality}
\label{sec: Duality}

In this section, we maintain the setting of Assumption \ref{Assumption trace}, and we
focus on the following snippet of the dual Hilbert complex (cf.~Sections \ref{sec: Hilbert
  complexes} and \ref{sec: Baisc setting}):
\[
  \tikz[
  overlay]{
    \filldraw[fill=black!10,draw=black!10] (5.65,1.1) rectangle (11.9,-1.2);
  }
  \def\arrowlength{6ex}
  \def\arrowdistance{.4}
  \begin{tikzcd}[column sep=\arrowlength, row sep=0.05cm]
    \cdots 
    \arrow[r, leftarrow, shift right=\arrowdistance, "\AT_{k-2}"']
    & 
    \dom(\AT_{k-2}) \subset\W_{k-1}
    \ar[r, leftarrow, shift right=\arrowdistance, "\AT_{k-1}"']
    & 
    \dom(\AT_{k-1}) \subset\W_k
    \arrow[r, leftarrow, shift right=\arrowdistance, "\AT_{k}"']
    & 
    \dom(\AT_{k}) \subset \W_{k+1}
    \arrow[r, leftarrow, shift right=\arrowdistance, "\AT_{k+1}"']
    &
    \cdots \\
    \,
    & 
    \cup\qquad\quad
    & 
    \cup\qquad\quad
    & 
    \cup\qquad\quad
    &
    \, \\
    \cdots 
    \arrow[r, leftarrow, shift right=\arrowdistance, "\A^*_{k-2}"']
    & 
    \dom(A_{k-2}^*) \subset\W_{k-1}
    \ar[r, leftarrow, shift right=\arrowdistance, "\A_{k-1}^{*}"']
    & 
    \dom(A_{k-1}^*) \subset\W_k
    \arrow[r, leftarrow, shift right=\arrowdistance, "\A_{k}^{*}"']
    & 
    \dom(A_{k}^*) \subset \W_{k+1}
    \arrow[r, leftarrow, shift right=\arrowdistance, "\A^*_{k+1}"']
    &
    \cdots
  \end{tikzcd}
\]

Recall the simple though important observation that because $(\AT_k)^*=\rA_k^{**}=\rA_k$, then we have
$\rA_k\subset\A_k\iff\A_k^*\subset\AT_k$. Given two operators $\A_k:\dom(\A_k)\subset\W_k\rightarrow\W_{k+1}$ and $\rA_k:\dom(\rA_k)\subset\W_k\rightarrow\W_{k+1}$ satisfying Assumption \ref{Assumption trace}, the Hilbert space adjoints $\AT_k:\dom(\AT_k)\subset \W_{k+1}\rightarrow\W_k$ and $\A_k^*:\dom(\A_k^*)\subset \W_{k+1}\rightarrow\W_k$ thus also satisfy Assumption \ref{Assumption trace}, but with the roles of $\W_k$ and $\W_{k+1}$ swapped. Indeed, both $\AT_k$ and $\A_k^*$ are densely defined and closed unbounded linear operators between the Hilbert spaces and $\AT_k$ is an extension of $\A_k^*$, i.e. $\dom(\A_k^*)\subset\dom(\AT_k)$ and $\A_k^*\y_*=\AT_k\y_*$ for all $\y_*\in \dom(\A_k^*)$. 

In \Cref{sec: Dual traces}, the \emph{dual Hilbert trace} $\dualtrace_{k}$ will be nothing more than the primal Hilbert trace from \Cref{def: trace enviro} but associated with the pair of operators $\AT_k$ and $\A_k^*$. Nevertheless, we state its properties for completeness and to set up notation, because it will be used for the important duality results of \Cref{sec: Dual system}.

\subsection{Dual traces}
\label{sec: Dual traces}

As before, it follows from \eqref{eq: trace continuity estimate} that the following
operator is well-defined.
\medskip

\defbox{
\begin{defi}
  \label{def:trace2}
  Under Assumption \ref{Assumption trace}, we call \cred{\emph{dual Hilbert trace}} the bounded operator 
  \begin{equation}
    \cred{\dualtrace_{k}:\dom(\AT_k)\rightarrow \dom(\A_k)'},
  \end{equation} defined  for all $\y\in \dom(\AT_k)$ and $\x\in \dom(\A_k)$ by
  \begin{align}\label{eq: def of n trace}
    \langle\dualtrace_{k}\y,\x\rangle_{\dom(\AT_k)'}:= (\AT_k\y,\x)_{\W_k} - (\y,\A_k\x)_{\W_{k+1}}.
  \end{align}
\end{defi}}

As in \eqref{eq: norm of weak trace is one}, we have $\norm{\dualtrace_{k}}=1$, where
$\norm{\cdot}$ is the operator norm. Note that for all $\x\in \dom(\A_k)$ and $\y\in \dom(\AT_k)$, 
\begin{equation}
  \label{eq: weak trace t is minus weak trace n}
  \langle \trace_{k}\x,\y\rangle_{\dom(\A_k)'} = - \langle\x,\dualtrace_{k}\y\rangle_{\dom(\AT_k)'}.
\end{equation}
In other words
\begin{align}\label{eq: weak traces are dual operators}
  (\trace_{k})'=-\dualtrace_{k} &&\text{and} &&(\dualtrace_{k})'=-\trace_{k}.
\end{align}

The results of \Cref{sec: Basic traces} can be mirrored by interchanging the roles of $\A_k$ and $\AT_k$ (and the roles of $\rA_k$ and $\A_k^*$ accordingly). We translate a few of them without proof.

\begin{prop}[cf. \Cref{lem: kernel of trace}]\label{lem: kernel of n trace}
  Under Assumption \ref{Assumption trace}, we have
  \begin{equation}
    \null(\dualtrace_{k}) = \dom(\A_k^*).
  \end{equation}
\end{prop}
The next proposition involves the annihilator of $\dom(\rA_k)$ in $\dom(\A_k)'$:
\begin{equation}
  \dom(\rA_k)^{\circ}:=\{\,\bm{\phi}\in \dom(\A_k)'\,\,\vert\,\,\langle\bm{\phi},\x_{\circ}\rangle =0,\,\forall\x_{\circ}\in \dom(\rA_k)\,\}.
\end{equation} 
\begin{prop}[cf. \Cref{prop: range is annihilator}]\label{prop: range of n is annihilator}
  Under Assumption \ref{Assumption trace}, we have
  \begin{equation}
    \range(\dualtrace_{k}) = \dom(\rA_k)^{\circ}.
  \end{equation}
\end{prop}

\defbox{%
\begin{defi}[cf. \Cref{def:tracespace1}]
  \label{def:tracespace2}
  We call \cred{\emph{dual trace spaces}} the quotient spaces
  \begin{equation}\label{eq: def of dual trace space}
    \cred{\dualtracespace:=\dom(\AT_k)/\dom(\A_k^*)},
  \end{equation}
  equipped with the quotient norm 
  \begin{align}\label{quotient norm dual trace space}
    \norm{[\y]}_{\dualtracespace}:=\inf_{\z_*\in \dom(\A_k^*)}\norm{\y-\z_{*}}_{\dom(\AT_k)} && \forall\y\in \dom(\AT_k).
  \end{align}
\end{defi}}

\begin{rem}
  Just as in \Cref{rem: trace space is quotient by kernel}, notice that due to \Cref{lem: kernel of n trace},
  \begin{equation}
    \dualtracespace =\dom(\AT_k)/\null(\dualtrace_{\A_k}).
  \end{equation}
\end{rem}

In \eqref{quotient norm dual trace space}, we used square brackets to denote the equivalence class in $\dualtracespace$ of $\y\in \dom(\AT_k)$, i.e. $[\y]=\{\y + \z_* \,\vert\, \z_*\in \dom(\A_k^*)\}$. We will write $\bm{\pi}^n_{k} : D(\AT_k)\to\dualtracespace$ for the associated canonical projection (quotient map), i.e. $\bm{\pi}^n_{k}(\y) = [\y]$.  Then, as previously detailed in \Cref{sec: Trace spaces}, there exists a bounded orthogonal projection 
$
\mathsf{P}^n_{k}: \dom(\AT_k)\to \dom(\A_k^*)^{\perp}
$
onto the complement space
\begin{equation}
  \dom(\A_k^*)^{\perp} :=\left\{\,\y\in \dom(\AT_k)\,\,\vert\,\,(\y,\z_*)_{\dom(\AT_k)}=0,\,\forall\z_*\in \dom(\A_k^*)\right\}
\end{equation}
satisfying
$\norm{\mathsf{P}^n_{k}\y}_{\dom(\AT_k)} =\norm{[\y]}_{\dualtracespace}$ for all $\y\in \dom(\AT_k)$. We denote by $\imath^n_{k}:\dom(\A_k^*)^{\perp}\hookrightarrow \dom(\AT_k)$ the canonical inclusion maps. 

The induced operator $\mathsf{G}^n_{k}:\dualtracespace\rightarrow \dom(\A_k^*)^{\perp}$ involved in the commutative diagram
\begin{equation}\label{commutative diagram Pn and Gn}
  \begin{tikzcd}
    \dom(\AT_k) \arrow[rr, two heads, darkspringgreen, "\mathsf{P}^n_{k}"] \arrow[rd, two heads, deepcarminepink, "\bm{\pi}^n_{k}"'] &  & \dom(\A_k^*)^{\perp} \\
    & \dom(\AT_k)/\null(\mathsf{P}^n_{k})=\dualtracespace \arrow[ru, leftrightarrow, darkspringgreen, dotted, "\mathsf{G}^n_{k}"'] & 
  \end{tikzcd}
\end{equation}
is an isometric isomorphism. Accordingly, any $\y\in \dom(\AT_k)$ can be uniquely decomposed as
\begin{align}\label{decomposition of D(AT)}
  \y = \mathsf{P}^n_{k}\y + \y_*, && \y_*:=(\id-\mathsf{P}^n_{k})\y\in \null(\mathsf{P}^n_{k})= \dom(\A_k^*).
\end{align}

\begin{example}{Classical dual traces}
  Using \eqref{eq: weak traces are dual operators}, we find for the de Rham complex that
  \begin{align}
    \dualtrace_{\scriptscriptstyle\mathbf{grad}} = -\gamma'\circ\gamma_n, &&
    \dualtrace_{\scriptscriptstyle\mathbf{curl}} = \gamma_{t}'\circ\gamma_{t},&&
    \dualtrace_{\scriptscriptstyle\text{div}} = -\gamma_n'\circ\gamma.
  \end{align}
  
  Recalling \eqref{eq: ker Dirichlet trace} to \eqref{eq: ket n trace}, we see from the table of the 3D de Rham setting \ref{spe: Boundary conditions} that based on \Cref{lem: kernel of n trace},
  \begin{align}
    \null(\dualtrace_{\scriptscriptstyle\mathbf{grad}})=\null(\gamma_n), &&\null(\dualtrace_{\scriptscriptstyle\mathbf{curl}})=\null(\gamma_t), &&\null(\dualtrace_{\text{div}})&=\null(\gamma).
  \end{align}

  The trace spaces provided by \Cref{def:tracespace2} in this setting are
  \begin{subequations}
    \begin{align}
      \mathcal{T}(\mathbf{grad}^{\top})&=\mathcal{T}(\text{div})=  \Hdiv/\Hdivnot,\\
      \mathcal{T}(\mathbf{curl}^{\top})&=\mathcal{T}(\mathbf{curl})=  \Hcurl/\Hcurlnot,\\
      \mathcal{T}(\text{div}^{\top})&=\mathcal{T}(\mathbf{grad})=  \Hone/\Honenot.\
    \end{align}
  \end{subequations}

  Notice that from \eqref{T0 and T2 are dual}, we also 
  \begin{align}\label{dual traces equality}
    (\trace_{\scriptscriptstyle\text{div}})'=\trace_{\scriptscriptstyle\mathbf{grad}} = -\dualtrace_{\scriptscriptstyle\text{div}}=-(\dualtrace_{\scriptscriptstyle\mathbf{grad}})'&&\text{and} && (\trace_{\scriptscriptstyle\mathbf{grad}})'=\trace_{\scriptscriptstyle\text{div}} = -\dualtrace_{\scriptscriptstyle\mathbf{grad}} = - (\dualtrace_{\scriptscriptstyle\text{div}})'.
  \end{align}
  Moreover, we see that the skew-symmetry behind \eqref{skew-symmetric pairing} is rooted in the fact that the identity $\A_1=\mathbf{curl}=\AT_1$ leads to skew-symmetry of the pairing
  \begin{equation}
    (\x,\y)\mapsto (\mathbf{curl}\,\x,\y)_{\mathbf{L}^2(\Omega)} - (\x,\mathbf{curl}\,\y)_{\mathbf{L}^2(\Omega)}.
  \end{equation}
  This is reflected in the observation that $(\gamma_t'\circ\gamma_t)'=(\trace_{1})'=-\dualtrace_{1}=-\gamma_t'\circ\gamma_t$, which indeed occurs when duality is taken with respect to the skew-symmetric pairing \eqref{skew-symmetric pairing}.
\end{example}

\begin{theo}[cf. \Cref{theo:trchar1}]
  \label{theo:trchar2}
  Under Assumption \ref{Assumption trace}, the linear map
  \begin{equation}\label{eq: def of In iso}
    \cred{\mathsf{I}^n_{k}}:\left\{
      \arraycolsep=1.4pt\def\arraystretch{1.2}
      \begin{array}{rcl}
        \dualtracespace &\to 
        &\range(\dualtrace)\\
        {[}\y{]} &\mapsto &\dualtrace_{\A_k}\y
      \end{array}
    \right.
  \end{equation}
  is a well-defined isometric isomorphism.
\end{theo}

We call \emph{dual quotient trace} the canonical projection (cf. \eqref{eq: quotient map as trace t})
\begin{equation}
  \cred{\bm{\pi}^n_{k}}: \left\{
    \arraycolsep=1.4pt\def\arraystretch{1.2}
    \begin{array}{rcl}
      \dom(\AT_k) &\rightarrow &\dualtracespace\\
      \y &\mapsto &{[}\y{]}
    \end{array}
  \right. .
\end{equation}

Similarly as before, notice that (cf. \eqref{eq: pi in terms of I inverse})
\begin{equation}\label{eq: rho in terms of J inverse}
  \bm{\pi}^n_{k}\y=(\mathsf{I}^n_{k})^{-1}\dualtrace_{k}\y,
\end{equation}
and the following diagram is commutative:

\begin{equation*}
  \begin{tikzcd}[row sep=huge,column sep=huge]
    & &  \dom(\rA_k)^{\circ}  \arrow[dl, rightarrow, "-\A_k\mathsf{R}_{\scaleto{\dom(\A_k)}{7pt}}^{-1}\quad" left, shift right = 1]\arrow[dd, hookrightarrow, two heads, dashed, "(\mathsf{I}^n_{k})^{-1}\,\," left, shift right = 1]  \\
    \dom(\AT_k) \arrow[r, two heads, darkspringgreen, "\mathsf{P}^n_{k}" description, shift right= 3] \arrow[rru,two heads, "\dualtrace_{k}" description, bend left=10, shift left = 3] \arrow[rrd,two heads, deepcarminepink, "\bm{\pi}^n_{k}" description, bend right=10, shift right=3]   &  \arrow[l, hookrightarrow, dashed, "\imath_{k}^n" description, shift right = 2.2] \dom(\A_k^*)^{\perp} \arrow[ur,hookrightarrow, two heads, dashed, "\dualtrace_{k}" description, shift right = 3.5]  \arrow[dr, hookrightarrow, two heads, deepcarminepink, dashed, "\bm{\pi}^n_{k}" description, shift right= 4.5]    \\
    &  & \arrow[ul, hookrightarrow, two heads, darkspringgreen, dashed, "\mathsf{G}^n_{k}" description, shift right = 2.3]  \dualtracespace 	\arrow[uu, hookrightarrow, two heads, dashed, "\,\,\mathsf{I}^n_{k}" right, shift right =2.5] 
  \end{tikzcd}
\end{equation*} 

\subsection{Duality of trace spaces}
\label{sec: Dual system}

In this section, we show that the trace spaces $\tracespace$ and $\dualtracespace$ can be
put in duality through an isometry. In fact, this follows immediately from a classical
result in functional analysis. Indeed, according to \cite[Thm. 4.9]{MR1157815}, we have
the isometric isomorphisms
\begin{align}
  \dom(\A^*_k)^{\circ}\cong \left(\dom(\AT_k)/\dom(\A^*_k)\right)'
  &&\text{and} &&
  \dom(\rA_k)^{\circ}\cong \left(\dom(\A_k)/\dom(\rA_k)\right)'.
\end{align}
Combining these results with propositions \ref{prop: range is annihilator} and \ref{prop:
  range of n is annihilator}, along with theorems \ref{theo:trchar1} and
\ref{theo:trchar2},
\begin{subequations}
  \begin{align}
    \tracespace\cong\range(\trace_{k})=
    \dom(\A^*_k)^{\circ}\cong \left(\dom(\AT_k)/\dom(\A^*_k)\right)'=(\dualtracespace)', \\
    \dualtracespace\cong\range(\dualtrace_{k})=\dom(\rA_k)^{\circ}\cong
    \left(\dom(\A_k)/\dom(\rA_k)\right)'=
    (\tracespace)'.
  \end{align}
\end{subequations}

Nevertheless, we provide a detailed proof below, not only for convenience and
completeness, but also because the exercise is illuminating. We proceed with the
definition of a continuous bilinear form on $\tracespace\times\dualtracespace$ and prove
that the associated induced linear operator is an isometry. This pairing will be at the
heart of sections \ref{sec: Compactness} and \ref{sec: Trace Hilbert Complexes}, where it
will be used to prove that Hilbert complexes affording so-called compact regular
decompositions spawn Fredholm trace Hilbert complexes.
\begin{lem}\label{def of pairing b}
  Under Assumption \ref{Assumption trace}, the bilinear form
  \begin{subequations}
    \begin{equation}
      \cred{\llangle\cdot,\cdot\rrangle_{k}:\tracespace\times\dualtracespace\rightarrow\mathbb{R}},
    \end{equation}
    defined by 
    \begin{align}
      \label{eq: def of duality pairing}
      \llangle {[}\x],[\y]\rrangle_{k}:=(\A_k\x,\y)_{\W_{k+1}} - (\x,\AT_k\y)_{\W_k} \quad
      \forall[\x]\in \tracespace,\forall[\y]\in \dualtracespace,
    \end{align}
  \end{subequations}
  is well-defined and continuous with norm $\leq1$.
\end{lem}
\begin{proof}
  Since $\llangle{[}\u],[\v]\rrangle_{k}=\langle\traceA\x,\y\rangle_{\dom(\A_k)'}$, it is well-defined thanks to \Cref{lem: kernel of trace} and \Cref{prop: range is annihilator}. By the same propositions, the orthogonal decompositions \eqref{decomposition of D(A)} and \eqref{decomposition of D(AT)} yield the estimate
  \begin{align}
    \begin{split}
      \vert\langle\traceA\x,\y\rangle_{\dom(\A_k)'}\vert&=\vert\langle\traceA\mathsf{P}^t_{k}\x,\mathsf{P}^n_{k}\y\rangle_{\dom(\A_k)'}\vert\\
      &=\vert (\A_k\mathsf{P}^t_{k}\x,\mathsf{P}^n_{k}\y)_{\W_{k+1}} - (\mathsf{P}^t_{k}\x,\AT_k\mathsf{P}^n_{k}\y)_{\W_k} \vert\\
      &\leq \|\A_k\mathsf{P}^t_{k}\x\|_{\W_{k+1}}\|\mathsf{P}^n_{k}\y\|_{\W_{k+1}} + \|\mathsf{P}^t_{k}\x\|_{\W_k}\|\AT_k\mathsf{P}^n_{k}\y\|_{\W_k}
      \\
      &\leq\, \|\mathsf{P}^t_{k}\x\|_{\dom(\A_k)}\|\mathsf{P}^n_{k}\y\|_{\dom(\AT_k)}\\
      &= \norm{[\x]}_{\tracespace}\norm{[\y]}_{\dualtracespace},
    \end{split}
  \end{align}
  showing that the bilinear form is continuous with norm $\leq1$.
\end{proof}

The next result shows in particular that $\tracespace$ and $\dualtracespace$ can be put in
duality through the bilinear form $\llangle\cdot,\cdot\rrangle_{k}$.
\medskip

\thmbox{%
\begin{theo}\label{def of B}
  Under Assumption \ref{Assumption trace}, the bounded linear operator
  \begin{equation}\label{eq: definition of K pairing}
    \mathsf{K}_{k}:\left\{
      \arraycolsep=1.4pt\def\arraystretch{1.2}
      \begin{array}{rcl}
        \tracespace &\rightarrow &\dualtracespace'\\
        {[}\x{]} &\mapsto &\llangle{[}\x{]},\cdot\rrangle_{k}
      \end{array}
    \right.
  \end{equation}
  induced by the bilinear form defined in \Cref{def of pairing b} is an isometric isomorphism.
\end{theo}}
\begin{proof}
  The key to the proof is that \eqref{eq: def of duality pairing} permits us to appeal to \Cref{theo:trchar1}.

  Notice that since $\range(\traceA)=\dom(\A_k^*)^{\circ}$, it follows from the orthogonal
  decomposition \eqref{decomposition of D(AT)} that $\mathsf{K}_{k}$ is the pullback by
  $\mathsf{G}^n_{k}$ of $\mathsf{I}^{t}_{k}$,
  i.e. $\mathsf{K}_{k}[\x]([\y])=\mathsf{I}^{t}_{k}[\x](\mathsf{G}^n_{k} [\y])$. We first
  show that it is an isomorphism.

  If $\mathsf{K}_{k}[\x]=\mathsf{K}_{k}[\z]$, then since $\mathsf{G}^n_{k}$ is an
  isomorphism onto $\dom(\A_k^*)^{\perp}$, it then follows from \Cref{prop: range is
    annihilator} and decomposition \eqref{decomposition of D(AT)} that
  $\mathsf{I}^{t}_{k}[\x](\y)=\mathsf{I}^{t}_{k}[\z](\y)$ for all $\y\in \dom(\AT_k)$. But
  $\mathsf{I}^{t}_{k}$ is also an isomorphism, so
  $\mathsf{I}^{t}_{k}[\x]=\mathsf{I}^{t}_{k}[\z]$ implies that $\x=\z$ and we conclude
  that $\mathsf{K}_{k}$ is injective.
  
  Suppose that $\bm{\phi}\in \dualtracespace'$. Then the pullback of $\bm{\phi}$ by the canonical quotient map $\bm{\pi}^n_{k}: \dom(\AT_k)\rightarrow \dualtracespace$ is a bounded linear functional on $\dom(\AT_k)$, i.e. $\bm{\phi}\circ \bm{\pi}^n_{k}\in \dom(\AT_k)'$. Indeed, this simply holds because
  \begin{align}
    \vert\bm{\phi}(\bm{\pi}^n_{k}\y)\vert = \norm{\bm{\phi}}\norm{\bm{\pi}^n_{k}\y}_{\dualtracespace} \leq \norm{\bm{\phi}}\norm{\bm{\pi}^n_{k}}\norm{\y}_{\dom(\AT_k)} &&\forall\y\in\dom(\AT_k).
  \end{align} 
  Moreover, since $\null(\bm{\pi}^n_{k})=\dom(\A_k^*)$, we find in particular that $\bm{\phi}\circ\bm{\pi}^n_{k}\in \dom(\A_k^*)^{\circ}=\range(\trace_k)$. But $\mathsf{I}^{t}_{k}$ is an isomorphism onto $\range(\traceA)$, so there exists $[\x]\in \tracespace$ such that
  $\mathsf{I}^{t}_{k}[\x]=\bm{\phi}\circ\bm{\pi}^n_{k}$. Evaluating
  \begin{equation}
    \mathsf{K}_{k}[\x] = \mathsf{I}^{t}_{k}[\x]\circ\mathsf{G}^n_{k} = \bm{\phi}\circ\bm{\pi}^n_{k}\circ\mathsf{G}^n_{k} =\bm{\phi}
  \end{equation}
  shows that $\mathsf{K}_{k}$ is surjective.
  
  We now prove that $\mathsf{K}_{k}$ is an isometry. Using similar arguments as above, we estimate
  \begin{equation}
    \norm{\mathsf{K}_{k}[\x]}=\sup_{\substack{[\y]\in \dualtracespace,\\\norm{[\y]}_{\dualtracespace}=1}}\vert\mathsf{K}_{k}{[\x]}([\y])\vert
    = \sup_{\substack{\y_{\perp}\in \dom(\A_k^*)^{\perp},\\\norm{\y_{\perp}}_{\dom(\AT_k)}=1}}\vert\mathsf{I}^t_{k}[\x](\y_{\perp})\vert
    = \norm{\mathsf{I}^t_{k}[\x]}
    =\norm{[\x]}_{\tracespace}.
  \end{equation}
\end{proof}

We have arrived at an integration by parts formula involving the traces from \Cref{sec: Hilbert traces} and \Cref{sec: Dual traces}: for all $\x\in \dom(\A_k)$ and $\y\in \dom(\AT_k)$,
\begin{equation}
  (\A_k\x,\y)_{\W_{k+1}}-(\x,\AT_k\y)_{\W_k} = \llangle\bm{\pi}^t_{k}\x,\bm{\pi}^n_{k}\y\rrangle_{\A_k}.
\end{equation}

\Cref{def of B}, in combination with \eqref{IBP div} and \eqref{IBP curl}, reveals the abstract version of the duality observed for the de Rham complex in \Cref{sec: intro}.

\section{Operators on Trace Spaces}
\label{sec: Operators on trace spaces}

Starting from this section, we start exploiting more of the structure of Hilbert complexes by introducing the \emph{minimal} Hilbert complex setting required to define what we will call \emph{surface operators}. We ``zoom in'' on short snippets of \eqref{Hilbert complex Ak} and \eqref{Hilbert complex rAk} of the form
\begin{equation}\label{Hilbert complex Ak minimal}
  \tikz[
  overlay]{
    \filldraw[fill=black!10,draw=black!10] (1.7,1.2) rectangle (11.7,-1);
  }
  \def\arrowlength{6ex}
  \def\arrowdistance{.8}
  \begin{tikzcd}[column sep=\arrowlength, row sep=0.05cm]
    \cdots 
    \arrow[r, rightarrow, shift left=\arrowdistance, "\A_{k-1}"] 
    & 
    \dom(\A_{k})\subset\W_{k} 
    \ar[r, rightarrow, shift left=\arrowdistance, "\A_{k}"] 
    & 
    \dom(\A_{k+1})\subset\W_{k+1}
    \arrow[r, rightarrow, shift left=\arrowdistance, "\A_{k+1}"] 
    & 
    \dom(\A_{k+2})\subset\W_{k+2}
    \arrow[r, rightarrow, shift left=\arrowdistance, "\A_{k+2}"] 
    &
    \cdots \\
    \,
    & 
    \cup\qquad\quad
    & 
    \cup\qquad\quad
    & 
    \qquad\quad
    &
    \, \\
    \cdots 
    \arrow[r, rightarrow, shift left=\arrowdistance, "\rA_{k-1}"] 
    & 
    \dom(\rA_{k})\subset\W_{k} 
    \ar[r, rightarrow, shift left=\arrowdistance, "\rA_{k}"] 
    & 
    \dom(\rA_{k+1})\subset\W_{k+1}
    \arrow[r, rightarrow, shift left=\arrowdistance, "\rA_{k+1}"] 
    & 
    \dom(\rA_{k+2})\subset\W_{k+2}
    \arrow[r, rightarrow, shift left=\arrowdistance, "\rA_{k+2}"] 
    &
    \cdots 
  \end{tikzcd}
\end{equation}
We may call the highlighted sequences ``minimal Hilbert complexes''. The index $k$ should
be considered arbitrary but fixed in this section.


\begin{example}{Minimal Hilbert complexes}
  \label{spe: Minimal Hilbert complexes}
  Based on the 3D de Rham setting \ref{spe: L2 de Rham complex} and \ref{spe: Boundary
    conditions}, we obtain two minimal complexes such as \eqref{Hilbert complex Ak
    minimal}. For $k=0$, we have
  \begin{equation}\label{Hilbert complex Ak minimal de Rham k=0}
    \def\arrowlength{6ex}
    \def\arrowdistance{.8}
    \begin{tikzcd}[column sep=\arrowlength, row sep = tiny]
      \Hone\subset \Ltwo
      \ar[r, rightarrow, shift left=\arrowdistance, "\mathbf{grad}"] 
      & 
      \Hcurl\subset\bLtwo
      \arrow[r, rightarrow, shift left=\arrowdistance, "\mathbf{curl}"] 
      & 
      \bLtwo,\\
      \Honenot\subset\Ltwo
      \ar[r, rightarrow, shift left=\arrowdistance, "\mathbf{grad}"] 
      & 
      \Hcurlnot\subset\bLtwo
      \arrow[r, rightarrow, shift left=\arrowdistance, "\mathbf{curl}"] 
      & 
      \bLtwo.
    \end{tikzcd}
  \end{equation}
  For $k=1$, we get
  \begin{equation}\label{Hilbert complex Ak minimal de Rham k=1}
    \def\arrowlength{6ex}
    \def\arrowdistance{.8}
    \begin{tikzcd}[column sep=\arrowlength, row sep = tiny]
      \Hcurl\subset \bLtwo
      \ar[r, rightarrow, shift left=\arrowdistance, "\mathbf{curl}"] 
      & 
      \Hdiv\subset\bLtwo
      \arrow[r, rightarrow, shift left=\arrowdistance, "\text{div}"] 
      & 
      \subset\bLtwo,\\
      \Hcurlnot\subset \bLtwo
      \ar[r, rightarrow, shift left=\arrowdistance, "\mathbf{curl}"] 
      & 
      \Hdivnot\subset\bLtwo
      \arrow[r, rightarrow, shift left=\arrowdistance, "\text{div}"] 
      & 
      \bLtwo.
    \end{tikzcd}
  \end{equation}
\end{example}

\subsection{Surface operators in domains}
\label{sec: Surface operators in domains} 
Notice that due to the complex property, we have in particular that
$\range(\A_k)\subset \dom(\A_{k+1})$ and $\range(\AT_{k+1})\subset \dom(\AT_{k})$.  The
following key operators are thus well-defined.
\medskip

\defbox{\begin{defi} \label{def: weak trace space operator}
  We call \cred{\emph{surface operators}} the bounded linear maps
  \begin{subequations}
    \begin{align}
      &\cred{\D^t_{k}:=(\AT_{k+1})':\dom(\AT_{k})'\rightarrow \dom(\AT_{k+1})'},\label{def: surface operator Dt}\\
      &\cred{\D^n_{k+1}:=\A_k':\dom(\A_{k+1})'\rightarrow \dom(\A_k)'},\label{def: surface operator Dn}
    \end{align}
  \end{subequations}
  dual to $\AT_{k+1}:\dom(\AT_{k+1})\rightarrow \dom(\AT_{k})$ and $\A_k:\dom(\A_k)\rightarrow \dom(\A_{k+1})$, respectively. Equivalently,
  \begin{subequations}
    \begin{align}
      \label{D51a}
      \langle\D^t_k\bm{\phi},\z\rangle_{\dom(\AT_{k+1})'} &= \langle
      \bm{\phi},\AT_{k+1}\z\rangle_{\dom(\AT_{k})'},
      &&\forall\bm{\phi}\in \dom(\AT_{k})',\forall\z\in \dom(\AT_{k+1})\subset\W_{k+2},\\
      \label{D51b}
      \langle\D^n_{k+1}\bm{\psi},\y\rangle_{D(\A_k)'} &= \langle
      \bm{\psi},\A_k\x\rangle_{\dom(\A_{k+1})'},
      &&\forall\bm{\psi}\in \dom(\A_{k+1})',\forall\x\in \dom(\A_k)\subset\W_k.
    \end{align}
  \end{subequations}
\end{defi}}
\begin{rem}
  Recall the distinction made in \Cref{sec: Operators on Hilbert spaces} between the
  notation for bounded and unbounded linear operators. We point out that in \Cref{def:
    weak trace space operator}, the operators
  $\AT_{k+1}:\dom(\AT_{k+1})\rightarrow \dom(\AT_{k})$ and
  $\A_k:\dom(\A_k)\rightarrow \dom(\A_{k+1})$ are \emph{bounded}.
\end{rem}

\begin{rem}
  The name `surface operators' was chosen by analogy with standard surface operators on
  the boundary of a domain, despite the fact that there is no boundary involved in the
  above definition. The relation between \Cref{def: weak trace space operator} and
  standard surface operators is made more explicit in the 3D de Rham settings \ref{spe:
    Surface operators in domains} and \ref{Commutative relations}.
\end{rem}

\begin{example}{Surface operators in domains}
  \label{spe: Surface operators in domains}
  In the 3D de Rham setting \ref{spe: Minimal Hilbert complexes}, we find the surface operators
  \begin{subequations}
    \begin{align}
      \D^t_0&:=\mathbf{curl}':\Hdiv'\rightarrow\Hcurl',\\
      \D^t_{1}&:=(-\mathbf{grad})':\Hcurl'\rightarrow\widetilde{H}^{-1}(\Omega),
    \end{align}
  \end{subequations}
  dual to the $\emph{bounded}$ operators 
  \begin{align}
    \mathbf{curl}:\Hcurl\rightarrow\Hdiv &&\text{and} &&-\mathbf{grad}:\Hone\rightarrow\Hcurl,
  \end{align}
  where we have written $\widetilde{H}^{-1}(\Omega):=\Hone'$. In other words,
  \begin{subequations}
    \begin{align}
      \langle\D^t_0\bm{\phi},\v\rangle_{\Hcurl'} &= \langle \bm{\phi},\mathbf{curl}\,\v\rangle_{\Hdiv'}, &&\forall\bm{\phi}\in \Hdiv'\forall\v\in \Hcurl,\\
      \langle\D^t_{1}\bm{\phi},\u\rangle_{\widetilde{H}^{-1}(\Omega)} &= \langle \bm{\psi},-\mathbf{grad} u\rangle_{\Hcurl'} &&\forall\bm{\phi}\in \Hcurl',\forall u \in \Hone.
    \end{align}
  \end{subequations}

  In the adjoint perspective, the bounded linear operators
  \begin{subequations}
    \begin{align}
      \D^n_1&:=\mathbf{grad}':\Hcurl'\rightarrow\widetilde{H}^{-1}(\Omega)\\
      \D^n_2&:=\mathbf{curl}':\Hdiv'\rightarrow\Hcurl'
    \end{align}
  \end{subequations}
  are dual to the bounded linear operators 
  \begin{align}
    \mathbf{grad}:\Hone\rightarrow\Hcurl &&\text{and} &&\mathbf{curl}:\Hcurl\rightarrow\Hdiv.
  \end{align}
  That is,
  \begin{subequations}
    \begin{align}
      \langle\D^n_{1}\bm{\phi},u\rangle_{\widetilde{H}^{-1}(\Omega)} &= \langle \bm{\phi},\mathbf{grad} u\rangle_{\Hcurl'} &&\forall\bm{\phi}\in \Hcurl',\forall u\in \Hone,\\
      \langle\D^n_2\bm{\psi},\v\rangle_{\Hcurl'} &= \langle \bm{\psi},\mathbf{curl}\,\v\rangle_{\Hdiv'} &&\forall\bm{\psi}\in \Hdiv',\forall\v\in \Hcurl.
    \end{align}
  \end{subequations}
\end{example}

Since
\begin{subequations}
  \begin{align}
    &\range(\A_k)\subset \dom(\A_{k+1})= \dom(\trace_{k+1}), && \range(\trace_{k})\subset \dom(\AT_{k})'=\dom(\D^t_k),\\
    &\range(\AT_{k+1})\subset \dom(\AT_{k})= \dom(\dualtrace_{k}), && \range(\dualtrace_{k+1})\subset \dom(\A_{k+1})'=\dom(\D^n_{k+1}),
  \end{align}
\end{subequations}
the linear operators
\begin{subequations}
  \begin{align}
    \D^t_k\circ\trace_{k}:\dom(\A_k)\rightarrow \dom(\AT_{k+1})',
    &&\trace_{k+1}\circ\A_k:\dom(\A_k)\rightarrow \dom(\AT_{k+1})',\\
    \D^n_{k+1}\circ\dualtrace_{k+1}:\dom(\AT_{k+1})\rightarrow \dom(\A_k)',
    &&\dualtrace_{k}\circ\AT_{k+1}:\dom(\AT_{k+1})\rightarrow \dom(\A_k)',
  \end{align} 
\end{subequations}
are also well-defined and bounded.
\medskip

\thmbox{%
\begin{lem}
  \label{lem: commutative relations}
  Assumption~\ref{Assumption trace} implies the following commuting relations:
  \begin{align}
    -\D^t_k\circ\trace_{k} = \trace_{k+1}\circ\A_k
    &&\text{and}
    && -\D^n_{k+1}\circ\dualtrace_{k+1}=\dualtrace_{k}\circ\AT_{k+1}.
  \end{align}
\end{lem}}

  \begin{proof}
    By symmetry, we need to verify only one relation. Recall that because of the complex property $\A_{k+1}\circ\A_k=0$, we also have $\AT_{k}\circ\AT_{k+1}=0$. Therefore, for all $\x\in \dom(\A_k)\subset\W_k$ and $\z\in \dom(\AT_{k+1})\subset\W_{k+2}$, we have on the one hand that
    \begin{align}
      \begin{split}
        \langle \D^t_k\trace_{k}\x,\z\rangle_{\dom(\AT_{k+1})'} &= \langle \trace_{k}\x,\AT_{k+1}\z\rangle_{\dom(\AT_{k})'}
        =(\A_k\x,\AT_{k+1}\z)_{\W_{k+1}}-(\u,\AT_{k}\AT_{k+1}\z)_{\W_k}\\
        &=(\A_k\x,\AT_{k+1}\z)_{\W_{k+1}}.
      \end{split}
    \end{align}
    On the other hand, we also evaluate
    \begin{align}
      \begin{split}
        \langle \trace_{k+1}\A_k\x,\z\rangle_{\dom(\AT_{k+1})'}
        &=(\A_{k+1}\A_k\x,\z)_{\W_{k+2}}-(\A_k\x,\AT_{k+1}\z)_{\W_{k+1}}\\
        &=-(\A_k\x,\AT_{k+1}\z)_{\W_{k+1}}.
      \end{split}
    \end{align}
  \end{proof}
  
\begin{rem}
  Consistent with \eqref{eq: weak traces are dual operators}, $(\D^t_k\circ\trace_{k})'=\D^n_{k+1}\circ\dualtrace_{k+1}$ and $\D^t_k\circ\trace_{k}=(\D^n_{k+1}\circ\dualtrace_{k+1})'$.
\end{rem}
\Cref{lem: commutative relations} states that the following diagrams commute:
\begin{equation}
  \begin{tikzcd}[row sep=large, column sep = large]
    \dom(\A_k)\arrow[r,rightarrow,"\A_k"] & \dom(\A_{k+1})& && &\dom(\AT_{k+1})\arrow[r,rightarrow,"\AT_{k+1}"] &  \dom(\AT_{k})\\
    \range(\trace_{k})\arrow[u,twoheadleftarrow, "\trace_{k}"]\arrow[r,rightarrow,"-\D^t_k"] &  \range(\trace_{k+1})\arrow[u,twoheadleftarrow, "\trace_{k+1}"] && & &\range(\dualtrace_{k+1})\arrow[u,twoheadleftarrow, "\dualtrace_{k+1}"]\arrow[r,rightarrow,"-\D^n_{k+1}"] &  \range(\dualtrace_{k})\arrow[u,twoheadleftarrow,"\dualtrace_{k}"]
  \end{tikzcd}
\end{equation}

An important consequence of this result is that
\begin{equation}\label{eq: Dt maps trace space to trace space}
  \mathsf{D}^t_k(\range(\trace_{k}))\subset\range(\trace_{k+1})=\dom(A^*_{k+1})^{\circ},
\end{equation}
an observation that is key to the introduction of trace Hilbert complexes in later sections.

\begin{example}{Commutative relations}\label{Commutative relations}
  In the 3D de Rham setting, it follows from \eqref{dual traces equality} that the four relations obtained from \Cref{lem: commutative relations} boil down to the single identity
  \begin{equation}\label{eq: applied commutative identities}
    \mathbf{grad}'\gamma_{t}'\circ\gamma_t = \gamma'\circ\gamma_n\mathbf{curl}.
  \end{equation}
  In particular, \eqref{eq: applied commutative identities} states that for all $\u\in\Hcurl$ and $v\in\Hone$,
  \begin{equation}\label{eq: int equation n curl 1}
    \int_{\Gamma}v\,\mathbf{n} \cdot\mathbf{curl}\,\u\,\dif\sigma = \int_{\Gamma}\mathbf{n}\times(\u\times\mathbf{n})\cdot (\mathbf{grad}\,v \times\mathbf{n})\,\dif\sigma.
  \end{equation}

  Recall that $\mathbf{n}\cdot\mathbf{curl}= \text{curl}_{\Gamma}\circ\gamma_t$ on $\Hcurl$, while the $L^2(\Gamma)$-dual operator $\text{curl}_{\Gamma}=\mathbf{curl}_{\Gamma}'$ is such that $\mathbf{grad}\cdot\times\mathbf{n}=\mathbf{curl}_{\Gamma}\circ\gamma$ on $\Hone$. Therefore, \eqref{eq: int equation n curl 1} expresses that
  \begin{align}\label{surface indentity 1}
    \int_{\Gamma} u\,\text{curl}_{\Gamma}\u\,\dif\sigma= \int_{\Gamma}
    \mathbf{curl}_{\Gamma} \u\cdot u \,\dif\sigma \quad \forall u\in H^{1/2}(\Gamma),\,\u\in \mathbf{H}^{-1/2}(\text{curl}_{\Gamma},\Gamma).
  \end{align}
  We conclude that the duality between the surface operators and their surface vector
  calculus counterparts in classical trace spaces is indeed captured by the duality in
  \Cref{sec: Dual system} and \Cref{lem: commutative relations}.

  We point out that if one works with the $\mathbf{L}^2(\Gamma)$-pairing instead of the
  skew-symmetric pairing \eqref{skew-symmetric pairing} from the start, then the two
  isometrically isomorphic perspectives of tangential and ``rotated'' tangential traces
  from \cite{buffa2002traces} are also captured by the abstract theory. Indeed, by
  introducing the trace $\gamma_{\tau}:\cdot\mapsto\cdot\times\mathbf{n}$, one obtains
  $\trace_{\mathbf{curl}}=\gamma_t'\circ\gamma_{\tau}$ and
  $\dualtrace_{\mathbf{curl}}=-\gamma_{\tau}'\circ\gamma_{\tau}$, which also satisfy
  \eqref{eq: weak traces are dual operators}. With these definitions, \Cref{lem:
    commutative relations} leads to two identities corresponding to \eqref{surface
    indentity 1} and
  \begin{align}
    \int_{\Gamma} v\,\text{div}_{\Gamma}\v\,\dif\sigma = -\int_{\Gamma}
    \mathbf{grad}_{\Gamma} v\cdot\v \dif\sigma
    \quad\forall v\in H^{1/2}(\Gamma),\,\v\in \mathbf{H}^{-1/2}(\text{div}_{\Gamma},\Gamma),
  \end{align}
  which is a ``rotated'' version of \eqref{surface indentity 1}, where
  $\gamma_n\mathbf{curl}=\text{div}_{\Gamma}\gamma_{\tau}$ on $\Hcurl$ and
  $\mathbf{H}^{-1/2}(\text{div}_{\Gamma},\Gamma)$ is defined by analogy with
  \eqref{def:minushalf curl trace space}.
\end{example}

\subsection{Surface operators in quotient spaces}
\label{sec:Surface operators in quotient spaces}
Let us investigate the properties of the linear operators between trace spaces induced by
the surface operators defined in \Cref{sec: Surface operators in domains}.
\medskip

\defbox{%
\begin{defi} \label{def: trace space operator}
  We call \cred{\emph{quotient surface operators}} the bounded linear maps
  \begin{align}\label{eq: def of S trace operators}
    \cred{\S^t_k} : \left\{
      \arraycolsep=1.4pt\def\arraystretch{1.2}
      \begin{array}{rcl}
        \mathcal{T}(\A_k) &\rightarrow &\mathcal{T}(\A_{k+1})\\
        {[}\x{]} &\mapsto &\bm{\pi}^t_{k+1}\A_k\x 
      \end{array}
    \right.
    &&\text{and}&&
    \cred{\S^n_{k+1}}:
    \left\{
      \arraycolsep=1.4pt\def\arraystretch{1.2}
      \begin{array}{rcl}
	\mathcal{T}(\AT_{k+1})&\rightarrow &\mathcal{T}(\AT_{k})\\	
        {[}\z{]}&\mapsto &\bm{\pi}^n_{k}\AT_{k+1}\z
      \end{array}
    \right..
  \end{align}
\end{defi}}
\medskip

We verify that $\S^t_k$ is well-defined. The analogous result holds for $\S^n_{k+1}$ by
duality. Suppose that $\x_{\circ}\in \dom(\rA_k)$. By the complex property, we evaluate
\begin{align}
  \begin{split}
    \llangle\bm{\pi}^t_{k+1}\A_k\x_{\circ},[\z]\rrangle_{\A_{k+1}}
    &=(\A_{k+1}\A_k\x_{\circ},\z)_{\W_{k+2}} - (\A_k\x_{\circ},\AT_{k+1}\z)_{\W_{k+1}}\\
    &=- (\rA_k\x_{\circ},\rA_{k+1}^*\z)_{\W_{k+1}}\\
    &= - (\rA_{k+1}\rA_k\x_{\circ},\z)_{\W_{k+1}}
    =0
  \end{split}
\end{align}
for all $\z\in \dom(\AT_{k+1})\subset\W_{k+2}$. By \Cref{sec: Dual system}, we conclude that $\bm{\pi}^t_{k+1}\A_k\mr{\x}=0$.

From the above, we also find that for all $\x\in \dom(\A_k)\subset\W_k$ and $\z\in \dom(\AT_{k+1})\subset\W_{k+2}$,
\begin{align}
  \llangle\mathsf{S}^t_k\circ\bm{\pi}^t_{k}\,\x,\bm{\pi}^n_{k+1}\,\z\rrangle_{k+1}  = - (\A_k\x,\AT_{k+1}\z)_{\W_{k+1}}=-\llangle\bm{\pi}^t_{k}\,\x,\S^n_{k+1}\circ\bm{\pi}^n_{k+1}\,\z\rrangle_{k}.
\end{align}
We can view the identity
\begin{align}\label{eq: IBP in trace spaces}
  \llangle\mathsf{S}^t_k[\x],[\z]\rrangle_{k+1}=-\llangle{[}\x],\S^n_{k+1}[\z]\rrangle_{k}
  \quad
  \forall[\x]\in\tracespace,\forall[\z]\in \mathcal{T}(\AT_{k+1}),
\end{align}
as an integration by parts formula in (quotient) trace spaces. 

Recalling \Cref{sec: Dual system}, we can rewrite \eqref{eq: IBP in trace spaces} as 
\begin{equation}
  \mathsf{K}_{k+1}\circ\S^t_k=-(\S^n_{k+1})'\circ\mathsf{K}_{k},
\end{equation}
which gives rise to the commutative diagram
\begin{equation}
  \label{commutative diagram surface operators in trace spaces}
  \begin{tikzcd}
    \mathcal{T}(\AT_{k})' \arrow[r, "-(\S^n_{k+1})'"]& \mathcal{T}(\AT_{k+1})'\\
    \mathcal{T}(\A_k)\arrow[r, "\S^t_k"]\arrow[u,hookrightarrow, two heads, dashed, "\mathsf{K}_{k}\quad"] &\mathcal{T}(\A_{k+1}) \arrow[u,hookrightarrow, two heads, dashed,"\quad\mathsf{K}_{k+1}" right]
  \end{tikzcd}
\end{equation}

We end this section by putting the results of the subsections \ref{sec: Surface operators
  in domains} and \ref{sec:Surface operators in quotient spaces} together into a single
diagram. On the one hand, for all $\x\in \dom(\A_k)$ and $\z\in \dom(\AT_{k+1})$, we find
from the proof of \Cref{lem: commutative relations} that
\begin{align}
  \begin{split}
    \langle\D^t_k\circ\trace_{k}\,\x,\z\rangle_{\dom(\AT_{k+1})'}&=\langle\D^n_{k+1}\circ\dualtrace_{k+1}\z,\x\rangle_{\dom(\A_k)'}\\
    &=\llangle\bm{\pi}^t_{k}\,\x,\S^n_{k+1}\circ\bm{\pi}^n_{k+1}\,\z\rrangle_{k}\\
    &=-\llangle\mathsf{S}^t_{k}\circ\bm{\pi}^t_{k}\,\x,\bm{\pi}^n_{k+1}\,\z\rrangle_{k+1}.
  \end{split}
\end{align}
On the other hand, we have by definition
\begin{align}
  \mathsf{S}^t_k\bm{\pi}^t_{k}\,\x = \bm{\pi}^t_{k+1}\A_k\x &&{\text{and}} &&\S^n_{k+1}\bm{\pi}^n_{k+1}\,\z =\bm{\pi}^n_{k}\AT_{k+1}\z.
\end{align}
Also recall \eqref{eq: pi in terms of I inverse} and \eqref{eq: rho in terms of J inverse}. In summary, the following diagrams are commutative:
\begin{equation}\label{commutative diagram surface operators in domain spaces}
  \begin{tikzcd}[row sep=large, column sep = large]
    \mathcal{T}(\A_k)\arrow[r,rightarrow, "\S^t_k"]\arrow[dd, twoheadleftarrow, two heads, dashed, bend right=50, shift right = 5, "\mathsf{I}^t_{k}" description] & \mathcal{T}(\A_{k+1})\arrow[dd, twoheadleftarrow, two heads, dashed, bend left=50, shift left = 5, "\mathsf{I}^t_{k+1}" description] & && &\mathcal{T}(\AT_{k+1})\arrow[r,rightarrow, "\S^n_{k+1}"]\arrow[dd, twoheadleftarrow, two heads, dashed, bend right=50, shift right = 5, "\mathsf{I}^n_{k+1}" description] & \mathcal{T}(\AT_{k})\arrow[dd, twoheadleftarrow, two heads, dashed, bend left=50, shift left = 5, "\mathsf{I}^n_{k}" description]\\
    \dom(\A_k)\arrow[u,rightarrow, two heads, "\bm{\pi}^t_{k}"]\arrow[r,rightarrow,"\A_k"] &  \dom(\A_{k+1})\arrow[u,rightarrow, two heads, "\bm{\pi}^t_{k+1}"] & && &\dom(\AT_{k+1})\arrow[u,rightarrow, two heads, "\bm{\pi}^n_{k+1}"]\arrow[r,rightarrow,"\AT_{k+1}"] &  \dom(\AT_{k})\arrow[u,rightarrow, two heads, "\bm{\pi}^n_{k}"]\\
    \range(\trace_{k})\arrow[u,twoheadleftarrow, "\trace_{k}"]\arrow[r,rightarrow,"-\D^t_k"] &  \range(\trace_{k+1})\arrow[u,twoheadleftarrow, "\trace_{k+1}"] && & &\range(\dualtrace_{k+1})\arrow[u,twoheadleftarrow, "\dualtrace_{k+1}"]\arrow[r,rightarrow,"-\D^n_{k+1}"] &  \range(\dualtrace_{k})\arrow[u,twoheadleftarrow,"\dualtrace_{k}"]
  \end{tikzcd}
\end{equation}

\section{Trace spaces: Characterization by Regular Subspaces}
\label{Characterization by regular subspaces}

\subsection{Bounded regular decompositions}
\label{sec: Bounded regular decompositions}
In this section, we augment Assumption \ref{Assumption trace}. We first detail results in
the setting of \Cref{def: trace enviro} for primal Hilbert traces, then formulate their
analogs in the dual setting of \Cref{def:trace2}. By symmetry, the primal and dual
settings are evidently two faces of the same coin. From an abstract point of view, they
are identical. Nevertheless, the dual setting is presented for convenience. The two
settings are covered independently to avoid loosing sight of the core considerations.

\subsubsection{Primal decomposition}
\label{sec: Primal decomposition}
Now, we aim at a more detailed characterization of the space $\dom(\AT_{k})'$. Recall that
by the complex property, $\range(\AT_{k+1})\subset\dom(\AT_{k})$.

We refer to \cite[Def. 2.12]{PS2021a} for the next assumption, which introduces additional
structure.
\begin{statement}
  \label{B}
  \label{Assumption continuous and dense embeddings AT}
  For all $k\in\mathbb{Z}$, Assumption \ref{Assumption trace} holds along with the
  following hypotheses:

  \begin{enumerate}[label=\textbf{\emph{\Roman*}}]
  \item The Hilbert spaces $\cred{\Wp_{k}}\subset\W_{k}$ are such that the inclusion maps spawn
    continuous and dense embeddings
    \begin{align}
      \label{dense embedding Yp and ZP}
      \Wp_{k}\hookrightarrow\dom(\AT_{k-1}). 
    \end{align}
    \label{assum B: Wplus}
  \item There exist bounded operators 
    \begin{align}
      \label{lifting and potential op}
      \cred{\mathsf{L}^t_{k+1}}: \dom(\AT_{k})\rightarrow\Wp_{k+1} &&
      \text{and} &&
      \cred{\mathsf{V}^t_{k+1}}:\dom(\AT_{k})\rightarrow \Wp_{k+2}
    \end{align}
    such that
    \begin{align}
      \label{decomposition of DAT with operators}
      \y=(\mathsf{L}^t_{k+1}+\AT_{k+1}\mathsf{V}^t_{k+1})\,\y&&\forall\y\in \dom(\AT_{k}).
    \end{align}\label{assum B: regular decomposition}
  \item The Hilbert spaces
    \begin{equation}
      \cred{\Wp_{k+2}(\AT_{k+1})}:=\left\{\,\z\in\Wp_{k+2}\,\,\vert\,\,\AT_{k+1}\z\in\Wp_{k+1}\,\right\},
    \end{equation}
    equipped with the graph inner product defined for all $\z_1,\z_z\in \Wp_{k+2}(\AT_{k+1})$ by
    \begin{equation}
      (\z_1,\z_2)_{\Wp_{k+2}(\AT_{k+1})} :=(\z_1,\z_2)_{\Wp_{k+2}} + (\AT_{k+1}\z_1,\AT_{k+1}\z_2)_{\Wp_{k+1}},
    \end{equation}
    are such that the inclusions $\Wp_{k+2}\subset \W_{k+2}$
    induce continuous and dense embeddings
    \begin{align}
      \Wp_{k+2}(\AT_{k+1})\hookrightarrow\dom(\AT_{k-1}). 
    \end{align}
  \end{enumerate}
\end{statement}

We adopt a shorter notation for the dual spaces: 
\begin{align}\label{eq: def of Wminus}
  \Wm_k:=(\Wp_k)',\quad k\in\mathbb{Z}.
\end{align} 

\begin{rem} In Hypothesis \ref{assum B: regular decomposition}, \eqref{decomposition of DAT with operators} is a stable regular decomposition of the form
  \begin{equation}\label{D(AT) regular decomposition}
    \dom(\AT_{k}) = \Wp_{k+1} + \AT_{k+1}\Wp_{k+2},\quad k\in\mathbb{Z}.
  \end{equation}
  By stable, we mean that the lifting and potential operators in \eqref{lifting and
    potential op} are bounded. We call it regular due to Hypothesis \ref{assum B: Wplus},
  based on which we can imagine the $\Wp_k$s as subspaces of ``extra regularity''.
\end{rem}

\begin{rem}
  The decomposition in \eqref{decomposition of DAT with operators}/\eqref{D(AT) regular decomposition} need not be direct.
\end{rem}

\begin{rem}
  Assumption \ref{Assumption continuous and dense embeddings AT} is stated for all
  $k\in\mathbb{Z}$. Strictly speaking, in the setting of a minimal complex with
  $k\in\mathbb{Z}$ fixed, to which we adhere in this section, only one stable regular
  decomposition (the one written in \eqref{decomposition of DAT with operators} and
  involving the regular spaces $\Wp_{k+1}$ and $\Wp_{k+2}$) is necessary for the
  characterization of $\dom(\AT_k)'$ and $\range(\trace_k)$.
\end{rem}

\begin{lem}
  \label{lem: extension of D to tildeYm}
  Under Assumption \ref{Assumption continuous and dense embeddings AT}, the surface
  operator $\D^t_{k}:\dom(\AT_{k})'\rightarrow \dom(\AT_{k+1})'$ defined in \eqref{def:
    surface operator Dt} can be extended to a continuous mapping
  \begin{equation}\label{D mapping}
    \cred{\D^t_{k}}:\left\{
      \arraycolsep=1.4pt\def\arraystretch{1.2}
      \begin{array}{rcl}
	\Wm_{k+1}&\rightarrow &\Wp_{k+2}(\AT_{k+1})'\\
        \bm{\phi} &\mapsto &\langle\bm{\phi}, \AT_{k+1} \cdot\,\rangle_{\Wm_{k+1}}
      \end{array}
    \right. ,
  \end{equation}
  still designated by the same notation. 
\end{lem}

\begin{proof}
  For all $\bm{\phi}\in\Wm_{k+1}$, it follows by definition that $\forall\z\in\Wp_{k+2}(\AT_{k+1})$,
  \begin{align}\label{eq: D extension well-defined}
    \vert\langle\bm{\phi}, \AT_{k+1} \z\rangle_{\Wm_{k+1}}\vert \leq \norm{\bm{\phi}}_{\Wm_{k+1}}\norm{\AT_{k+1}\z}_{\Wp_{k+1}}
    \leq\norm{\bm{\phi}}_{\Wm_{k+1}}\norm{\z}_{\Wp_{k+2}(\AT_{k+1})}.
  \end{align}
\end{proof}	

\subsubsection{Dual decomposition}
We may also adopt the adjoint perspective. It goes without saying that the development is
completely symmetric to \Cref{sec: Primal decomposition}. We present it for completeness.

\begin{statement}\emph{(cf. Assumption \ref{B})}
  \label{C}
  \label{dual regular decomposition}
  For all $k\in\mathbb{Z}$, beside Assumption \ref{Assumption trace} we stipulate the
  following: 

  \begin{enumerate}[label=\textbf{\emph{\Roman*}}]
  \item The Hilbert spaces $\cred{\Wp_{k}}\subset\W_{k}$ are such that the inclusion maps spawn 
    continuous and dense embeddings
    \begin{align}\label{dense embedding Wplus in DA}
      \Wp_{k}\hookrightarrow\dom(\A_{k}). 
    \end{align}
  \item There exist bounded operators
    \begin{align}
      \cred{\mathsf{L}^n_{k+1}}: \dom(\A_{k+1})\rightarrow\Wp_{k+1} &&
      \text{and} &&
      \cred{\mathsf{V}^n_{k+1}}:\dom(\A_{k+1})\rightarrow \Wp_{k}
    \end{align}
    such that
    \begin{align}
      \label{decomposition of DA with operators}
      \y=(\mathsf{L}^n_{k+1}+\A_{k}\mathsf{V}^n_{k+1})\,\y&&\forall\y\in \dom(\A_{k+1}).
    \end{align}\label{assum C: regular decomposition}
  \item The Hilbert spaces
    \begin{equation}
      \cred{\Wp_{k}(\AT_{k})}:=\left\{\,\x\in\Wp_{k}\,\,\vert\,\,\A_{k}\x\in\Wp_{k+1}\,\right\},
    \end{equation}
    equipped with the graph inner product defined for all $\x_1,\x_z\in \Wp_{k}(\A_{k})$ by
    \begin{equation}
      (\x_1,\x_2)_{\Wp_{k}(\A_{k})} :=(\x_1,\x_2)_{\Wp_{k}} + (\A_{k}\x_1,\AT_{k+1}\x_2)_{\Wp_{k+1}},
    \end{equation}
    are such that the inclusions ${\Wp_{k}}\subset\W_{k}$ induce continuous and dense
    embeddings
    \begin{align}
      \Wp_{k}(\A_{k})\hookrightarrow\dom(\A_{k}). 
    \end{align}
  \end{enumerate}
\end{statement}

\begin{lem}\label{lem: extension of E to tildeYm}
  Under Assumption \ref{dual regular decomposition}, the surface operator $\D^n_{k+1}$ can be extended to a continuous mapping
  \begin{equation}\label{Dn mapping}
    \D^n_k: \left\{
      \arraycolsep=1.4pt\def\arraystretch{1.2}
      \begin{array}{rcl}
	\Wm_{k+1} &\rightarrow &\Wp_{k}(\A_{k})'\\
        \bm{\psi}&\mapsto &\langle\bm{\psi}, \A_k \cdot\,\rangle_{\Wm_{k+1}}
      \end{array}
    \right. .
  \end{equation}
\end{lem}
\begin{proof}
  Parallel to the proof of \Cref{lem: extension of D to tildeYm}, it follows by definition that given $\bm{\psi}\in\Wm_{k+1}$,
  \begin{align}\label{eq: E extension well-defined}
    \vert\langle\bm{\psi}, \A_k \x\rangle_{\Wm_{k+1}}\vert \leq \norm{\bm{\psi}}_{\Wm_{k+1}}\norm{\A_k\x}_{\Wp_{k+1}} \leq\norm{\bm{\psi}}_{\Wm_{k+1}}\norm{\x}_{\Wp_{k}(\A_k)} &&\forall\x\in\Wp_{k}(\A_k).
  \end{align}
\end{proof}	

It is not excluded that both assumptions \ref{Assumption continuous and dense embeddings AT} and \ref{dual regular decomposition} hold, in which case the inclusion
\begin{equation}\label{eq: inclusion cap AT}
  \Wp_{k+1}\hookrightarrow\dom(\AT_{k})\cap\dom(\A_{k+1})
\end{equation}
is assumed to be a dense embedding.

\begin{example}{Stable regular decompositions}\label{eg: stable regular decompositions}
  There is some freedom in choosing the spaces $\Wp_k$, $k\in\mathbb{Z}$. For the de Rham complex though, there are obvious candidates satisfying \eqref{eq: inclusion cap AT} that also satisfy both assumptions \ref{Assumption continuous and dense embeddings AT} and \ref{dual regular decomposition}: functions in the Sobolev space $\Hone$ and vector-fields with components in $\Hone$, which by Rellich's lemma are compactly embedded in the spaces $L^2(\Omega)$ and $\mathbf{L}^2(\Omega)$, respectively.
  \medskip
  \begin{center}
    \begin{tabular}{c|cccc}
      $k$          & $0$           & $1$              & $2$ &  $3$\\
      \midrule
      $\W_k$      & $L^2(\Omega)$  &$\bLtwo$    &  $\bLtwo$ & $\Ltwo$\rule{0pt}{3ex}\\
      $\Wp_k$      & $H^1(\Omega)$  &$\H^1(\Omega)$    &  $\H^1(\Omega)$ &$H^1(\Omega)$\rule{0pt}{3ex}\\
      $\dom(\A_k)$   & $H^1(\Omega)$  &$\Hcurl$          &$\Hdiv$    &$L^2(\Omega)$\rule{0pt}{3ex}\\
      $\dom(\AT_k)$  & $\Hdiv$        &$\Hcurl$          &$H^1(\Omega)$ & $\{0\}$\rule{0pt}{3ex}\\
    \end{tabular}
  \end{center}
  \bigskip

  It is well-known (cf. \cite[Sec.2]{HPE17}, \cite[Lem. 2.4]{HIP02} and \cite[Sec. 3]{HIX06}) that the graph spaces $\dom(\A_k)$ and $\dom(\AT_k)$ given in the above table admit the stable decompositions
  \begin{subequations}
    \begin{align}
      \dom(\A_2)&=\dom(\AT_0) = \Hdiv = \mathbf{H}^1(\Omega) + \mathbf{curl}\,\mathbf{H}^1(\Omega),\label{reg decomposition Hdiv}\\
      \dom(\A_1)&=\dom(\AT_1) = \Hcurl = \mathbf{H}^1(\Omega) + \mathbf{grad}\,H^1(\Omega)\label{reg decomposition Hcurl}
    \end{align}
  \end{subequations}
  These satisfy assumptions \ref{Assumption continuous and dense embeddings AT} and \ref{dual regular decomposition}. Moreover, you may recall that
  \begin{align}
    \H^1(\Omega)\hookrightarrow \Hcurl\cap\Hdiv
  \end{align}
  is a dense embedding \cite[Prop. 2.3]{ABD96}.
\end{example}

\subsection{Characterization of dual spaces}
\label{Characterization of dual spaces}
In light of \Cref{lem: extension of D to tildeYm}, the Hilbert space
\begin{equation}
  \cred{\Wm_{k+1}(\D^t_{k})}
  := \left\{\,\bm{\phi}\in\Wm_{k+1}\,\,\vert\,\,\mathsf{D}^t_k\bm{\phi}\in\Wm_{k+2}\,\right\},
\end{equation}
equipped with the graph norm $\norm{\cdot}^2_{\Wm_{k+1}(\D^t_{k})}:=\norm{\cdot}^2_{\Wm_{k+1}}+\norm{\D^t_{k}\cdot}^2_{\Wm_{k+2}}$, is well-defined under Assumption \ref{Assumption continuous and dense embeddings AT}. In this setting, observe that, if $\bm{\phi}\in\Wm_{k+1}(\D^t_{k})$, then based on the decomposition \eqref{decomposition of DAT with operators}, the evaluation
\begin{equation}\label{eq: evaluation of tildeYm at DAT}
  \bm{\phi}(\y) = \bm{\phi}(\L^t_{k+1}\y) + \bm{\phi}(\AT_{k+1}\mathsf{V}^t_{k+1}\y) = \bm{\phi}(\L^t_{k+1}\y) + \D^t_{k}\bm{\phi}(\mathsf{V}^t_{k+1}\y)
\end{equation}
is well-defined for all $\y\in\dom(\AT_{k})$ thanks to the hypothesis that guarantees $\range(\L^t_{k+1})\subset\Wp_{k+1}$ and $\range(\mathsf{V}^t_{k+1})\subset\Wp_{k+2}$.

\begin{theo}\label{thm: characterization of dual T}
  Assumption \ref{Assumption continuous and dense embeddings AT} guarantees the following
  isomorphism of normed vector spaces,  
  \begin{equation}
    \dom(\AT_k)'\cong\Wm_{k+1}(\D^t_{k}).
  \end{equation}
\end{theo}
\begin{proof}
  Due to \eqref{dense embedding Yp and ZP} from Hypothesis \emph{\ref{assum B: Wplus}} of Assumption \ref{B}, the restriction of functionals $\dom(\AT_{k+1})'\hookrightarrow \Wm_{k+2}$ is a continuous embedding, so the inclusion $\dom(\AT_k)'\subset\Wm_{k+1}(\D^t_{k})$ is immediate from \Cref{def: weak trace space operator}.

  Moreover, for all $\bm{\phi}\in\Wm_{k+1}(\D^t_{k})$, we estimate using \eqref{eq:
    evaluation of tildeYm at DAT} that
  \begin{align}
    \label{eq: estimate tildeYm in D(AT) prime}
    \begin{split}
      \vert\bm{\phi}(\y)\vert&\leq \norm{\bm{\phi}}_{\Wm_{k+1}}\norm{\L^t_{k+1}\y}_{\Wp_{k+1}} + \norm{\D^t_{k}\bm{\phi}}_{\Wm_{k+2}}\norm{\mathsf{V}^t_{k+1}\y}_{\Wp_{k+2}}\\
      &\leq C (\norm{\bm{\phi}}_{\Wm_{k+1}}+\norm{\D^t_{k}\bm{\phi}}_{\Wm_{k+2}})\norm{\y}_{\dom(\AT_k)}
    \end{split}
  \end{align}
  for all $\y\in\dom(\AT_k)$, where $C>0$ is a constant of continuity related to the
  boundedness of the potential and lifting operators in hypothesis \emph{\ref{assum B:
      regular decomposition}} of Assumption \ref{Assumption continuous and dense
    embeddings AT}. We conclude that
  \begin{align}
    \Wm_{k+1}(\D^t_{k})\subset\dom(\AT_k)'.
  \end{align}

  Notice that it also follows from \eqref{eq: estimate tildeYm in D(AT) prime} that
  \begin{equation}
    \norm{\bm{\phi}}_{\dom(\AT_k)'}=\sup_{0\neq\y\in\dom(\AT)}\frac{\vert\bm{\phi}(\y)\vert}{\norm{\y}_{\dom(\AT_k)}}
    \leq C (\norm{\bm{\phi}}_{\Wm_{k+1}}+\norm{\D^t_{k}\bm{\phi}}_{\Wm_{k+2}})
    = C\norm{\bm{\phi}}_{\Wm_{k+1}(\D^t_{k})}
  \end{equation}
  for all $\bm{\phi}\in\Wm_{k+1}(\D^t_{k})$. In other words, the identity map is continuous as a mapping 
  \begin{align}
    \Wm_{k+1}(\D^t_{k})\hookrightarrow\dom(\AT_k)'.
  \end{align} 
  Appealing to the bounded inverse theorem verifies the equivalence of norms.
\end{proof}

Similarly, under Assumption \ref{dual regular decomposition}, \Cref{lem: extension of E to tildeYm} ensures that the Hilbert space
\begin{equation}
  \Wm_{k+1}(\D^n_{k+1}):= \left\{\,\bm{\psi}\in\Wm_{k+1}\,\,\vert\,\,\mathsf{D}^n_{k+1}\bm{\psi}\in\Wm_k\,\right\},
\end{equation}
equipped with the graph norm $\norm{\cdot}_{\Wm_{k+1}(\D^n_{k+1})}:=\norm{\cdot}_{\Wm_{k+1}}+\norm{\D^n_{k+1}\cdot}_{\Wm_k}$, is well-defined. We obtain the following analogous result.

\begin{theo}[cf. \Cref{thm: characterization of dual T}]\label{thm: characterization of dual N}
  Under Assumption \ref{dual regular decomposition}, we conclude the isomorphism of normed
  vector spaces
  \begin{equation}
    \dom(\A_{k+1})'\cong\Wm_{k+1}(\D^n_{k+1}).
  \end{equation}
\end{theo}

\begin{example}{Characterization of dual spaces}\label{eg. Characterization of dual spaces}
  Now, we specialize the theoretical results of \Cref{Characterization of dual spaces} to the 3D de Rham setting using the table in example \ref{eg: stable regular decompositions}. We obtain the following characterization of the dual spaces:
  \begin{subequations}
    \begin{align}
      \Hcurl'&=\dom(\A_1)'=\dom(\AT_1)' \cong \left\{\,\bm{\phi}\in \widetilde{\H}^{-1}(\Omega)\,\,\vert\,\,\mathbf{grad}'\,\bm{\phi}\in\widetilde{H}^{-1}(\Omega)\right\},\\
      \Hdiv' &=\dom(\A_2)'\,=\dom(\AT_0)'\cong \left\{\,\bm{\phi}\in \widetilde{\H}^{-1}(\Omega)\,\,\vert\,\,\mathbf{curl}'\,\bm{\phi}\in \widetilde{\mathbf{H}}^{-1}(\Omega)\right\}.
    \end{align}
  \end{subequations}

  Note that these characterizations are interesting in their own right. They do not depend on the theory of traces developed in the previous sections. The take-home message from the de Rham settings \ref{eg: stable regular decompositions} and \ref{eg. Characterization of dual spaces} is that via the decompositions \eqref{reg decomposition Hdiv} and \eqref{reg decomposition Hcurl}, the dual spaces of $\Hcurl$ and $\Hdiv$ can be characterized using more regular spaces such as $H^1(\Omega)$ and $\mathbf{H}^1(\Omega)$.
\end{example}

\subsection{Characterization of trace spaces}
\label{sec: Characterization of trace spaces} 
We have almost reached characterizations of the ranges of the Hilbert traces
$\range(\trace_k)$ and $\range(\dualtrace_k)$ in terms of the spaces of ``extra
regularity'' provided by Assumptions \ref{Assumption continuous and dense embeddings
  AT} and \ref{dual regular decomposition}. To achieve these new characterizations, we
introduce the following spaces for all $k\in\mathbb{Z}$:
\begin{align}
  \cred{\mathring{\W}^{n,+}_k:=\Wp_{k}\cap\dom(\A^*_{k-1})},\label{def of not t space}
  &&\text{and} &&
  \cred{\mathring{\W}^{t,+}_k:=\Wp_{k}\cap\dom(\rA_{k})}.
\end{align}
Notice that by propositions \ref{lem: kernel of n trace} and \ref{lem: kernel of trace}, we have
\begin{align}
  \mathring{\W}^{n,+}_k=\Wp_k\cap\null(\dualtrace_{k-1}),
  &&\text{and}
  &&\mathring{\W}^{t,+}_k=\Wp_k\cap\null(\trace_k),
\end{align}
respectively. 

\begin{statement}
  \label{density if Wnnot spaces}
  \label{D}
  Suppose that Assumption \ref{Assumption continuous and dense embeddings AT} holds. For
  all $k\in\mathbb{Z}$, we make the hypothesis that the inclusion map
  $\W^{+}_k\subset\dom(\AT_{k-1})$ spawns a \emph{continuous} and \emph{dense} embedding
  \begin{align}\label{eq: embedding of circ t in DA*}
    \mathring{\W}^{n,+}_{k}\hookrightarrow\dom(\A_{k-1}^*).
  \end{align}
\end{statement}

The next result involves the annihilator
\begin{equation}
  \cred{(\mathring{\W}^{n,+}_{k+1})^{\circ}} :=\left\{\bm{\phi}\in\Wm_{k+1} \,\,\vert\,\,\langle\bm{\phi},\y\rangle_{\Wm_{k+1}}=0,\,\forall\y\in\mathring{\W}^{n,+}_{k+1}\right\}.
\end{equation}

\thmbox{%
\begin{theo}\label{thm: characterization of range Tt}
  Taking for granted Assumption \ref{density if Wnnot spaces} we obtain the characterization
  \begin{equation}
    \label{eq: range of trace is characterizaton}
    \range(\trace_{k}) = \Wm_{k+1}(\D^t_{k})\cap (\mathring{\W}^{n,+}_{k+1})^{\circ}=\left\{\,\bm{\psi}\in(\mathring{\W}^{n,+}_{k+1})^{\circ}\,\,\vert\,\,\mathsf{D}^t_k\bm{\psi}\in(\mathring{\W}^{n,+}_{k+2})^{\circ}\,\right\},
  \end{equation}
  in the sense of equality of functionals in $\Wm_{k+1}$ and with equivalent norms. 
\end{theo}}

\begin{proof}
  We already know by \Cref{prop: range is annihilator} that
  $\range(\trace_k)=\dom(\A_{k}^*)^{\circ}$.  To verify the equality on the right, recall
  that
  $\mathsf{D}^t_k(\range(\trace_k))\subset\range(\trace_{k+1})=\dom(\A^*_{k+1})^{\circ}$.

  ``$\subset$'': On the one hand, since $\dom(\A_{k}^*)^{\circ}\subset\dom(\AT_{k})'$, it
  follows immediately from \Cref{thm: characterization of dual T} and \eqref{eq: embedding
    of circ t in DA*} that $\range(\trace_{k})\subset\Wm_{k+1}(\D^t_{k})$. Moreover, as
  $\mathring{\W}^{n,+}_{k+1}\subset\dom(\A_{k}^*)$, any functional in the annihilator of
  $\dom(\A_k^*)$ will, in particular, vanish on $\mathring{\W}^{n,+}_{k+1}$, which implies
  $\dom(\A_{k}^*)^{\circ} \subset (\mathring{\W}^{n,+}_{k+1})^{\circ}$.

  Thanks to the continuous embedding of Assumption \ref{B}\ref{assum B: Wplus}
  and \eqref{D51a} from the definition of the operator $\D^{t}_{k}$, we find for every
  $\bm{\varphi}\in\dom(\AT_{k})'$:
  \begin{align*}
    \norm{\bm{\varphi}}_{\Wm_{k+1}} + \norm{\D^{t}_{k}\bm{\varphi}}_{\Wm_{k+2}} & =
    \sup\limits_{\w\in\Wp_{k+1}}\frac{|\bm{\varphi}(\w)|}{\norm{\w}_{\Wp_{k+1}}} +
    \sup\limits_{\w\in\Wp_{k+2}}\frac{|\bm{\varphi}(\AT_{k+1}\w)|}{\norm{\w}_{\Wp_{k+2}}} \\
    & \leq C
    \sup\limits_{\w\in\dom(\AT_{k})}\frac{|\bm{\varphi}(\w)|}{\norm{\w}_{\dom(\AT_{k})}} +
    \sup\limits_{\w\in\dom(\AT_{k+1})}\frac{|\bm{\varphi}(\AT_{k+1}\w)|}{\norm{\w}_{\dom(\AT_{k+1})}}
    \leq 2C \norm{\bm{\varphi}}_{\dom(\AT_{k})'},
  \end{align*}
  for some constant $C>0$ independent of $\bm{\varphi}$.

  ``$\supset$'': On the other hand, it also follows by \Cref{thm: characterization of dual
    T} that any $\bm{\phi}\in \Wm_{k+1}(\D^t_{k})\cap (\mathring{\W}^{n,+}_{k+1})^{\circ}$
  is a continuous functional in $\dom(\AT_{k})'$ vanishing on
  $\mathring{\W}^{n,+}_{k+1}$. By Assumption~\ref{D} $\mathring{\W}^{n,+}_{k+1}$ is
  densely embedded in $\dom(\A_{k}^*)$. Thus, $\bm{\phi}$ must also vanish on
  $\dom(\A_{k}^*)$ by continuity. We conclude that the inclusion
  $\Wm_{k+1}(\D^t_{k})\cap
  (\mathring{\W}^{n,+}_{k+1})^{\circ}\subset\range(\trace_{k})=\dom(\A_{k}^*)^{\circ}$
  holds.

  Finally, the estimate \eqref{eq: estimate tildeYm in D(AT) prime} gives us
  \begin{align*}
    \norm{\bm{\phi}}_{\dom(\AT_{k})'} \leq C (\norm{\bm{\phi}}_{\Wm_{k+1}}+\norm{\D^t_{k}\bm{\phi}}_{\Wm_{k+2}})
  \end{align*}
  with $C>0$ independent of $\bm{\phi}$.
\end{proof}

Of course, there is a symmetric statement on the dual side. 
\begin{statement}\emph{(cf. Assumption~\ref{D})}
  \label{E}
  \label{density if Wtnot spaces}
  Suppose that Assumption \ref{dual regular decomposition} holds. For all $k\in\mathbb{Z}$, we make the hypothesis that the inclusion map $\W^{+}_k\subset\dom(\A_{k})$ spawns a continuous and dense embedding
  \begin{align}
    \mathring{\W}^{t,+}_{k}\hookrightarrow\dom(\rA_{k}).
  \end{align}
\end{statement}

\thmbox{%
  \begin{theo}[cf. \Cref{thm: characterization of range Tt}]
    \label{thm: characterization of range Tn}
  Under Assumption \ref{density if Wtnot spaces} we have equality in $\Wm_{k+1}$ with
  equivalent norms, 
  \begin{equation}
    \label{eq: range of trace is characterizaton dual}
    \range(\dualtrace_{k+1}) =
    \Wm_{k+1}(\D^n_{k+1})\cap (\mathring{\W}^{t,+}_{k+1})^{\circ}=
    \left\{\,\bm{\psi}\in(\mathring{\W}^{t,+}_{k+1})^{\circ}\,\,\vert\,\,\mathsf{D}^n_{k+1}
      \bm{\psi}\in(\mathring{\W}^{t,+}_{k})^{\circ}\,\right\}.
  \end{equation}
\end{theo}}

\begin{example}{Characterization of trace spaces}\label{Characterization of trace spaces}
  We specialize the theoretical results of \Cref{sec: Characterization of trace spaces} to the 3D de Rham setting.
  \medskip
  \begin{center}
    \renewcommand{\arraystretch}{1.2}
    \begin{tabular}{c|cccc}
      $k$          & $0$           & $1$              & $2$ &  $3$\rule{0pt}{3ex}\\
      \midrule
      $\W_k$      & $L^2(\Omega)$  &$\bLtwo$    &  $\bLtwo$ & $\Ltwo$\rule{0pt}{3ex}\\
      $\Wp_k$      & $H^1(\Omega)$  &$\H^1(\Omega)$    &  $\H^1(\Omega)$ &$H^1(\Omega)$\\
      $\mathring{\W}^{t,+}_k$      & $\Honenot$  &$\H^1(\Omega)\cap\Hcurlnot$    &  $\H^1(\Omega)\cap\Hdivnot$ & $\Honenot$\rule{0pt}{3ex}\\
      $\mathring{\W}^{n,+}_k$      & $\Honenot$  &$\H^1(\Omega)\cap\Hdivnot$    &  $\H^1(\Omega)\cap\Hcurlnot$ &$\Honenot$\rule{0pt}{3ex}\\
    \end{tabular}
  \end{center}
  \bigskip

  Loosely speaking, theorems \ref{thm: characterization of range Tt} and \ref{thm:
    characterization of range Tn} state that the range of the Hilbert trace is a subspace
  of functionals in the dual of a regular space $\Wp_k$ whose image under the
  corresponding surface operator also lies in the dual of $\Wp_{k+1}$. Linear functionals
  in that subspace vanish on a dense subset of the dual trace's kernel:
  \begin{subequations}
    \begin{align}
      \range(\trace_{\mathbf{curl}})&=\range(\dualtrace_{\mathbf{curl}}) = \left\{\,\bm{\phi}\in \widetilde{\H}^{-1}(\Omega)\cap\Hcurlnot^{\circ}\,\,\vert\,\,\mathbf{grad}'\,\bm{\phi}\in\widetilde{H}^{-1}(\Omega)\cap\Honenot^{\circ}\right\},\label{eq: charact dualtrace curl}\\
      \range(\trace_{\mathbf{grad}})&=\range(\dualtrace_{\text{div}}) \,\,\,= \left\{\,\bm{\phi}\in \widetilde{\H}^{-1}(\Omega)\cap\Hdivnot^{\circ}\,\,\vert\,\,\mathbf{curl}'\,\bm{\phi}\in \widetilde{\mathbf{H}}^{-1}(\Omega)\cap\Hcurlnot^{\circ}\right\}.\label{eq: charact dualtrace div}
    \end{align}
  \end{subequations}

  One thing immediately apparent is that
  $\range(\dualtrace_{\mathbf{curl}})=\range(\trace_{\mathbf{curl}})$ and
  $\range(\dualtrace_{\text{div}})=\range(\trace_{\mathbf{grad}})$, which is expected
  because we already know from previous sections that
  \begin{subequations}
    \begin{gather}
      \range(\dualtrace_{\mathbf{curl}}) = \dom(\rA_1)^{\circ}
      = \Hcurlnot' = \dom(\A^*_1)^{\circ} = \range(\trace_{\mathbf{curl}}),\\
      \range(\dualtrace_{\text{div}}) = \dom(\rA_2)^{\circ}
      = \Hdivnot' = \dom(\A^*_1) =\range(\trace_{\mathbf{grad}}).
    \end{gather}
  \end{subequations}

  Before we compare these characterizations with \eqref{def:minushalf trace space} and \eqref{def:minushalf curl trace space}, we want to reformulate them in terms of quotient spaces in the next section.
\end{example}

\subsection{Characterization of trace spaces in quotient spaces}
\label{Characterization of trace spaces based on quotient spaces} 
We can reformulate the characterizations of \Cref{sec: Characterization of trace spaces}
in terms of quotient spaces. To proceed, let us set
\begin{subequations}
  \begin{align}
    \cred{\mathbf{T}^{t,+}_k}&:=\Wp_{k}/\mathring{\W}^{t,+}_{k},  & \cred{\mathbf{T}^{t,-}_k} &:= \left(\mathbf{T}^{t,+}_k\right)',\label{eq: def of Tt}\\
    \cred{\mathbf{T}^{n,+}_k}&:=\Wp_{k}/\mathring{\W}^{n,+}_{k},  & \cred{\mathbf{T}^{n,-}_k} &:= \left(\mathbf{T}^{n,+}_k\right)'.\label{eq: def of Tn}
  \end{align}
\end{subequations}

Under Assumption \ref{density if Wnnot spaces} (resp. \ref{density if Wtnot spaces}), it follows by definition of the space $\mr{\W}^{n,+}_{k}$ (resp. $\mr{\W}^{t,+}_k$) that the dense embedding $\Wp_k\hookrightarrow\dom(\AT_{k-1})$ (resp. $\Wp_k\hookrightarrow\dom(\A_{k})$) induces a well-defined and dense embedding
\begin{align}
  \left\{
    \arraycolsep=1.4pt\def\arraystretch{1.2}
    \begin{array}{rcl}
      \mathbf{T}^{n,+}_{k} &\hookrightarrow &\mathcal{T}(\AT_{k-1})\\
      {[}\x{]}&\mapsto &\bm{\pi}^n_{k-1}\x
    \end{array}
  \right. 
  &&\text{\Bigg(resp.}
  &&
  \left\{
    \arraycolsep=1.4pt\def\arraystretch{1.2}
    \begin{array}{rcl}
      \mathbf{T}_{k}^{t,+} &\hookrightarrow &\mathcal{T}(\A_k)\\
      {[}\x{]}&\mapsto &\bm{\pi}^t_k\x
    \end{array}
  \right.\text{\quad\Bigg)}
\end{align}
on the quotient spaces. Accordingly, the associated restriction of functionals
\begin{align}\label{restriction of functionals quotient spaces}
  \left\{
    \arraycolsep=1.4pt\def\arraystretch{1.2}
    \begin{array}{rcl}
      \mathcal{T}(\AT_{k-1})' &\hookrightarrow &\mathbf{T}_{k}^{n,-}\\
      \bm{\psi}&\mapsto &\left\{\,[\x]\mapsto\bm{\psi}(\bm{\pi}^n_{k-1}\x)\,\right\}
    \end{array}
  \right. 
  &&\text{\Bigg(resp.}
  &&
  \left\{
    \arraycolsep=1.4pt\def\arraystretch{1.2}
    \begin{array}{rcl}
      \mathcal{T}(\A_k)' &\hookrightarrow &\mathbf{T}_{k}^{t,-}\\
      \bm{\phi}&\mapsto &\left\{\,[\x]\mapsto\bm{\phi}(\bm{\pi}^t_k\x)\,\right\}
    \end{array}
  \right. \text{\quad\Bigg)}
\end{align}
is also well-defined and gives rise to dense embeddings.

In the next lemma, we make explicit the mappings induced on the quotient spaces by restricting the operators $\AT_{k-1}$ and $\A_{k}$ to $\Wp_{k}$. Those are the restrictions of the surface operators $\mathsf{S}^n_{k-1}$ and $\mathsf{S}^t_k$ to $\mathbf{T}^{n,+}_{k}$ and $\mathbf{T}^{t,+}_{k}$, respectively; cf. \Cref{def: trace space operator}.

\begin{lem}
  Assumptions \ref{density if Wnnot spaces} and \ref{density if Wtnot spaces} imply that the mappings
  \begin{align}\label{eq: def of the S operators in lemma}
    \cred{\hat{\mathsf{S}}^n_{k+1}}: \left\{
      \arraycolsep=1.4pt\def\arraystretch{1.2}
      \begin{array}{rcl}
	\mathbf{T}^{n,+}_{k+2} &\rightarrow &\mathcal{T}(\AT_{k})\\
        {[}\z{]}&\mapsto &\bm{\pi}^n_{k}\AT_{k+1}\z
      \end{array}
    \right.
    &&\text{and}
    &&			\cred{\hat{\mathsf{S}}^t_k}: \left\{
      \arraycolsep=1.4pt\def\arraystretch{1.2}
      \begin{array}{rcl}
	\mathbf{T}^{t,+}_k &\rightarrow &\mathcal{T}(\A_{k+1})\\
        {[}\x{]}&\mapsto &\bm{\pi}^t_{k+1}\A_k\x
      \end{array}
    \right., 
  \end{align}
  respectively, are well-defined and continuous.
\end{lem}

\begin{proof}
  Consider the mapping on the left. We know from the complex property for $\AT_k$ in
  Assumption \ref{Assumption trace} that $\AT_{k+1}\z\in\dom(\AT_{k})$ for all
  $\z\in\Wp_{k+1}$. We only need to verify that $\AT_{k+1}\z_{\circ}\in\dom(\A^*_{k})$ for
  all $\z_{\circ}\in\mr{\W}^{n,+}_{k+2}=\Wp_{k+2}\cap\dom(\A^*_{k+1})$, but this
  immediately follows from the complex property for $\A^*_{k+1}$, also provided by
  Assumption \ref{Assumption trace}. The proof is similar for $\hat{\mathsf{S}}^t_k$.
\end{proof}

Using the same strategy as in lemmas \ref{lem: extension of D to tildeYm} and \ref{lem:
  extension of E to tildeYm}, the mappings
\begin{align}\label{def D hat}
  {\hat{\mathsf{D}}^t_{k}} := (	\hat{\mathsf{S}}^n_{k+1})':\mathcal{T}(\AT_k)'\rightarrow \mathbf{T}^{n,-}_{k+2} &&\text{and}
  &&
  {\hat{\mathsf{D}}^n_{k}} := (	\hat{\mathsf{S}}^t_{k})':\mathcal{T}(\A_{k+1})'\rightarrow \mathbf{T}^{t,-}_k,
\end{align}
defined as the bounded operators dual to $\hat{\mathsf{S}}^n_{k+1}$ and
$\hat{\mathsf{S}}^t_{k}$, can be extended, using \eqref{restriction of functionals quotient spaces},
to the continuous mappings
\begin{align}\label{extension of D hat}
  \cred{\hat{\mathsf{D}}^t_{k}} : \mathbf{T}^{n,-}_{k+1}\rightarrow \mathbf{T}^{n,+}_{k+2}(\hat{\mathsf{S}}^n_{k+1})' &&\text{and}
  && \cred{\hat{\mathsf{D}}^n_{k}} : \mathbf{T}^{t,-}_{k+1}\rightarrow \mathbf{T}^{t,+}_{k}(\hat{\mathsf{S}}^t_{k})',
\end{align}
involving the dual spaces of the Hilbert spaces
\begin{subequations}
  \begin{align}
    \mathbf{T}^{n,+}_{k+2}(\hat{\mathsf{S}}^n_{k+1})&:= \left\{\,[\z]\in\mathbf{T}^{n,+}_{k+2}\,\,\vert\,\,\hat{\mathsf{S}}^n_{k+1}[\z]\in\mathbf{T}^{n,+}_{k+1}\,\right\},\\
    \mathbf{T}^{t,+}_{k}(\hat{\mathsf{S}}^t_{k})&:= \left\{\,[\x]\in\mathbf{T}^{t,+}_{k}\,\,\,\vert\,\,\hat{\mathsf{S}}^t_{k}[\x]\in\mathbf{T}^{t,+}_{k+1}\,\right\},
  \end{align}
\end{subequations}
equipped with the natural graph inner products. 

With the operators \eqref{extension of D hat}, we can reformulate theorems \ref{thm: characterization of range Tt} and \ref{thm: characterization of range Tn} using the isometric isomorphisms
\begin{align}
  (\Wp_{k}/\mathring{\W}^{t,+}_{k})'\cong(\mathring{\W}^{t,+}_{k})^{\circ} &&\text{and} && \Wp_{k}/\mathring{\W}^{n,+}_{k}\cong(\mathring{\W}^{n,+}_{k})^{\circ}
\end{align}
provided by \cite[Thm. 4.9]{MR1157815}.
\medskip

\thmbox{%
\begin{theo}\label{characterization of range T}
  Under assumptions \ref{density if Wnnot spaces} and \ref{density if Wtnot spaces} we
  have the isomorphisms of Hilbert spaces 
  \begin{subequations}
    \begin{align}
      \range(\trace_k) \cong \left\{\,\bm{\phi}\in\mathbf{T}^{n,-}_{k+1}\,\,\vert\,\,\hat{\mathsf{D}}^t_k\bm{\phi}\in\mathbf{T}^{n,-}_{k+2}\,\right\} &&\text{and}
      &&\range(\dualtrace_{k}) \cong  \left\{\,\bm{\phi}\in\mathbf{T}^{t,-}_{k}\,\,\,\,\vert\,\,\hat{\mathsf{D}}^n_k\bm{\phi}\in\mathbf{T}^{t,-}_{k-1}\,\right\},
    \end{align}
  \end{subequations}
  respectively.
\end{theo}}

\begin{example}{Characterization of trace spaces by quotient spaces}\label{eg: Characterization of trace spaces using quotient spaces}
  Recall from \eqref{eq: ker t trace} and \eqref{eq: ket n trace} that
  $\null(\gamma_t)= \Hcurlnot$ and $\null(\gamma_n)=\Hdivnot$. So let us denote the spaces of $\mathbf{H}^1$-regular vector fields with vanishing tangential and normal traces by
  \begin{subequations}
    \begin{align}
      \Ht&:=\null(\gamma_t\big\vert_{\mathbf{H}^1(\Omega)})=\mathbf{H}^1(\Omega)\cap\Hcurlnot\label{eq: ker of t trace on H1}\\ 
      \Hn&:=\null(\gamma_n\big\vert_{\mathbf{H}^1(\Omega)})=\mathbf{H}^1(\Omega)\cap\Hdivnot,\label{eq: ker of n trace on H1}
    \end{align}
  \end{subequations}
  respectively.
  \medskip
  \begin{center}
    \renewcommand{\arraystretch}{1.2}
    \begin{tabular}{c|cccc}
      $k$          & $0$           & $1$              & $2$ &  $3$\rule{0pt}{3ex}\\
      \midrule
      $\W_k$      & $L^2(\Omega)$  &$\bLtwo$    &  $\bLtwo$ & $\Ltwo$\rule{0pt}{3ex}\\
      $\Wp_k$      & $H^1(\Omega)$  &$\H^1(\Omega)$    &  $\H^1(\Omega)$ &$H^1(\Omega)$\\
      $\mathbf{T}^{t,+}_k$      & $\Hone/\Honenot$ &$\H^1(\Omega)/\Ht$    &  $\H^1(\Omega)/\Hn$ & $\Hone/\Honenot$ \rule{0pt}{3ex}\\
      $\mathbf{T}^{n,+}_k$      & $\Hone/\Honenot$  &$\H^1(\Omega)/\Hn$    &  $\H^1(\Omega)/\Ht$ &$\Hone/\Honenot$\rule{0pt}{3ex}\\
    \end{tabular}
  \end{center}
  \bigskip

  Reformulating \eqref{eq: charact dualtrace curl} and \eqref{eq: charact dualtrace div}, we obtain
  \begin{subequations}
    \begin{align}
      \range(\trace_{\mathbf{curl}})&=\range(\dualtrace_{\mathbf{curl}}) \cong \left\{\,\bm{\phi}\in \Big(\mathbf{H}^1(\Omega)/\Ht\Big)'\,\,\vert\,\,\mathbf{grad}'\,\bm{\phi}\in\Big(\Hone/\Honenot\Big)'\right\},\label{eq: char quotient 1}\\
      \range(\trace_{\mathbf{grad}})&=\range(\dualtrace_{\text{div}}) \,\,\,\cong \left\{\,\bm{\phi}\in \Big(\mathbf{H}^1(\Omega)/\Hn\Big)'\,\,\vert\,\,\mathbf{curl}'\,\bm{\phi}\in \Big(\mathbf{H}^1(\Omega)/\Ht\Big)'\right\}.\label{eq: char quotient 2}
    \end{align}
  \end{subequations}

  These characterizations are to be compared with
  \begin{subequations}
    \begin{align}
      \mathbf{H}^{-1/2}(\text{curl}_\Gamma,\Gamma)=\left\{\,\bm{\phi}\in \mathbf{H}_t^{-1/2}(\Gamma)\,\,\vert\,\,\text{curl}_{\Gamma}\,\bm{\phi}\in H^{-1/2}(\Gamma)\,\right\}=\range(\gamma_t),\\
      H^{1/2}(\Gamma)=\left\{\,\bm{\phi}\in H^{-1/2}(\Gamma)\,\,\vert\,\,\mathbf{curl}_{\Gamma}\,\bm{\phi}\in \mathbf{H}_t^{-1/2}(\Gamma)\,\right\}=\range(\gamma),
    \end{align}
  \end{subequations}
  where as before the two spaces
  \begin{subequations}
    \begin{align}
      H^{-1/2}(\Gamma)&=\left(H^{1/2}(\Gamma)\right)'=\left(\gamma H^1(\Omega)\right)'\\ 
      \H_t^{-1/2}(\Gamma)&=\left(\H_t^{1/2}(\Gamma)\right)'=\left(\gamma_t\H^1(\Omega)\right)'
    \end{align}
  \end{subequations}
  are dual to the more regular spaces $\gamma \,H^1(\Omega)$ and $\gamma_t\,\H^1(\Omega)$, respectively. 

  In the classical trace spaces, the quotient spaces involved in \eqref{eq: char quotient 1} and \eqref{eq: char quotient 2} are featured implicitly, because as previously stated in \eqref{eq: ker of t trace on H1} and \eqref{eq: ker of n trace on H1}, $\Ht$ and $\Hn$ are kernels which vanish under application of the traces. In fact, since $\gamma:\Hone\rightarrow H^{1/2}(\Gamma)$ and  $\gamma_t:\mathbf{H}^1(\Omega)\rightarrow \mathbf{H}_t^{1/2}(\Gamma)$ are surjective, it follows from \eqref{eq: ker of t trace on H1} and \eqref{eq: ker of n trace on H1} that the same argument as in the 3D de Rham setting \ref{spe: trace spaces} shows that the traces induce the isomorphisms
  \begin{align}
    \mathbf{H}_t^{1/2}(\Gamma)\cong\mathbf{H}^1(\Omega)/\Ht &&\text{and}   &&H^{1/2}(\Gamma)\cong\Hone/\Honenot,
  \end{align}
  which in turn imply isomorphisms between the dual spaces.

  We would like to draw the reader's attention to the fact that it is an annihilator
  related to the kernel of the dual trace that is used to characterize the range of the
  primal trace and vice-versa. This is in agreement with the characterizations provided in
  \cite{buffa2002traces}, where the range of $\gamma_t$ is characterized using the dual
  space $(\gamma_{\tau}\mathbf{H}^1(\Omega))'$, involving the rotated tangential trace
  $\gamma_{\tau}$ discussed in the 3D de Rham setting \ref{Commutative relations}. As in \cite{buffa2002traces},
  recall that if the skew-symmetric pairing \eqref{skew-symmetric pairing} is replaced
  with the $\mathbf{L}^2(\Gamma)$-pairing, the dual trace $\dualtrace_{\mathbf{curl}}$,
  corresponding with the rotated tangential trace (roughly speaking), arises in the
  abstract setting of \Cref{sec: Dual traces} as dual to $\trace_{\mathbf{curl}}$, which
  corresponds to $\gamma_t$.

  Finally, notice that the surface operators $\text{curl}_{\Gamma}$ and
  $\mathbf{curl}_{\Gamma}$ are dual to the domain operators on which the relevant traces
  are applied, which is in line with \eqref{eq: char quotient 1} and \eqref{eq: char
    quotient 2}, i.e. (cf. \cite{buffa2002traces})
\begin{align}
\text{curl}_{\Gamma}\circ\gamma = (\gamma_t\circ\nabla)' &&\text{and} &&\mathbf{curl}_{\Gamma}\circ\gamma_t=(\gamma_n\circ\mathbf{curl})'.
\end{align}
\end{example}

\section{Trace Hilbert Complexes}
\label{sec: Trace Hilbert Complexes}
From now on, we make use of the full setting of Hilbert complexes as presented in
\Cref{sec: Hilbert complexes}. Both Assumptions \ref{density if Wnnot spaces} and
\ref{density if Wtnot spaces} are not required for the mere characterization of the trace
Hilbert complexes in \Cref{sec: Complexes of quotient spaces}: each one of these
hypotheses suffices for the corresponding characterization. However, we do rely on
\emph{both} decompositions for the upcoming compactness result in \Cref{sec: Compactness},
where we must take \eqref{eq: inclusion cap AT} for granted.

\subsection{Complexes of quotient spaces}
\label{sec: Complexes of quotient spaces}
It is easy to verify that $\mathsf{D}^t_{k+1}\circ\mathsf{D}^t_{k}=0$, $\mathsf{D}^n_{k}\circ\mathsf{D}^n_{k+1}=0$, $\mathsf{S}^t_{k+1}\circ\mathsf{S}^t_{k}=0$ and
$\mathsf{S}^n_{k}\circ\mathsf{S}^n_{k+1}=0$. Therefore, we have already seen from \eqref{commutative diagram surface operators in domain spaces} that Hilbert complexes give rise to Hilbert complexes in trace spaces. The bounded complexes
\begin{subequations}
  \begin{equation}\label{Hilbert trace complex RTt}
    \def\arrowlength{7ex}
    \def\arrowdistance{.8}
    \begin{tikzcd}[column sep=\arrowlength]
      \cdots 
      \arrow[r, rightarrow, shift left=\arrowdistance, "\mathsf{D}^t_k"] 
      & 
      \range(\trace_k) 
      \ar[r, rightarrow, shift left=\arrowdistance, "\mathsf{D}^t_k"] 
      & 
      \range(\trace_{k+1})
      \arrow[r, rightarrow, shift left=\arrowdistance, "\mathsf{D}^t_{k+1}"] 
      & 
      \range(\trace_{k+2})
      \arrow[r, rightarrow, shift left=\arrowdistance, "\mathsf{D}_{k+2}"] 
      &
      \cdots,
    \end{tikzcd}
  \end{equation}
  and
  \begin{equation}\label{Hilbert trace complex RTn}
    \def\arrowlength{7ex}
    \def\arrowdistance{.8}
    \begin{tikzcd}[column sep=\arrowlength]
      \cdots 
      \arrow[r, leftarrow, shift right=\arrowdistance, "\mathsf{D}_k"']
      & 
      \range(\dualtrace_{\A_{k}})
      \ar[r, leftarrow, shift right=\arrowdistance, "\mathsf{D}^n_{k+1}"']
      & 
      \range(\dualtrace_{\A_{k+1}})
      \arrow[r, leftarrow, shift right=\arrowdistance, "\mathsf{D}^n_{k+2}"']
      & 
      \range(\dualtrace_{\A_{k+2}})
      \arrow[r, leftarrow, shift right=\arrowdistance, "\mathsf{D}_{k+3}"']
      &
      \cdots,
    \end{tikzcd}
  \end{equation}
\end{subequations}
are isometrically isomorphic to the bounded complexes of quotient spaces
\begin{subequations}
  \begin{equation}\label{Hilbert trace complex TAk}
    \def\arrowlength{7ex}
    \def\arrowdistance{.8}
    \begin{tikzcd}[column sep=\arrowlength]
      \cdots 
      \arrow[r, rightarrow, shift left=\arrowdistance, "\mathsf{S}^t_k"] 
      & 
      \tracespace 
      \ar[r, rightarrow, shift left=\arrowdistance, "\mathsf{S}^t_k"] 
      & 
      \mathcal{T}(\A_{k+1})
      \arrow[r, rightarrow, shift left=\arrowdistance, "\mathsf{S}^t_{k+1}"] 
      & 
      \mathcal{T}(\A_{k+2})
      \arrow[r, rightarrow, shift left=\arrowdistance, "\mathsf{S}^t_{k+2}"] 
      &
      \cdots,
    \end{tikzcd}
  \end{equation}
  and
  \begin{equation}\label{Hilbert trace complex TATk}
    \def\arrowlength{7ex}
    \def\arrowdistance{.8}
    \begin{tikzcd}[column sep=\arrowlength]
      \cdots 
      \arrow[r, leftarrow, shift right=\arrowdistance, "\mathsf{S}^n_k"']
      & 
      \mathcal{T}(\AT_k) 
      \ar[r, leftarrow, shift right=\arrowdistance, "\mathsf{S}^n_{k+1}"']
      & 
      \mathcal{T}(\AT_{k+1})
      \arrow[r, leftarrow, shift right=\arrowdistance, "\mathsf{S}^n_{k+2}"']
      & 
      \mathcal{T}(\AT_{k+2})
      \arrow[r, leftarrow, shift right=\arrowdistance, "\mathsf{S}^n_{k+3}"']
      &
      \cdots.
    \end{tikzcd}
  \end{equation}
\end{subequations}

While the bounded domain complexes are interesting in their own right, the rich structure
of Hilbert complexes reveals itself when closed densely defined unbounded operators are
introduced. As stated in \cite[Chap. 4]{MR3908678}, the complex produced by the latter
contains more information than the associated domain complexes. It turns out that the
characterizations provided in \Cref{Characterization by regular subspaces} shed more light
on the structure of \eqref{Hilbert trace complex RTt}-\eqref{Hilbert trace complex
  TATk}. The next theorem provides a first characterization of what we call \emph{trace
  Hilbert complexes}.
\medskip

\noindent\fbox{%
  \begin{minipage}[c]{0.975\textwidth}
    \begin{theo}\label{thm: trace Hilbert complexes definition}
      Under assumptions \ref{density if Wnnot spaces} and \ref{density if Wtnot spaces}
      respectively, the sequences of unbounded operators
      \begin{gather}
        \label{trace Hilbert complex RTt full}
        \begin{CD}
          \cdots @>{\mathsf{D}^t_{k-1}}>>
          \range(\trace_{k})\subset(\mathring{\W}^{n,+}_{k+1})^{\circ}
          @>{\mathsf{D}^t_{k}}>>
          \range(\trace_{k+1})\subset(\mathring{\W}^{n,+}_{k+2})^{\circ}
          @>{\mathsf{D}^t_{k+2}}>>\cdots
        \end{CD}
        \intertext{and}
        \label{trace Hilbert complex RTn full}
        \begin{CD}
          \cdots
          @<<{\mathsf{D}^n_{k}}<
          \range(\dualtrace_{k}) \subset(\mathring{\W}^{t,+}_{k})^{\circ}
          @<<{\mathsf{D}^n_{k+1}}<
          \range(\dualtrace_{k+1}) \subset(\mathring{\W}^{t,+}_{k+1})^{\circ}
          @<<{\mathsf{D}^n_{k+2}}<
          \cdots
        \end{CD}
      \end{gather}
      are Hilbert complexes as defined in \Cref{sec: Hilbert complexes}.
    \end{theo}%
  \end{minipage}%
}

\begin{proof}
  By symmetry, it is sufficient to verify the claim for \eqref{trace Hilbert complex RTt full}. In light for \eqref{Hilbert trace complex RTt} and \Cref{thm: characterization of range Tt}, we simply need to show that $\mathsf{D}^t_k:\range(\trace_k)\subset (\mathring{\W}^{n,+}_{k+1})^{\circ}\rightarrow (\mathring{\W}^{n,+}_{k+2})^{\circ}$ is a densely defined and closed unbounded linear operator. In fact, since $\range(\trace_k)=\dom(\A^*_k)\subset\dom(AT_k)'$ is a Hilbert space by \Cref{prop: range is annihilator}, we already know that such an operator must be closed, and we only need to confirm that $\range(\trace_k)$ is dense in $(\mathring{\W}^{n,+}_{k+1})^{\circ}$.

  We need two key mappings:
  \begin{itemize}
  \item Recall that since $\Wp_{k+1}$ is a Hilbert space and Hilbert spaces are reflexive (cf. \cite[Sec. 4.5]{MR1157815}, \cite[Thm. 5.5]{MR2759829}), the map
    \begin{equation}
      \bm{\rho}: \left\{
	\arraycolsep=1.4pt\def\arraystretch{1.2}
	\begin{array}{rcl}
          \Wp_{k+1} &\longrightarrow &(\Wm_{k+1})'\\
          \y &\mapsto
          &\left\{
            \arraycolsep=1.4pt\def\arraystretch{1.2}
            \begin{array}{rcl}
              \Wm_{k+1} &\rightarrow &\mathbb{R}\\
              \bm{\phi} &\mapsto &\bm{\rho}\y(\bm{\phi})=\bm{\phi}(\y)
            \end{array}
          \right.
	\end{array}
      \right.
    \end{equation}
    is an isometric isomorphism. Substituting $\bm{\rho}^{-1}(\widetilde{\bm{\phi}})$ for $\y$ in the definition $(\bm{\rho}\y)(\bm{\phi})=\bm{\phi}(\y)$, we find a useful formula involving the inverse:
    \begin{equation}\label{eq: inverse of rho}
      \widetilde{\bm{\psi}}(\bm{\phi}) = \bm{\phi}(\bm{\rho}^{-1}\widetilde{\bm{\psi}})
    \end{equation}
    for all $\bm{\phi}\in\Wm_{k+1}$ and $\widetilde{\bm{\psi}}\in (\Wm_{k+1})'$.
  \item Since the inclusion $\Wp_{k+1}\hookrightarrow \dom(\AT_k)$ is continuous and dense by Assumption \ref{Assumption continuous and dense embeddings AT}, the \emph{restriction} of functionals $\mathsf{J}:\dom(\AT_k)'\rightarrow \Wm_{k+1}$ is also a continuous and dense embedding. In particular, because $\range(\trace_k)=\dom(\A^*)^{\circ}$ by \Cref{prop: range is annihilator} and $\mathring{\W}^{n,+}_{k+1}\subset\dom(\A_k^*)$ by definition, it satisfies the important property that $\mathsf{J}(\range(\trace_k))\subset (\mathring{\W}^{n,+}_{k+1})^{\circ}$.
  \end{itemize}

  To prove density, we show that an arbitrary functional $\widetilde{\bm{\phi}}_{\circ}\in((\mathring{\W}^{n,+}_{k+1})^{\circ})'$ such that $\widetilde{\bm{\phi}}_{\circ}(\mathsf{J}\bm{\xi})=0$ for all $\bm{\xi}\in\range(\trace_k)$ vanish in $((\mathring{\W}^{n,+}_{k+1})^{\circ})'$. We proceed in three short steps.

  \begin{enumerate}[label=\textbf{\arabic*})]
  \item First, we use the Hahn--Banach theorem to extend $\widetilde{\bm{\phi}}_{\circ}$ to a functional $\widetilde{\bm{\phi}}\in(\Wm_{k+1})'$. By definition,
    \begin{align}\label{vanishing of phi tilde}
      \widetilde{\bm{\phi}}(\mathsf{J}\bm{\xi})=0 &&\forall\bm{\xi}\in\range(\trace_k).
    \end{align}

  \item Secondly, we set $\y:=\bm{\rho}^{-1}\widetilde{\bm{\phi}}\in\Wp_{k+1}\subset\dom(\AT_k)$. Based on \eqref{eq: inverse of rho}, it follows from \eqref{vanishing of phi tilde} that
    \begin{align}\label{eq: property of y}
      \bm{\xi}(\y)=\mathsf{J}\bm{\xi}(\y)=\mathsf{J}\bm{\xi}(\bm{\rho}^{-1}\widetilde{\bm{\phi}})=\widetilde{\bm{\phi}}(\mathsf{J}\bm{\xi})=0&&\forall\bm{\xi}\in\range(\trace_k)=\dom(\A^*_k)^{\circ}.
    \end{align}
    In particular, we obtain from \eqref{eq: property of y} that $\y\in\dom(\A_k^*)$. Thus, under the choice made in \eqref{def of not t space}, $\y\in\dom(\A_k^*)\cap\Wp_{k+1}=\mathring{\W}^{n,+}_k$.
  \item Finally, the previous step implies that 
    \begin{align}
      \widetilde{\bm{\phi}}(\bm{\phi}_{\circ}) = \bm{\rho}\y(\bm{\phi}_{\circ})=\bm{\phi}_{\circ}(\y)=0 &&\forall\bm{\phi}_{\circ}\in (\mathring{\W}^{n,+}_{k+1})^{\circ}.
    \end{align}
  \end{enumerate}
  Therefore, $\widetilde{\bm{\phi}}_{\circ}=\widetilde{\bm{\phi}}\big\vert_{(\mathring{\W}^{n,+}_{k+1})^{\circ}}=0$, which concludes the proof.
\end{proof}

Now, rewriting the trace Hilbert complexes \eqref{trace Hilbert complex RTt full} and \eqref{trace Hilbert complex RTn full} in terms of the isometrically isomorphic characterizations given in \Cref{characterization of range T}, we obtain the Hilbert complexes  

\begin{subequations}
  \begin{equation}\label{Hilbert trace complex St}
    \def\arrowlength{7ex}
    \def\arrowdistance{.4}
    \begin{tikzcd}[column sep=\arrowlength]
      \cdots 
      \ar[r, rightarrow, shift left=\arrowdistance, "\hat{\mathsf{D}}^t_{k-1}"] 
      & 
      \mathbf{T}^{n,-}_{k+1}(\hat{\mathsf{D}}^t_{k}) \subset \mathbf{T}^{n,-}_{k+1}
      \arrow[r, rightarrow, shift left=\arrowdistance, "\hat{\mathsf{D}}^t_{k}"] 
      & 
      \mathbf{T}^{n,-}_{k+2}(\hat{\mathsf{D}}^t_{k+1}) \subset\mathbf{T}^{n,-}_{k+2}
      \arrow[r, rightarrow, shift left=\arrowdistance, "\hat{\mathsf{D}}^t_{k+1}"] 
      &
      \cdots
    \end{tikzcd}
  \end{equation}
  and
  \begin{equation}\label{Hilbert trace complex Sn}
    \def\arrowlength{7ex}
    \def\arrowdistance{.4}
    \begin{tikzcd}[column sep=\arrowlength]
      \cdots 
      \arrow[r, leftarrow, shift right=\arrowdistance, "\hat{\mathsf{D}}^n_k"']
      & 
      \mathbf{T}^{t,-}_{k}(\hat{\mathsf{D}}^n_{k})\subset\mathbf{T}^{t,-}_{k} 
      \ar[r, leftarrow, shift right=\arrowdistance, "\hat{\mathsf{D}}^n_{k+1}"']
      & 
      \mathbf{T}^{t,-}_{k+1}(\hat{\mathsf{D}}^n_{k+1}) \subset \mathbf{T}^{t,-}_{k+1}
      \arrow[r, leftarrow, shift right=\arrowdistance, "\hat{\mathsf{D}}^n_{k+2}"']
      & 
      \cdots.
    \end{tikzcd}
  \end{equation}
\end{subequations}

\subsection{Compactness property}\label{sec: Compactness}
It is well-known that compact embeddings of the regular spaces $\Wp_k\subset\W_k$ in the
stable decompositions \eqref{decomposition of DAT with operators} and \eqref{decomposition
  of DA with operators} lead to the Hilbert complexes \eqref{Hilbert complex Ak} and
\eqref{Hilbert dual complex ATk} being Fredholm. For convenience, we review this result in
the next lemma.
\begin{statement}
  \label{F}
  \label{assum: G}
  Suppose that the dense inclusions $\imath^+_k:\Wp_{k}\hookrightarrow\W_k$ are compact for all $k\in\mathbb{Z}$.
\end{statement}
\begin{lem}\label{lem: compactness property DA}
  Under Assumption \ref{assum: G}, Assumptions \ref{Assumption continuous and dense
    embeddings AT} and \ref{dual regular decomposition} guarantee compactness of the
  inclusions
  \begin{align}
    \dom(\AT_k)\cap\dom(\rA_{k+1})\hookrightarrow \W_{k+1} &&\text{and} && \dom(\A_{k+1})\cap\dom(\A_{k}^*)\hookrightarrow \W_{k+1},
  \end{align}
  respectively.
\end{lem}
\begin{proof}
  By symmetry, it is sufficient to prove that, under Assumption \ref{assum: G}, it follows from Assumption \ref{dual regular decomposition} that the dense inclusion $\dom(\A_{k+1})\cap\dom(\A_{k}^*)\hookrightarrow \W_{k}$ is a compact operator. In particular, let $(\y_{\ell})_{\ell\in\mathbb{Z}}\subset \dom(\A_{k+1})\cap\dom(\A_{k}^*)$ be an arbitrary sequence that is bounded in $\dom(\A_{k+1})\cap\dom(\A_{k}^*)$. We only need to show that there exists a subsequence $(\y_{\ell_\rho})_{\rho\in\mathbb{Z}}$ that is Cauchy in $\W_k$.

  By Assumption \ref{dual regular decomposition}, for all $\ell\in\mathbb{Z}$, there exist
  $\mathbf{p}^+_{\ell}\in\Wp_{k+1}$ and $\x^+_{\ell}\in\Wp_{k}$ such that
  \begin{align}\label{decomposition compactness property proof domain}
    \y_{\ell}=\mathbf{p}^+_{\ell}+\A_k\x^+_{\ell} && \Big(\text{in particular, }\mathbf{p}^+_{\ell}:=\mathsf{L}^n_{k+1}\y_{\ell} \text{ and } \x^{+}_{\ell}:=\mathsf{V}^n_{k+1}\y_{\ell}\Big).
  \end{align}
  The norm in $\dom(\A_{k+1})\cap\dom(\A_{k}^*)$ is stronger than the norm in
  $\dom(\A_{k+1})$, so since the decomposition is stable by hypothesis \emph{\ref{assum C:
      regular decomposition}} from Assumption \ref{dual regular decomposition}, the
  sequences $(\mathbf{p}^+_{\ell})_{\ell}$ and $(\x^+_{\ell})_{\ell}$ are bounded in
  $\Wp_{k+1}$ and $\Wp_k$, respectively. Under Assumption \ref{assum: G}, we can thus find
  subsequences $(\mathbf{p}^+_{\ell_{\rho}})_{\rho}$ and $(\x^+_{\ell_{\rho}})_{\rho}$
  that are Cauchy in $\W_{k+1}$ and $\W_k$, respectively. Evaluating
  \begin{align*}
    \norm{\y_{\ell_{n}}&-\y_{\ell_{m}}}^2_{\W_{k+1}}
    = \left(\mathbf{p}^+_{\ell_{n}}-\mathbf{p}^+_{\ell_{n}},\y_{\ell_{n}}-\y_{\ell_{n}}\right)_{\W_{k+1}}+\left(\A_k\left(\x^+_{\ell_{n}}-\x^+_{\ell_{n}}\right),\y_{\ell_{n}}-\y_{\ell_{n}}\right)_{\W_{k+1}}\\
    &\leq \norm{\mathbf{p}^+_{\ell_{n}}-\mathbf{p}^+_{\ell_{n}}}_{\W_{k+1}}
    \norm{\y_{\ell_{n}}-\y_{\ell_{n}}}_{\W_{k+1}}+\left(\x^+_{\ell_{n}}-
      \x^+_{\ell_{n}},\A_k^*\left(\y_{\ell_{n}}-\y_{\ell_{n}}\right)\right)_{\W_{k+1}}\\
    &\leq
    \underbrace{\norm{\mathbf{p}^+_{\ell_{n}}-\mathbf{p}^+_{\ell_{n}}}_{\W_{k+1}}}_{\rightarrow
      0\text{ as }n,m\rightarrow 0}\norm{\y_{\ell_{n}}-\y_{\ell_{n}}}_{\W_{k+1}}+
    \underbrace{\norm{\x^+_{\ell_{n}}-\x^+_{\ell_{n}}}_{\W_{k}}}_{\rightarrow 0\text{ as
      }n,m\rightarrow 0}
    \norm{\A_k^*\left(\y_{\ell_{n}}-\y_{\ell_{n}}\right)}_{\W_{k+1}},
  \end{align*}
  we arrive at the conclusion once noticing that
  $\norm{\y_{\ell_{n}}-\y_{\ell_{n}}}_{\W_{k+1}}$ and
  $\norm{\A_k^*\left(\y_{\ell_{n}}-\y_{\ell_{n}}\right)}_{\W_{k+1}}$ are also bounded by
  hypothesis.
\end{proof}

In other words, under Assumption \ref{assum: G}, the stable decompositions of \Cref{sec:
  Bounded regular decompositions} imply complex properties, which as stated in \Cref{sec:
  Hilbert complexes}, guarantee that the associated Hilbert complexes are Fredholm. The
goal of this section is to show that this carries over to the trace spaces. Ultimately,
this is because what is essential for \Cref{lem: compactness property DA} to go through is
not compactness of the spaces, but rather that the potential and lifting operators are
compact operators.

In order to obtain the complex properties for the trace Hilbert complexes, we find it most
convenient to work with the characterizations provided in \Cref{thm: trace Hilbert
  complexes definition}, because it allows us to harness the theory developed in
\Cref{sec: Riesz representatives}. By symmetry, we may focus on \eqref{trace Hilbert
  complex RTt full}.

For any $\x\in\dom(\A_{k})$, it follows from Assumption \ref{dual regular decomposition}
and the commuting relations of \Cref{lem: commutative relations} that
\begin{equation}\label{eq: decomposition under trace}
  \trace_{k}\x = \trace_{k}\mathsf{L}^n_{k}\x + \trace_{k}\A_{k-1}\mathsf{V}^n_k\x = \trace_{k}\mathsf{L}^n_{k}\x - \mathsf{D}^t_{k-1}\trace_{k-1}\mathsf{V}^n_{k}\x. 
\end{equation}
Recall from \Cref{lem: riesz rep in domains} that the $\dom(\A_{k})$-harmonic extension operators $-\AT_k\mathsf{R}_{\scaleto{\dom(\AT_k)}{7pt}}^{-1}:\range(\traceA)\rightarrow\dom(\A_k)$ satisfy $\trace_{\A_k}(-\AT_k\mathsf{R}_{\scaleto{\dom(\AT_k)}{7pt}}^{-1}\bm{\phi})=\bm{\phi}$ for all $\bm{\phi}\in\range(\trace_k)$. Inserting this identity in
\eqref{eq: decomposition under trace} yields the decomposition
\begin{equation}\label{eq: decomposition trace space}
  \bm{\phi} = -\trace_{k}\mathsf{L}^n_{k}\AT_k\mathsf{R}_{\scaleto{\dom(\AT_k)}{7pt}}^{-1}\bm{\phi} + \mathsf{D}^t_{k-1}\trace_{k-1}\mathsf{V}^n_{k}\AT_k\mathsf{R}_{\scaleto{\dom(\AT_k)}{7pt}}^{-1}\bm{\phi}
\end{equation}
for all $\bm{\phi}\in\range(\trace_k)$. 

Compare \eqref{eq: decomposition trace space} with the regular decompositions provided in
\eqref{decomposition of DAT with operators} and \eqref{decomposition of DA with
  operators}. In \eqref{eq: decomposition trace space}, the bounded maps
\begin{align}
  -\trace_{k}\mathsf{L}^n_{k}\AT_k\mathsf{R}_{\scaleto{\dom(\AT_k)}{7pt}}^{-1}:\range(\trace_k)\rightarrow \trace_k(\Wp_{k})\subset\range(\trace_k)
\end{align} 
and 
\begin{align}
  \trace_{k-1}\mathsf{V}^n_{k}\AT_k\mathsf{R}_{\scaleto{\dom(\AT_k)}{7pt}}^{-1}: \range(\trace_k)\rightarrow \trace_{k-1}(\Wp_{k-1})\subset\range(\trace_{k-1})   
\end{align} play the roles of lifting and potential operators. Compactness of these operators as mappings $\range(\trace_k)\rightarrow (\mathring{\W}^{n,+}_{k+1})^{\circ}$ and $\range(\trace_k)\rightarrow (\mathring{\W}^{n,+}_{k})^{\circ}$ follows upon observing that under Assumption \ref{assum: G}, the map 
\begin{equation}\label{compact trace}
  \trace_k:\Wp_{k}\rightarrow(\mathring{\W}^{n,+}_{k})^{\circ}
\end{equation}
is a compact operator, because the product of two bounded linear operators between normed spaces is compact if any one of the operand is \cite[Thm. 2.16]{MR1723850}. To confirm that \eqref{compact trace} is compact, it is sufficient to recall from \Cref{def: trace enviro} that it is the operator associated with the compact bilinear form (cf. \cite[Chap. 3]{MR2361676})
\begin{equation}
  \left\{
    \arraycolsep=1.4pt\def\arraystretch{1.2}
    \begin{array}{rcl}
      \Wp_{k}\times\Wp_{k+1} &\rightarrow &\mathbb{R}\\
      (\x,\y)&\mapsto &(\A_k\x,\imath^{+}_{k+1}\y)_{\W_{k+1}} - (\imath^{+}_k\x,\AT_k\y)_{\W_{k}}
    \end{array}
  \right.
\end{equation}
where we have introduced for clarity the compact inclusions supplied by Assumption \ref{assum: G}.

In the next theorem, the unbounded linear operators
\begin{subequations}
  \begin{align}
    (\mathsf{D}^t_{k})^*&: \dom\left((\mathsf{D}^t_k)^*\right)\subset(\mathring{\W}^{n,+}_{k+2})^{\circ}\rightarrow (\mathring{\W}^{n,+}_{k+1})^{\circ},\\
    (\mathsf{D}^n_{k})^*&: \dom\left((\mathsf{D}^n_{k})^*\right)\subset(\mathring{\W}^{t,+}_{k-1})^{\circ}\rightarrow (\mathring{\W}^{t,+}_{k})^{\circ},
  \end{align}
\end{subequations}
are the Hilbert space adjoints of the closed densely defined unbounded operators 
\begin{align}
  \mathsf{D}^t_{k}:\range(\trace_{k})\subset(\W^{n,+}_{k+1})^{\circ}\rightarrow (\W^{n,+}_{k+2})^{\circ} &&\text{and} && \mathsf{D}^n_{k}:\range(\dualtrace_{k})\subset(\mathring{\W}^{t,+}_{k})^{\circ}\rightarrow (\mathring{\W}^{t,+}_{k-1})^{\circ},
\end{align} 
respectively.
\medskip

\thmbox{%
\begin{theo}
  Under assumptions \ref{density if Wnnot spaces}, \ref{density if Wtnot spaces} and \ref{assum: G}, the inclusions
  \begin{align}\label{trace space compactness property}
    \range(\trace_k)\cap\dom\left((\mathsf{D}^t_{k-1})^*\right)\hookrightarrow (\W^{t,+}_{k+1})^{\circ}
    &&\text{and}
    &&\range(\dualtrace_k)\cap\dom\left((\mathsf{D}^n_{k+1})^*\right)\hookrightarrow (\W^{n,+}_{k})^{\circ}
  \end{align}
  are compact.
\end{theo}}

\begin{proof}
  We follow the arguments in the proof of \Cref{lem: compactness property DA}. Let $(\bm{\phi}_{\ell})_{\ell\in\mathbb{Z}}\subset\range(\trace_k)\cap\dom\left((\mathsf{D}^t_k)^*\right)$ be a bounded sequence in $\range(\trace_k)\cap\dom\left((\mathsf{D}^t_k)^*\right)$. 

  The goal is to find a subsequence $(\bm{\phi}_{\ell_{\rho}})_{\rho\in\mathbb{Z}}$ that is Cauchy in $(\W^{t,+}_{k+1})^{\circ}$. 
  Similarly to \eqref{decomposition compactness property proof domain}, we use the stable decomposition in trace spaces \eqref{eq: decomposition trace space}:
  \begin{equation}
    \bm{\phi}_{\ell} = \bm{\xi}_{\ell}^+ + \mathsf{D}^t_{k-1}\bm{\zeta}_{\ell}^+
  \end{equation}
  for all $\ell\in\mathbb{Z}$, where $\bm{\xi}_{\ell}^+:=-\trace_{k}\mathsf{L}^n_{k}\AT_k\mathsf{R}_{\scaleto{\dom(\AT_k)}{7pt}}^{-1}\bm{\phi}_{\ell}$ and $\bm{\zeta}_{\ell}:=\trace_{k-1}\mathsf{V}^n_{k}\AT_k\mathsf{R}_{\scaleto{\dom(\AT_k)}{7pt}}^{-1}\bm{\phi}_{\ell}$. Since the norm in $\range(\trace_k)\cap\dom\left((\mathsf{D}^t_k)^*\right)$ is stronger than the norm in $\range(\mathsf{T}^t_k)$, the sequence $(\bm{\phi}_{\ell})_{\ell\in\mathbb{Z}}$ is bounded in the norm of $\range(\trace_k)$. Hence, by compactness of the operators $-\trace_{k}\mathsf{L}^n_{k}\AT_k\mathsf{R}_{\scaleto{\dom(\AT_k)}{7pt}}^{-1}:\range(\trace_k)\rightarrow (\W^{t,+}_{k+1})^{\circ}$ and $\trace_{k-1}\mathsf{V}^n_{k}\AT_k\mathsf{R}_{\scaleto{\dom(\AT_k)}{7pt}}^{-1}: \range(\trace_k)\rightarrow (\W^{t,+}_{k})^{\circ}$, there exist subsequences $(\bm{\xi}^+_{\ell_{\rho}})_{\rho\in\mathbb{Z}}$ and $(\bm{\zeta}_{\ell_{\rho}}^+)_{\rho\in\mathbb{Z}}$ that are Cauchy in $(\W^{t,+}_{k+1})^{\circ}$ and $(\W^{t,+}_{k})^{\circ}$, respectively.

  Now, we verify that $(\bm{\phi}_{\ell_{\rho}})_{\rho\in\mathbb{Z}}$ is indeed Cauchy in $(\W^{t,+}_{k+1})^{\circ}$. We evaluate directly
  \begin{align*}
    \norm{&\bm{\phi}_{\ell_{n}}-\bm{\phi}_{\ell_{n}}}_{(\W^{t,+}_{k+1})^{\circ}}\\
    &= \,\,\left(\bm{\xi}_{\ell_{n}}-\bm{\xi}_{\ell_{n}},\bm{\phi}_{\ell_{n}}-\bm{\phi}_{\ell_{n}}\right)_{(\W^{t,+}_{k+1})^{\circ}} + \left(\mathsf{D}^t_{k-1}\left(\bm{\zeta}_{\ell_{n}}-\bm{\zeta}_{\ell_{n}}\right),\bm{\phi}_{\ell_{n}}-\bm{\phi}_{\ell_{n}}\right)_{(\W^{t,+}_{k+1})^{\circ}} \\
    &\,\,\leq\norm{\bm{\xi}_{\ell_{n}}-\bm{\xi}_{\ell_{n}}}_{(\W^{t,+}_{k+1})^{\circ}}\norm{\bm{\phi}_{\ell_{n}}-\bm{\phi}_{\ell_{n}}}_{(\W^{t,+}_{k+1})^{\circ}}+\left(\bm{\zeta}_{\ell_{n}}-\bm{\zeta}_{\ell_{n}},(\mathsf{D}^t_{k-1})^*\left(\bm{\phi}_{\ell_{n}}-\bm{\phi}_{\ell_{n}}\right)\right)_{(\W^{t,+}_{k})^{\circ}},
  \end{align*}
  from which we conclude that
  \begin{align*}
    \norm{\bm{\phi}_{\ell_{n}}-\bm{\phi}_{\ell_{n}}}_{(\W^{t,+}_{k+1})^{\circ}}
    \leq&\underbrace{\norm{\bm{\xi}_{\ell_{n}}-\bm{\xi}_{\ell_{n}}}_{(\W^{t,+}_{k+1})^{\circ}}}_{\rightarrow 0\text{ as }m,n\rightarrow 0}\norm{\bm{\phi}_{\ell_{n}}-\bm{\phi}_{\ell_{n}}}_{(\W^{t,+}_{k+1})^{\circ}}\\
    &\quad+\underbrace{\norm{\bm{\xi}_{\ell_{n}}-\bm{\xi}_{\ell_{n}}}_{(\W^{t,+}_{k})^{\circ}}}_{\rightarrow 0\text{ as }m,n\rightarrow 0}\norm{(\mathsf{D}^t_{k-1})^*\left(\bm{\phi}_{\ell_{n}}-\bm{\phi}_{\ell_{n}}\right)}_{(\W^{t,+}_{k})^{\circ}}.
  \end{align*}
  The desired result thus follows because $\norm{\bm{\phi}_{\ell_{n}}-\bm{\phi}_{\ell_{n}}}_{(\W^{t,+}_{k+1})^{\circ}}$ and $\norm{(\mathsf{D}^t_{k-1})^*\left(\bm{\phi}_{\ell_{n}}-\bm{\phi}_{\ell_{n}}\right)}_{(\W^{t,+}_{k})^{\circ}}$ are bounded by hypothesis.
\end{proof}

\begin{cor}
  Under assumptions \ref{density if Wnnot spaces}, \ref{density if Wtnot spaces} and
  \ref{assum: G}, the trace Hilbert complexes introduced in \Cref{thm: trace Hilbert
    complexes definition} are Fredholm.
\end{cor}

It is particularly interesting that while only one decomposition was sufficient to obtain
\Cref{lem: compactness property DA}, we needed \emph{both} decompositions (assumptions
\ref{Assumption continuous and dense embeddings AT} and \ref{dual regular decomposition})
to achieve a proof of the compactness property for the trace Hilbert complex: one for the
space characterization and the other for the decomposition formula itself. The question
whether it is \emph{necessary} to have both remains open.

\begin{example}{Trace de Rham complexes}\label{eg trace de Rham complexes}
  Trace Hilbert complexes for the de Rham complex in 3D arise from the results of \ref{eg: Characterization of trace spaces using quotient spaces}:
  \begin{equation}
    \def\arrowlength{4ex}
    \def\arrowdistance{.4}
    \begin{tikzcd}[column sep=\arrowlength]
      \{0\}	\arrow[d, hookrightarrow, shift left=\arrowdistance, "\imath"] &\{0\}	\arrow[d, hookrightarrow, shift left=\arrowdistance, "\imath"]\\
      \dom(\mathbf{curl}')\subset \widetilde{\H}^{-1}(\Omega)\cap\Hdivnot^{\circ}
      \arrow[d, rightarrow, shift left=\arrowdistance, "\mathbf{curl}'"] &\dom(\mathbf{curl}')\subset \Big(\mathbf{H}^1(\Omega)/\Hn\Big)'\arrow[d, rightarrow, shift left=\arrowdistance, "\mathbf{curl}'"]\\
      \dom(\mathbf{grad}')\subset \widetilde{\H}^{-1}(\Omega)\cap\Hcurlnot^{\circ}
      \arrow[d, rightarrow, shift left=\arrowdistance, "\mathbf{grad}'"]&\dom(\mathbf{grad})\subset\Big(\mathbf{H}^1(\Omega)/\Ht\Big)'\arrow[d, rightarrow, shift left=\arrowdistance, "\mathbf{grad}'"]\\
      \widetilde{H}^{-1}(\Omega)\cap\Honenot^{\circ}\arrow[d, twoheadrightarrow, shift left=\arrowdistance, "0"]&\Big(\Hone/\Honenot\Big)'\arrow[d, twoheadrightarrow, shift left=\arrowdistance, "0"]\\
      \{0\}&\{0\}
    \end{tikzcd}
  \end{equation}
  In light of the de Rham setting \ref{eg: Characterization of trace spaces using quotient
    spaces}, they correspond to 
  \begin{equation}
    \def\arrowlength{3ex}
    \def\arrowdistance{.5}
    \begin{tikzcd}[column sep=\arrowlength]
      \{0\}
      \arrow[r, hookrightarrow, shift left=\arrowdistance, "\imath"] 
      & 
      H^{1/2}(\Gamma)\subset H^{-1/2}(\Gamma)
      \arrow[r, rightarrow, shift left=\arrowdistance, "\mathbf{curl}_{\Gamma}"] 
      & 
      \mathbf{H}^{-1/2}(\text{curl}_{\Gamma},\Gamma)\subset \mathbf{H}^{-1/2}_t
      \ar[r, rightarrow, shift left=\arrowdistance, "\text{curl}_{\Gamma}"] 
      & 
      H^{-1/2}(\Gamma)
      \arrow[r, twoheadrightarrow, shift left=\arrowdistance, "0"] 
      &\{0\}
    \end{tikzcd}
  \end{equation}
  or its rotated version.

  Since by Rellich's lemma the embeddings $\Hone\hookrightarrow\Ltwo$ and
  $\mathbf{H}^1(\Omega)\hookrightarrow\mathbf{L}^2(\Omega)$ are compact, the de Rham
  complexes in \eqref{structure Hilbert complexes 3D de Rham} satisfy Assumption
  \ref{assum: G} with the regular decompositions presented in the de Rham setting \ref{eg: stable regular
    decompositions}. Therefore, the associated trace de Rham complexes are Fredholm. As a
  consequence, their cohomology spaces are finite-dimensional.
\end{example}

\section{Conclusion}
\label{sec:concl}

As we have demonstrated in the present article, it takes only a pair of Hilbert complexes
linked by the sub-complex relationship of their domain complexes to recover essential
aspects of the structures inherent in the trace operators and trace spaces for the de Rham
complex. Relying on notions of trace spaces as dual spaces or quotient spaces, we could
establish detailed characterizations merely assuming the existence of stable regular
decompositions induced by bounded lifting operators. These developments culminated in the
discovery of associated trace Hilbert complexes, which are Fredholm under the mild
additional assumption that the lifting operators are compact.

Hilbert complexes have recently moved into the focus of applied mathematicians, since they
underlie a host of PDE-based mathematical models in areas as diverse as linear elasticity,
gravity, and fluid dynamics. The related complexes are known as the elasticity complex,
\cite[Sect.~11]{AFW06} and \cite{PAZ20b}, conformal complex, or Stokes complex
\cite[Sect.~4.4]{ARH20}. These and many more complexes \cite{PAZ20a,PAZ17t} arise from the
de Rham complex through the powerful Bernstein-Gelfand-Gelfand (BGG) construction, as has
been shown in \cite{ARH20}. Most likely, many more Hilbert complexes relevant for
mathematical modeling still await discovery.

This backdrop lends relevance to our present work. Once the Hilbert complex structure is
established, trace operators and trace spaces become available, which can serve as
stepping stones towards the study of boundary value problems and the development of
integral representations.


\bibliographystyle{siam}
\bibliography{hps}
\end{document}